\documentclass[11pt]{article}
\usepackage{amsmath,latexsym,amsbsy,amssymb,enumerate,amsthm}
\usepackage{enumitem}
\usepackage[dvipsnames]{xcolor}
\usepackage{hyperref}
\usepackage[capitalise]{cleveref}
\usepackage{graphicx}
\usepackage{cases}
\usepackage{relsize}
\usepackage{subfigure}
\usepackage{float}
\usepackage{comment}
\usepackage[numbers]{natbib}
\usepackage{tikz}
\usepackage{forest}
\usepackage{tikz-qtree}
\usepackage{amsbsy}
\usepackage[enableskew,stdtext]{youngtab}
\usepackage{ytableau}
\usepackage{curves}
\usetikzlibrary{tqft}
\usetikzlibrary{positioning}
\usetikzlibrary{arrows.meta, decorations.pathmorphing,calc,cd}
\def\epsilon{\varepsilon}

\newcommand{\set}[1]{\left\{#1\right\}}
\newcommand{\sett}[2]{\left\{#1 ~\Big| ~ #2\right\}}

\newcommand{\M}[0]{\mathcal{M}}
\newcommand{\E}[0]{\mathcal{E}}

\newcommand{\lewk}[2]{#1 \leq_{wk} #2}

\def\zh{\hat{0}}
\def\oneh{\hat{1}}

\newcommand{\inv}[2]{\text{inv}_{#2}\left(#1\right)}
\newcommand{\precdot}{\prec\mathrel{\mkern-5mu}\mathrel{\cdot}}

\newcommand{\lex}[2]{#1<_{lex}#2}
\newcommand{\cleq}[2]{#1 \preceq #2}
\newcommand{\cord}[2]{{#1}_{#2}(2)}
\newcommand{\cordot}[2]{#1 \precdot #2}
\newcommand{\lm}[0]{\lambda}
\newcommand{\cl}[3]{{#1 \prec_{#3} #2}}
\newcommand{\EL}[0]{EL-labeling}
\newcommand{\CL}[0]{CL-labeling}
\newcommand{\cordoti}[3]{#1 \precdot_{#3} #2}
\newcommand{\cleqi}[3]{#1 \preceq_{#3} #2}
\newcommand{\Y}[0]{\mathcal{Y}}
\newcommand{\Lin}[0]{\mathcal{L}}

\newtheorem{theorem}{Theorem}[section]
\newtheorem{lemma}[theorem]{Lemma}
\newtheorem{corollary}[theorem]{Corollary}
\newtheorem{proposition}[theorem]{Proposition}
\newtheorem{definition}[theorem]{Definition}
\theoremstyle{remark}
\newtheorem{remark}[theorem]{Remark}

\theoremstyle{definition}
\newtheorem{example}[theorem]{Example}

\title{Maximal chain descent orders}
\author{Stephen Lacina\footnote{The author was supported by NSF grants DMS-1953931 and DMS-2039316.}}
\date{}

\begin{document}

\maketitle

\begin{abstract}
This paper introduces a partial order on the maximal chains of any finite bounded poset $P$ which has a {\CL} $\lm$. We call this the maximal chain descent order induced by $\lm$, denoted $\cord{P}{\lm}$. As a first example, letting $P$ be the Boolean lattice and $\lm$ its standard {\EL} gives $\cord{P}{\lm}$ isomorphic to the weak order of type A. We discuss in depth other seemingly well-structured examples: the max-min {\EL} of the partition lattice gives maximal chain descent order isomorphic to a partial order on certain labeled trees, and particular cases of the linear extension {\EL}s of finite distributive lattices produce maximal chain descent orders isomorphic to partial orders on standard Young tableaux.

We observe that the order relations which one might expect to be the cover relations, those given by the ``polygon moves" whose transitive closure defines the maximal chain descent order, are not always cover relations. Several examples illustrate this fact. Nonetheless, we characterize the {\EL}s for which every polygon move gives a cover relation, and we prove many well known {\EL}s do have the expected cover relations. 

One motivation for $\cord{P}{\lm}$ is that its linear extensions give all of the shellings of the order complex of $P$ whose restriction maps are defined by the descents with respect to $\lm$. This yields strictly more shellings of $P$ than the lexicographic ones induced by $\lm$. Thus, the maximal chain descent order $\cord{P}{\lm}$ might be thought of as encoding the structure of the set of shellings induced by $\lm$.
\end{abstract}

\section{Introduction}

Since its inception, lexicographic shellability of partially ordered sets (posets) has become an important tool for understanding the topological, combinatorial, and algebraic structure of posets, a central topic within the larger field of topological combinatorics. For instance, such shellings are used to compute the homotopy types and the $h$-vectors of poset order complexes and to determine Cohen-Macaulayness of the associated Stanley-Reisner rings. In this paper, we introduce a partial order on the maximal chains of any finite bounded poset $P$ endowed with a {\CL} $\lm$ which encodes the structure of the set of shellings induced by $\lm$, including the lexicographic shellings. The goal is to further understand the structure of this set of shellings. We call this partial order the maximal chain descent order induced by $\lm$, denoted by $\cord{P}{\lm}$. 

A good first example is the Boolean lattice with its standard {\EL}. In this case, the maximal chain descent order is isomorphic to the weak order of type A (\cref{thm:boolean lat cord weak order symm grp}). Analyzing one instance of a Boolean lattice lets us exhibit the construction of a maximal chain descent order and provides a running example with which to interpret our main results. 

Consider the poset of all subsets of $[3]=\set{1,2,3}$ ordered by subset inclusion. This is the Boolean lattice $B_3$ which is shown in \cref{fig:B3 weak order example} (a). We label a cover relation $B\lessdot B\cup \set{i}$ in the Hasse diagram of $B_3$ by $\lm(B,B\cup \set{i})=i$. The labeling $\lm$ is a well known {\EL}. We form a partial order on the maximal chains of $B_3$ by taking the transitive closure of moves of the form: $(\emptyset \lessdot \set{1} \lessdot \set{1,2} \lessdot \set{1,2,3}) \to (\emptyset \lessdot \set{1} \lessdot \set{1,3} \lessdot \set{1,2,3})$ where $(\emptyset \lessdot \set{1} \lessdot \set{1,3} \lessdot \set{1,2,3})$ contains exactly one element which is not in $(\emptyset \lessdot \set{1} \lessdot \set{1,2} \lessdot \set{1,2,3})$ and $(\emptyset \lessdot \set{1} \lessdot \set{1,2} \lessdot \set{1,2,3})$ is lexicographically first (thus, ascending) with respect to the labeling by $\lm$ in the interval $[\set{1},\set{1,2,3}]$ on which the chains differ. We call such moves on maximal chains polygon moves (see \cref{def:increase and polygon} for a precise definition). The resulting maximal chain descent order $\cord{{B_3}}{\lm}$ is shown in \cref{fig:B3 weak order example} (b) with the maximal chains of $B_3$ represented by their label sequences since the label sequences are all distinct. In this case, the partial order is precisely the weak order on the symmetric group $S_3$ as we expected.

\begin{figure}[H]
    \centering
    \subfigure[$B_3$ with {\EL} $\lm$.]{
    \scalebox{0.8}{
    \begin{tikzpicture}[very thick]
    \node (zh) at (0,0) {$\emptyset$};
    \node (1) at (-2,2) {$\set{1}$};
    \node (2) at (0,2) {$\set{2}$};
    \node (3) at (2,2) {$\set{3}$};
    \node (12) at (-2,4) {$\set{1,2}$};
    \node (13) at (0,4) {$\set{1,3}$};
    \node (23) at (2,4) {$\set{2,3}$};
    \node (123) at (0,6) {$\set{1,2,3}$};

    \node[blue] (1 13) at (-1.2,2.5) {$3$};
    \node[blue] (2 12) at (-0.3,2.7) {$1$};
    \node[blue] (2 23) at (0.3,2.7) {$3$};
    \node[blue] (3 13) at (1.2,2.5) {$1$};
    
    \draw (zh) --node[below left,blue]{$1$} (1) --node[left,blue]{$2$} (12) --node[above left,blue]{$3$} (123) --node[above right,blue]{$1$} (23) --node[right,blue]{$2$} (3) --node[below right,blue]{$3$} (zh);
    \draw (zh) --node[left,blue]{$2$} (2) -- (12);
    \draw (1) -- (13) -- (3);
    \draw (2) -- (23);
    \draw (13) -- node[left,blue]{$2$} (123);
    \end{tikzpicture}
    }
    }
    \hspace{10mm}
    \subfigure[Maximal chain descent order $\cord{{B_3}}{\lm}$.]{
    \scalebox{1}{
    \begin{tikzpicture}[very thick]
    \node[blue] (123) at (0,0) {$123$};
    \node[blue] (132) at (-1,1) {$132$};
    \node[blue] (213) at (1,1) {$213$};
    \node[blue] (312) at (-1,2) {$312$};
    \node[blue] (231) at (1,2) {$231$};
    \node[blue] (321) at (0,3) {$321$};
    
    \draw (123) -- (132) -- (312) -- (321);
    \draw (123) -- (213) -- (231) -- (321);
    \end{tikzpicture}
    }
    }
    \caption{Boolean lattice with an {\EL} and its maximal chain descent order.}
    \label{fig:B3 weak order example}
\end{figure}
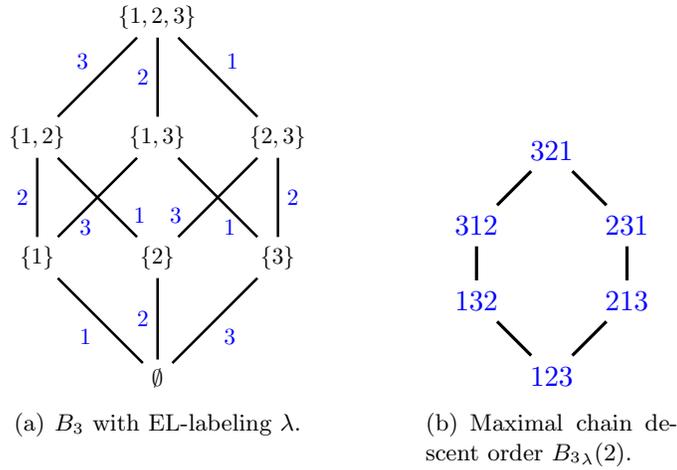

Our polygon moves on maximal chains are somewhat reminiscent of the ``polygon flips" between monotone paths in polytopes (oriented by a linear functional) which were employed by Athanasiadis, Edelman, and Reiner in \cite{monotonepathsathanedelmanreiner2000}. Similar moves on the maximal chains of a finite poset with an $S_n$ {\EL} were considered by McNamara in \cite{snelsupersolvmcnamara2003} to define a $0$-Hecke algebra action on the maximal chains of such posets. Related ``diamond moves" on thin posets were considered by Chandler in \cite{thinposetsandcatschandler2019}, but without the partial order structure on maximal chains from a {\CL} which we introduce here.  

Our first main theorem is that $\cord{P}{\lm}$ precisely encodes all of the shellings which can be ``derived from $\lm$." Namely, we prove the following:
\begin{theorem}\label{thm:introduction shelling order with descent restriction iff linext of cord}
Let $P$ be a finite, bounded poset with a {\CL} $\lm$. For any total order $\Omega:m_1,m_2,\dots, m_t$ on the maximal chains of $P$, the following are equivalent:
\begin{itemize}
    \item [(1)] $\Omega$ is a linear extension of the maximal chain descent order $\cord{P}{\lm}$.
    \item [(2)] $\Omega$ induces a shelling order of the order complex $\Delta(P)$ with the property that for each $1\leq i\leq t$ the restriction face $R(m_i)$ of $m_i$ is precisely the face \\ $R(m_i)=\sett{x\in m_i}{\text{x is a descent of $m_i$ w.r.t. $\lm$}}$.
\end{itemize}
\end{theorem}
We prove this as \cref{thm:shelling order with descent restriction iff linext of cord}. It is perhaps worth noting that the analogous statement for {\EL}s holds as well, because {\EL}s are instances of {\CL}s. The lexicographic shellings from a {\CL} are given by ordering the maximal chains of $P$ according to lexicographic order on their label sequences and breaking ties arbitrarily. We show that the linear extensions of $\cord{P}{\lm}$ give all of the lexicographic shellings as well as additional shellings, though our proof that the linear extensions of $\cord{P}{\lm}$ are shellings does follow a similar line of reasoning to the proof of Bj{\"o}rner and Wachs for the lexicographic ones (see their proof in \cite{shllblnonpureIbjornerwachs1996}). We offer examples illustrating these additional shellings. This can already be observed in the example of the Boolean lattice with its standard {\EL} whose maximal chain descent order is isomorphic to the weak order of type A. In this case, \cref{thm:introduction shelling order with descent restriction iff linext of cord} allows us to recover Bj{\"o}rner's result (in type A) from \cite{coxetercomplexesbjorner1984} that any linear extension of the weak order on a Coxeter group gives rise to a shelling order of its Coxeter complex. 

Besides encoding the structure of the set of shellings induced by {\CL} $\lm$, maximal chain descent orders possess seemingly interesting structure as posets in their own right. For instance, every cover relation in $\cord{P}{\lm}$ comes from a polygon move, as one can easily observe, a fact which might lead one to assume that all such polygon moves give cover relations. This is true for some examples such as the example above. However, this is not generally the case as exhibited by the poset in \cref{fig:noncover diamond move example} with an {\EL} and its induced maximal chain descent order. Specifically, there is a polygon move between the maximal chain labeled $123$ and the maximal chain labeled $543$ despite the fact that $543$ does not cover $123$ in the maximal chain descent order.

\begin{figure}[H]
    \centering
    \subfigure[$P$ with {\EL} $\lm$]{
    \scalebox{1}{
    \begin{tikzpicture}[very thick]
    \node (zh) at (0,0) {$\zh$};
    \node (a1) at (-1,1) {$z_1$};
    \node (a2) at (1,1) {$z_2$};
    \node (b1) at (-1,2) {$x_1$};
    \node (b2) at (1,2) {$x_2$};
    \node (oneh) at (0,3) {$\oneh$};

    \node (spacer) at (1.5,0) {};

    \node (5) at (-0.3,1.1) {\textcolor{blue}{$5$}};
    \node (4) at (0.3,1.1) {\textcolor{blue}{$4$}};
    
    \draw (zh) --node [blue, below left]{$1$} (a1) --node [blue, left]{$2$} (b1) --node [blue, above left]{$3$} (oneh);
    \draw (zh) --node [blue, below right]{$5$} (a2) --node [blue, right]{$3$} (b2) --node [blue, above right]{$4$} (oneh);
    \draw (a1) -- (b2);
    \draw (a2) -- (b1);
    \end{tikzpicture}
    }
    }
    \subfigure[$\cord{P}{\lm}$]{
    \scalebox{1}{
    \begin{tikzpicture}[very thick]
    \node[blue] (123) at (0,0) {$123$};
    \node[blue] (154) at (0,1) {$154$};
    \node[blue] (534) at (0,2) {$534$};
    \node[blue] (543) at (0,3) {$543$};
    
    \draw (123) -- (154) -- (534) -- (543);

    \coordinate (O) at (0.3,0,0);
    \coordinate (A) at (0.3,3,0);

    \draw[dashed] (O) to [bend right=35] (A);

    \node (dm) at (3, 2) {\text{polygon move}};
    \node (dm) at (3, 1.5) {\text{that does not}};
    \node (dm) at (3, 1) {\text{give a cover relation}}; 
    \node (spacer) at (-2.5,0) {};
    \end{tikzpicture}
    }
    }
    
    \caption{{\EL} which is not polygon complete.}
    \label{fig:noncover diamond move example}
\end{figure}
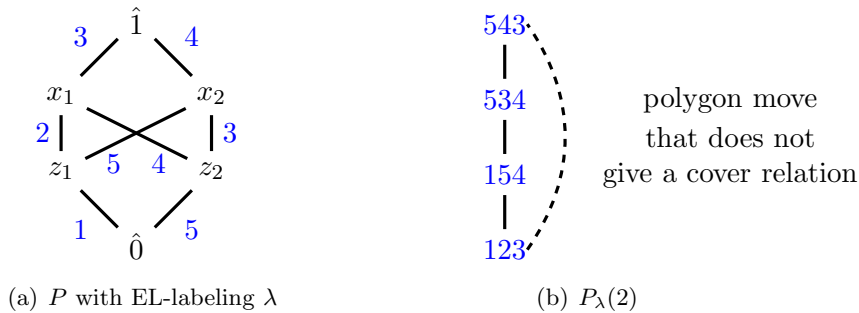

Informally, one might view this phenomenon as it sometimes being possible to ``go around the back" to prevent a polygon move from giving a cover relation. We call a {\CL} polygon complete if every polygon move gives a cover relation (see \cref{def:polygon complete}). Our second main result is a technical characterization of the polygon complete {\EL}s in \cref{thm:characterization of non-polygon completeness}.

In addition, we give two quite checkable sufficient conditions for polygon completeness. These are much simpler than the full characterization for {\EL}s. The first of these conditions is a condition on {\EL}s which we call polygon strong. It is a weakening of Bj{\"o}rner's notion of strongly lexicographically shellable from \cite{shellblposets}. We show that polygon strong implies polygon complete in \cref{thm:polygonstrongELincreasesgivecovers}, and several well known families of {\EL}s are shown to be polygon strong in \cref{sec:examples of polygon strong ELs}. For the second condition, we introduce a notion of inversions for {\CL}s in \cref{sec:inversions in ELs}. Then we formulate a condition on such inversions which we show to be sufficient for polygon completeness in \cref{thm:inversion condition implies ranked cord}. Moreover, \cref{thm:inversion condition implies ranked cord} shows that this condition on inversions with respect to a {\CL} $\lm$ also implies that $\cord{P}{\lm}$ is ranked and implies that the homology facets of the induced shellings of the proper part of $P$ are determined by rank in the maximal chain descent order. The same notion of inversions, but in the setting of $S_n$ {\EL}s in the sense of McNamara \cite{snelsupersolvmcnamara2003}, was considered in \cite{snelsupersolvmcnamara2003} to make induction arguments about supersolvability and $0$-Hecke algebra actions on maximal chains of lattices with an $S_n$ {\EL}. We also prove in \cref{lem:easyincrnocovercondition} that a simple condition on labelings of certain induced subposets guarantees a {\CL} is not polygon complete. 

We also develop the following more technical structural properties of maximal chain descent orders. In \cref{prop:non ranked cords have zero hat}, we prove every maximal chain descent order has a unique minimal element given by the unique ascending maximal chain of $P$ with respect to $\lm$. We prove in \cref{lem:number downward cord covers} that the number of downward cover relations in $\cord{P}{\lm}$ from a maximal chain $m$ of $P$ is bounded above by the number of descents of $m$ with respect to $\lm$. In \cref{cor:homology facets from downward covers}, we show that certain homology facets of the induced shellings of the proper part of $P$ can be detected from the poset structure of $\cord{P}{\lm}$ and the lengths of maximal chains in $P$. We prove in \cref{lem:restrictedrelslift} that maximal chain descent orders additionally satisfy a certain lifting property which reflects the recursive nature of {\CL}s. 

Along with the weak order of type A, there are many more examples of maximal chain descent orders. We analyse the structure of a few of these in \cref{sec:first examples}. In \cref{thm:Sn cords weak lower intervals}, we prove that all intervals in any maximal chain descent order induced by Stanley's $M$-chain {\EL}s of any finite supersolvable lattice (see \cite{stanleysupersolvablelats1972}) are isomorphic to intervals in the weak order of type A, doing so via the map assigning to each maximal chain its label sequence. We use this in \cref{thm:dist lat Sn cord isom weak ideal} to show that the linear extension {\EL}s of any finite distributive lattice produce maximal chain descent orders isomorphic to certain order ideals of the type A weak order. When the finite distributive lattice is any interval in Young's lattice, \cref{thm:youngs lat cords tableau swaps} shows that the maximal chain descent orders are isomorphic to partial orders on standard Young tableaux or standard skew tableaux. A result of Bj{\"o}rner and Wachs in \cite{genquotcxtrgrpsbjornerwachs1988} then implies that particular maximal chain descent orders are isomorphic to the weak order on generalized quotients of type A introduced by Bj{\"o}rner and Wachs in \cite{genquotcxtrgrpsbjornerwachs1988}. Lastly, we prove in \cref{thm:partition cord and tree poset isom} that the ``min-max" {\EL} of the partition lattice $\Pi_{n}$ yields a maximal chain descent order isomorphic to a natural partial order on certain labeled trees.

The structure of this paper is as follows: \cref{sec:bkgrnd} contains the necessary background on partial orders, simplicial complexes, lexicographic shellability and weak order on the symmetric group. In \cref{sec:cord def and properties}, we present the definition of a maximal chain descent order along with the natural example of the Boolean lattice and the weak order of type A. We also show several structural properties of these orders including proving \cref{thm:introduction shelling order with descent restriction iff linext of cord}. Then \cref{sec:polygon completeness} addresses cover relations of maximal chain descent orders and polygon completeness of {\EL}s. We prove the characterization of polygon complete {\EL}s in \cref{thm:characterization of non-polygon completeness}. We also give the sufficient condition for polygon complete {\EL}s called polygon strong. We then prove many well known {\EL}s are polygon strong, and so polygon complete. \cref{sec:inversions in ELs} presents a sufficient condition for polygon complete {\CL}s by introducing a notion of inversions for {\CL}s. In \cref{sec:first examples}, we discuss in depth several other examples of maximal chain descent orders. \cref{sec:first examples} is largely independent of the previous sections, and so can be read first if desired.

\section{Background on Posets, Lexicographic Shellability, and Weak Order on the Symmetric Group}\label{sec:bkgrnd}

\begin{subsection}{Posets and Simplicial Complexes} \label{subsec:bkgrposetssimpcx}
A partially ordered set or poset is a pair $(P,\leq)$ of 
a set $P$ and a binary relation $\leq$ on $P$ which is reflexive, transitive, and antisymmetric. For $x,y\in P$ satisfying $x\leq y$, the \textbf{closed interval} from $x$ to $y$ is the set $[x,y]=\sett{z\in P}{x\leq z\leq y}$. The \textbf{open interval} from $x$ to $y$ is defined analogously with strict inequalities and denoted $(x,y)$. We say that $y$ \textbf{covers} $x$, denoted $x\lessdot y$, if $x\leq z\leq y$ implies $z=x$ or $z=y$. $P$ is a \textbf{lattice} if each pair $x,y\in P$ has a unique least upper bound called the \textbf{join}, denoted $x\vee y$, and a unique greatest lower bound called the \textbf{meet}, denoted $x\wedge y$. We denote by $\zh$ (respectively $\oneh$) the unique minimal (respectively unique maximal) element of $P$ if such elements exist. If $P$ has both a $\zh$ and a $\oneh$, we say $P$ is \textbf{bounded}. If $P$ is bounded, we denote the proper part of $P$ as $\overline{P}=P\setminus \set{\zh,\oneh}$. We note that finite lattices are always bounded. The elements which cover $\zh$ are called \textbf{atoms}, and the elements which are covered by $\oneh$ are called \textbf{coatoms}. For $x,y\in P$ satisfying $x\leq y$, a \textbf{$\mathbf{k}$-chain from $\mathbf{x}$ to $\mathbf{y}$} in $P$ is a subset $C=\set{x_0,x_1,\dots,x_k}\subseteq P$ such that $x=x_0<x_1<\dots< x_k=y$. A chain $C$ is said to be \textbf{saturated} if $x_i\lessdot x_{i+1}$ for all $i$. A chain is said to be \textbf{maximal} if it is not properly contained in any other chain. We denote by $\mathbf{\boldsymbol\E(P)}$ the set of edges in the Hasse diagram of $P$, that is, the set of cover relations of $P$. And we denote by $\mathbf{\boldsymbol\M(P)}$ the set of maximal chains of $P$. We say P is \textbf{ranked} if there is a function $rk:P\to \mathbb{N}$ such that $rk(x)=0$ if $x$ is minimal in $P$ and $rk(y)=rk(x)+1$ if $x\lessdot y$ in $P$. We say $P$ is graded if all maximal chains of $P$ have the same length. We note that for non-bounded posets, ranked and graded are distinct concepts. We recall this fact because the distinction appears in this work. A map $f:P\to Q$ between posets $P$ and $Q$ is called \textbf{order preserving} or a \textbf{poset map} if $x\leq y$ in $P$ implies $f(x)\leq f(y)$ in $Q$. An order preserving bijection $f$ is called a \textbf{poset} isomorphism if $f^{-1}$ is also order preserving. An order preserving bijection $e:P\to [|P|]$ where $[|P|]$ has its usual total order is called a \textbf{linear extension} of $P$.

We will use the following four pieces of notation when working with chains. These notations will be ubiquitous in this work.

\begin{definition}\label{def:chain notations}
If $c_1$ and $c_2$ are chains such that the largest element of $c_1$ is less than or equal to the smallest element of $c_2$ in $P$, we denote the concatenated chain $c_1\cup c_2$ by $\mathbf{c_1*c_2}$. 

For a chain $c$ containing $x\in P$, we denote the subchain $c\cap [\zh,x]$ by $\mathbf{c^x}$, i.e. everything not above $x$ in $c$. We denote the subchain $c\cap [x,\oneh]$ by $\mathbf{c_x}$, i.e. everything not below $x$ in $c$. Then for $y\in c$ satisfying $x\leq y$, the subchain $c\cap [x,y]$ is denoted $\mathbf{c^y_x}$.

If $m$ and $c$ are chains, then $\mathbf{m\setminus c}$ denotes the subchain of $m$ with all elements of $c$ removed.
\end{definition}

The \textbf{order complex} of $P$, denoted $\Delta(P)$, is the abstract simplicial complex with vertices the elements of $P$ and $i$-dimensional faces the $i$-chains of $P$. The maximal chains of $P$ are precisely the facets of $\Delta(P)$. For $x,y\in P$, we denote by $\Delta(x,y)$ the order complex of the open interval $(x,y)$ as an induced subposet of $P$. Thus, when we refer to topological properties of $P$, we mean the topological properties of any geometric realization of $\Delta(P)$. In particular, the homotopy type of $P$ refers to the homotopy type of $\Delta(P)$. It is a well known theorem of P. Hall (see \cite{rotamobiusfunctions1964}) that $\mu_P$, the M{\"o}bius function of $P$, satisfies $\mu_P(x,y) = \tilde{\chi}(\Delta(x,y))$. Here, $\tilde{\chi}$ is the reduced Euler characteristic. This provides one of the important connections between the combinatorial and enumerative structure of a poset and its topology.
 
A \textbf{shelling} of a simplicial complex $\Delta$ is a total order $F_1,F_2,\dots, F_t$ of the facets of $\Delta$ such that $\overline{F_j}\cap (\cup_{i<j} \overline{F_i})$ is a pure codimension one subcomplex of $\overline{F_j}$ for each $1<j\leq t$. A simplicial complex which possesses a shelling is said to be \textbf{shellable}. A facet $F_j$ such that $\overline{F_j}\cap (\cup_{i<j} \overline{F_i})=\partial F_j$ is called a \textbf{homology facet}. Shellable simplicial complexes are homotopy equivalent to a (possibly empty) wedge of spheres with the spheres of each dimension indexed by the homology facets of that dimension. Let $\Delta_k =\bigcup_{1\leq i \leq k} \overline{F_i}$ be the subcomplex of $\Delta$ formed by deleting the last $t-k$ facets in the shelling for each $1\leq k\leq t$. A shelling induces a \textbf{restriction face} of each facet $F_j$ which is the minimal face of $F_j$ which is not contained in any facet prior to $F_j$ in the shelling. Denoting the restriction face of $F_j$ by $\mathbf{R(F_j)}$, we have $R(F_j)=\sett{x\in F_j}{F_j\setminus \set{x}\in \Delta_{j-1}}$. Facet $F_j$ is a homology facet if and only if $R(F_j)=F_j$. Taking the restriction face of each facet defines the \textbf{restriction map} $R$ of the shelling, a map from the facets to $\Delta$. For a face $f$ contained in facet $F$, we denote the set of faces of $F$ which contain $f$ by $[f,F]$. Then the following proposition gives a useful reformulation of a shelling.

\begin{proposition}[Proposition 2.5 \cite{shllblnonpureIbjornerwachs1996}]\label{prop:restriction version of shelling}
For a total ordering $F_1,F_2,\dots,F_t$ of the facets of a simplicial complex $\Delta$ and a map $R:\set{F_1,F_2,\dots,F_t}\to \Delta$, the following are equivalent:
\begin{itemize}
    \item [(1)] $F_1,F_2,\dots,F_t$ is a shelling and $R$ is its restriction map.
    \item [(2)] $\Delta$ is the disjoint union $\Delta=\bigsqcup_{1\leq i\leq t}[R(F_i),F_i]$ and $R(F_i)\subseteq F_j$ implies $i\leq j$ for all $i,j$.
\end{itemize}
\end{proposition}
\end{subsection} \label{sec:bkgrposetssimpcx}

\begin{subsection}{Lexicographic Shellability} \label{sec:bkgrlexshell}
In this subsection, we recall the notion of lexicographic shellability, both EL-labelings and CL-labelings. 

An \textbf{edge labeling} of poset $P$ is a map $\lambda:\E(P) \to \Lambda$ for a poset $\Lambda$, that is, a label $\lambda(x, y)$ for each cover relation $x\lessdot y$ in $P$. An edge labeling $\lambda$ induces a label sequence $\lambda(m)$ for each maximal chain $m\in \M(P)$ with $m:x_0\lessdot x_1 \lessdot \dots \lessdot x_{n-1}\lessdot x_n$ and $\lambda(m):\lambda_1(m),\lm_2(m)\dots,\lambda_n(m)$ where $\lm_i(m)=\lm(x_{i-1},x_i)$. We say $m\in \M(P)$ is an \textbf{ascending chain} with respect to $\lm$ if the label sequence $\lambda(m)$ is non-decreasing with respect to $\Lambda$. Further, we say that $m$ has a \textbf{descent} at position $i$ or we say that $x_i$ is a \textbf{descent} of $m$ if $\lambda_i(m)\not \leq \lambda_{i+1}(m)$ in $\Lambda$. We say that $m$ has an \textbf{ascent} at position $i$ or that $x_i$ is an \textbf{ascent} of $m$ if $\lambda_i(m)\leq \lambda_{i+1}(m)$ in $\Lambda$. Then $\Lambda$ induces a \textbf{lexicographic order} on maximal chains with $\lex{m}{m'}$ for $m,m'\in \M(P)$ if $i$ is the first index of the label sequences of $m$ and $m'$ at which they disagree and $\lambda_i(m)<\lambda_i(m')$ in $\Lambda$. We break ties in the lexicographic order arbitrarily, that is, maximal chains with identical label sequences are thought of as incomparable in the lexicographic order until an arbitrary total order is assigned to the maximal chains which share the same label sequence. For our constructions, we use a well known kind of edge labeling called an \textbf{EL-labeling} and a widely used generalization called a \textbf{CL-labeling}. EL-labelings and the notion of lexicographic shellability were introduced by Bj{\"o}rner in \cite{shellblposets}. CL-labelings were then introduced by Bj{\"o}rner and Wachs in \cite{bruhatshellbjornerwachs1982} and more deeply understood, particularly their recursive nature, by the same authors in \cite{lexshellposetsbjornerwachs1983}. Bj{\"o}rner and Wachs' initial work on shelling was focused on graded posets (pure simplicial complexes), but they extended the ideas of shelling, including both {\EL}s and {\CL}s, to non-graded posets (non-pure simplicial complexes) in \cite{shllblnonpureIbjornerwachs1996} and \cite{shllblnonpureIIbjornerwachs1997}.

\begin{definition}[Section 2 \cite{shellblposets}]\label{def:ellabeling}
An edge labeling $\lambda$ of a finite, bounded poset $P$ is an \textbf{EL-labeling} if for each pair $x,y\in P$ satisfying $x<y$, there is a unique ascending maximal chain with respect to $\lm$ in the closed interval $[x,y]$ and this ascending chain lexicographically precedes all other maximal chains in $[x,y]$.
\end{definition}

\begin{theorem}[Proof of Theorem 2.3 \cite{shellblposets} and Theorem 5.8 \cite{shllblnonpureIbjornerwachs1996}]\label{thm:elshell}
If finite, bounded poset $P$ admits an EL-labeling $\lm$, then any total order of the maximal chains of $P$ which is compatible with the lexicographic order on maximal chains induced by $\lm$ is a shelling order of the order complex $\Delta(P)$. Moreover, the restriction map of any such shelling is given by $R(m)=\sett{x\in m}{\text{x is a descent of $m$ w.r.t. $\lm$}}$ for any maximal chain $m\in \M(P)$, and the homology facets of the induced shelling of $\Delta(\overline{P})$ are given by the maximal chains $m\in \M(P)$ with descending label sequence with respect to $\lm$.
\end{theorem}

A generalization of an edge labeling is a \textbf{chain edge labeling}. Intuitively, a chain edge labeling is an edge labeling which depends on a choice of saturated chain from $\zh$ to the bottom of the edge. Let $\M\E(P)$ be the set of pairs $(m,x\lessdot y)$ for a maximal chain $m\in \M(P)$ and cover relation $x\lessdot y \in \E(P)$ such that $x\lessdot y$ is a cover relation in $m$. A \textbf{chain edge labeling} is a map $\lambda:\M\E(P) \to \Lambda$ for a poset $\Lambda$ such that if two maximal chains agree along their bottom $d$ edges then their labels of those edges are the same. Just as with an edge labeling, a chain edge labeling induces a label sequence of each maximal chain. To make an analogy with EL-labeling, we must restrict the label sequences of maximal chains to closed intervals $[x,y]$, but their is no unique restriction to $[x,y]$. However, fixing a maximal chain $r$ of $[\zh,x]$ determines a unique restriction of $\lambda$ to $\M\E([x,y])$. We call $r$ a \textbf{root} and the pair $[x,y]_r$ a \textbf{rooted interval}. Thus, we may refer to ascending and descending chains, the lexicographic order on maximal chains, and ascents and descents all with respect to $\lm$ and a root $r$.

\begin{definition}[Definition 3.2 \cite{bruhatshellbjornerwachs1982}]\label{def:cllabeling}
A chain edge labeling $\lm$ of a finite, bounded poset $P$ is a \textbf{CL-labeling} if each rooted interval $[x,y]_r$ has a unique ascending maximal chain with respect to $\lm$ and this ascending chain lexcographically precedes all other maximal chains in $[x,y]_r$.
\end{definition}

{\EL}s are {\CL}s in which edge labels do not depend on roots. Just like EL-labelings, CL-labelings induce lexicographic shellings of order complexes.

\begin{theorem}[Proof of Theorem 3.3 \cite{bruhatshellbjornerwachs1982} and Theorem 5.8 \cite{shllblnonpureIbjornerwachs1996}]\label{thm:clshell}
If finite, bounded poset $P$ admits a {\CL} $\lm$, then any total order of the maximal chains of $P$ which is compatible with the lexicographic order on maximal chains induced by $\lm$ is a shelling order of the order complex $\Delta(P)$. Moreover, the restriction map of any such shelling is given by $R(m)=\sett{x\in m}{\text{x is a descent of $m$ w.r.t. $\lm$}}$ for any maximal chain $m\in \M(P)$, and the homology facets of the induced shelling of $\Delta(\overline{P})$ are given by the maximal chains $m\in \M(P)$ with descending label sequence with respect to $\lm$.
\end{theorem}

Another useful fact which we record here for later use is that EL-labelings and CL-labelings restrict to EL-labelings and CL-labelings of closed intervals and rooted intervals, respectively.

\begin{proposition}\label{prop:restrictinglabelings}
The restriction of an EL-labeling to any closed interval is an EL-labeling. The restriction of a CL-labeling to any closed rooted interval is a CL-labeling.
\end{proposition}

\begin{proof}
These both follow from the recursive natures of \cref{def:ellabeling} and \cref{def:cllabeling}.
\end{proof}

\end{subsection}

\begin{subsection}{Weak Order on the Symmetric Group (Type A)}\label{sec:wkordsymmgroup}
Here we recall the definition and basics of the weak order of type A, that is, the weak order on the symmetric group of permutations. In \cref{sec:bkgrndyoungslat}, we briefly mention weak order on a general Coxeter group. For general Coxeter groups and proofs of the facts presented here, see \cite{bjornerbrenitcxgp} whose presentation we largely follow. 

Let $S_n$ be the symmetric group of permutations of $[n]$. Let $S$ be the set of adjacent transpositions in $S_n$, that is, $s_i=(i,i+1)$ for $i\in [n-1]$. We note that $(S_n,S)$ is isomorphic to the Coxeter system of type $A_{n-1}$. $(W,S)$ be a Coxeter system. The \textbf{length} of a permutation $w\in S_n$, denoted $l(w)$, is the minimal $k$ such that $w=s_{i_1}s_{i_2}\dots s_{i_k}$ for $s_{i_j}\in S$. We note that $l(w)$ is also the number of inversions of $w$, that is, the number of pairs of entries $i<j$ of $W$ which appear in decreasing order in the one line notation of $w$. We denote the set of inversions of $w$ by $\inv{w}{}$, so $l(w)=|\inv{w}{}|$. The \textbf{weak order} on $S_n$ is the partial order $\lewk{}{} $ defined by $\lewk{u}{w}$ if and only if $w=us_{i_1}s_{i_2}\dots s_{i_k}$ such that $s_{i_j}\in S$ and $l(us_{i_1}s_{i_2}\dots s_{i_j})=l(u)+j$ for each $1\leq j\leq k$. We will use a subscript of "$wk$" to refer to objects in weak order on $S_n$. An important property of weak order is that cover relations have a very nice form. Namely, $u\lessdot_{wk} w$ if and only $l(u)=l(w)-1$ and $u^{-1}w=s$ for some $s\in S$. Since we are dealing with right multiplication and right multiplication in the symmetric group corresponds to acting on the positions of permutations in one line notation, $u\lessdot_{wk} w$ for permutations $u,w\in S_n$ if and only if $w$ is obtained by transposing an ascent of $u$ in one line notation. The corresponding positions in $w$ will thus have a descent. The weak order on the symmetric group is well known to be a lattice. There is also another useful description of weak order on $S_n$ in terms of containment of inversion sets.
 
\begin{proposition}\label{prop:wkordinvsetcontainment}
For $u,w\in S_n$, $\lewk{u}{w}$ if and only if $\inv{u}{}\subseteq \inv{w}{}$.
\end{proposition}

\end{subsection}

\section{Maximal Chain Descent Orders: Definition and Fundamental Properties} \label{sec:cord def and properties}

\begin{subsection}{Introduction to Maximal Chain Descent Orders}\label{sec:def of cord}
 
Now we introduce the notion of a maximal chain descent order induced by a {\CL} and give several fundamental properties of these partial orders. We will not assume our posets are graded. Since {\EL}s are instances of {\CL}s, the following constructions and properties work for {\EL}s as well. We begin with two definitions which together describe the polygon moves giving rise to maximal chain descent orders.

\begin{definition}\label{def:differ by diamond}
Let $P$ be a finite, bounded poset. Let $m,m'\in \M(P)$ be maximal chains of $P$ with $m:x_0=\zh\lessdot x_1\lessdot \dots \lessdot x_{r-1} \lessdot x_r=\oneh$ and $m':x'_0=\zh\lessdot x'_1\lessdot \dots \lessdot x'_{r'-1} \lessdot x'_{r'}=\oneh$. Suppose $r' \leq r$ since $P$ is not necessarily graded. We say that $m$ and $m'$ \textbf{differ by a polygon} if there is some $1 \leq i\leq r-1$ such that $x'_j=x_j$ for all $j<i$, $x'_{i+1}=x_{i+l}$ for some $1\leq l$, $x'_i\neq x_{i+k}$ for all $0\leq k\leq l$, and $x'_{i+1+m}=x_{i+l+m}$ for all $0\leq m\leq r-i-l=r'-i-1$.
\end{definition}

Maximal chains differing by a polygon are illustrated in \cref{fig:differ by a polygon} below. Intuitively, $m$ and $m'$ differ by a polygon if they agree everywhere except on an interval where $m'$ has length two. 

\begin{figure}[H]
    \centering
    \scalebox{1}{
    \begin{tikzpicture}[very thick]
    \node (xi-1) at (0,0) {$x_{i-1}$};
    \node (xi) at (-1,1) {$x_i$};
    \node (xi+l) at (0,3) {$x_{i+l}$};
    \node (x'i) at (1,1.5) {$x'_i$};
    \node (zh) at (0,-2) {$\zh$};
    \node (oneh) at (0,5) {$\oneh$};
    \node (m) at (-2,1.5) {$m$};
    \node (m') at (2,1.5) {$m'$};
    
    \draw (xi-1) -- (xi);
    \draw (xi-1) -- (x'i) -- (xi+l);
    \draw [dashed] (xi) -- (xi+l);
    \draw [dashed] (zh) -- (xi-1);
    \draw [dashed] (xi+l) -- (oneh);
    \end{tikzpicture}
    }
    \caption{Maximal chains which differ by a polygon.}
    \label{fig:differ by a polygon}
\end{figure}
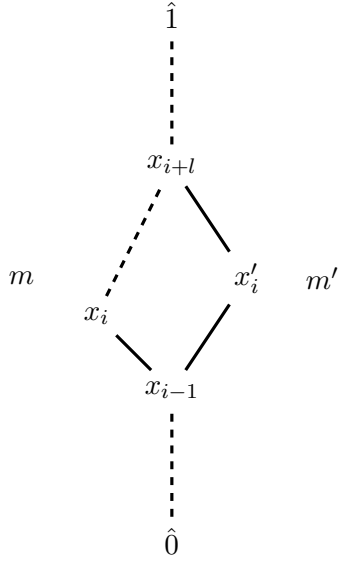

\begin{definition}\label{def:increase and polygon}
Let $P$ be a finite, bounded poset with a {\CL} $\lm$. We say that \textbf{$\mathbf{m}$ increases by a polygon move to $\mathbf{m'}$} and write $\mathbf{m \boldsymbol\to m'}$ if $m$ and $m'$ differ by a polygon as in \cref{def:differ by diamond} and $m$ restricted to the rooted interval $[x_{i-1},x'_{i+1}]_{m^{x_{i-1}}}$ is the unique ascending maximal chain with respect to $\lm$ while $m'$ restricted to the rooted interval $[x_{i-1},x'_{i+1}]_{m^{x_{i-1}}}$ is a descent with respect to $\lm$. We refer to the pair of saturated chains $m_{x_{i-1}}^{x'_{i+1}}$ and ${m'}_{x_{i-1}}^{x'_{i+1}}$ as the \textbf{polygon corresponding to $\mathbf{m\to m'}$}. Thus, \textbf{the bottom element of the polygon corresponding to $\mathbf{m\boldsymbol\to m'}$} refers to $x_{i-1}$ and \textbf{the top element of the polygon corresponding to $\mathbf{m\boldsymbol\to m'}$} refers to $x'_{i+1}$.
\end{definition}

Intuitively, $m\to m'$ if $m$ and $m'$ differ by a polygon and $m$ is the lexicographically least (hence, ascending) chain when restricted to the rooted interval where they differ. 

\begin{example}\label{ex:max chain increases and polygons}
\cref{fig:example of max chain increase and corresponding polygon} subfigure (a) shows a poset with an {\EL} which has two increases by polygon moves: $(\zh \lessdot a \lessdot b \lessdot c \oneh) \to (\zh \lessdot a \lessdot d \lessdot \oneh)$ and $(\zh \lessdot a \lessdot b \lessdot c \lessdot \oneh) \to (\zh \lessdot a \lessdot e \lessdot \oneh)$. The polygon corresponding to $(\zh \lessdot a \lessdot b \lessdot c \lessdot \oneh) \to (\zh \lessdot a \lessdot e \lessdot \oneh)$ is depicted in subfigure (b). We emphasize that $(\zh \lessdot a \lessdot d \lessdot \oneh) \not \to (\zh \lessdot a \lessdot e \lessdot \oneh)$ despite being lexicographically smaller because $(\zh \lessdot a \lessdot d \lessdot \oneh)$ is not the ascending (and thus, not the lexicographically least) saturated chain of $[a,\oneh]$.

\begin{figure}[H]
    \centering
    \subfigure[Poset with an {\EL}.]{
    \scalebox{1}{
    \begin{tikzpicture}[very thick]
    \node (zh) at (0,0) {$\zh$};
    \node (a) at (0,1) {$a$};
    \node (b) at (-1,2) {$b$};
    \node (c) at (-1,3) {$c$};
    \node (oneh) at (0,4) {$\oneh$};
    \node (d) at (1,2) {$d$};
    \node (e) at (2,2) {$e$};
    
    \draw (zh) --node[left,blue]{$1$} (a) --node[left,blue]{$2$} (b) --node[left,blue]{$3$} (c) --node[left,blue]{$4$} (oneh) --node[left,blue]{$2$}  (d) --node[above left,blue]{$4$} (a);
    \draw (a) --node[below right,blue]{$5$} (e) --node[right,blue]{$1$} (oneh);
    \end{tikzpicture}
    }
    }
    \hspace{5mm}
    \subfigure[Polygon corresponding to $(\zh \lessdot a \lessdot b \lessdot c \lessdot \oneh) \to (\zh \lessdot a \lessdot e \lessdot \oneh).$]{
    \scalebox{1}{
    \begin{tikzpicture}[very thick]
    \node (a) at (0,1) {$a$};
    \node (b) at (-1,2) {$b$};
    \node (c) at (-1,3) {$c$};
    \node (oneh) at (0,4) {$\oneh$};
    \node (e) at (2,2) {$e$};
    
    \draw (a) --node[left,blue]{$2$} (b) --node[left,blue]{$3$} (c) --node[left,blue]{$4$} (oneh);
    \draw (a) --node[below right,blue]{$5$} (e) --node[right,blue]{$1$} (oneh);
    \end{tikzpicture}
    }
    }
    \caption{An {\EL} with two increases by polygon moves.}
    \label{fig:example of max chain increase and corresponding polygon}
\end{figure}
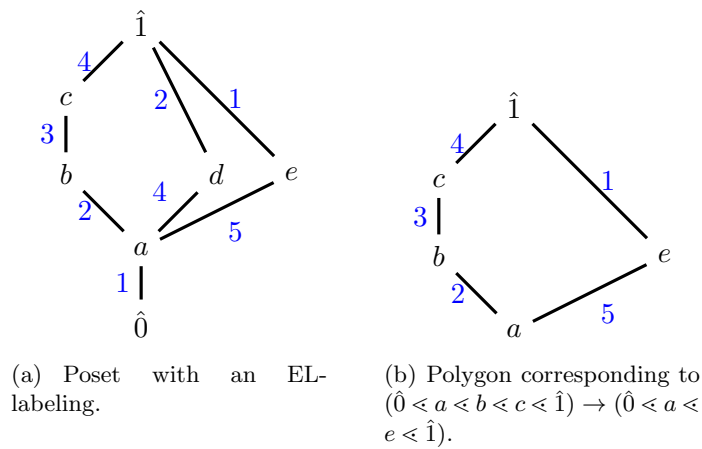
\end{example}

We give two simple and useful propositions about increases by polygon moves before proceeding to the definition of a maximal chain descent order.

\begin{proposition}\label{prop:nongradedmaxchainincreaselexincrease}
Let $P$ be a finite, bounded poset with a {\CL} $\lm$. If $m\to m'$, then $\lex{\lm(m)}{\lm(m')}$.
\end{proposition}

\begin{proof}
This follows from the fact that the ascending chain in any rooted interval lexicographically precedes all other maximal chains in that rooted interval.
\end{proof}

\begin{proposition}\label{prop:descents give unique incrs}
Let $P$ be a finite, bounded poset with a {\CL} $\lm$. If $m' \in \M(P)$ has a descent at $x\in m'$ with respect to $\lm$, then there is an unique $m\in \M(P)$ such that $m\to m'$ and $m'\setminus m=\set{x}$. 
\end{proposition}

\begin{proof}
Let $w$ and $z$ be the elements of $m'$ satisfying $w\lessdot x\lessdot z$. Let $c$ be the unique ascending maximal chain of the rooted interval $[w,z]_{m'^{w}}$ with respect to $\lm$ guarateed by the definition of a {\CL}. Set $m=m'^w * c * m'_z$. Then by \cref{def:increase and polygon}, we have $m\to m'$ with $m'\setminus m=\set{x}$ and $m$ is unique.
\end{proof}

Next we introduce the notion of a maximal chain descent order.

\begin{definition}[Hersh and L.]\label{def:nongradedmaxchainorders}
Let $P$ be a finite, bounded poset with a {\CL} $\lm$. The \textbf{maximal chain descent order induced by $\boldsymbol\lambda$} is the partial order $\mathbf{\cleq{ \boldsymbol}{_{\boldsymbol\lambda}}}$ on the maximal chains $\M(P)$ defined as the reflexive and transitive closure of the relation $m\to m'$, i.e. if $m$ increases by a polygon move to $m'$ with respect to $\lambda$. Denote the poset $(\M(P),\cleq{ }{_\lambda })$ by $\mathbf{\cord{P}{\boldsymbol\lambda}}$.
\end{definition}

\begin{remark}\label{rmk:cord def works for EL}
We note as above that \cref{def:nongradedmaxchainorders} applies to {\EL}s since they are themselves {\CL}s. 
\end{remark}

\begin{remark}\label{rmk:proper part shelling gives same poset}
Let $P$ be a finite, bounded poset with a {\CL} $\lm$. Since the proper part of $P$ often has more interesting topology than $P$ itself (a bounded poset is contractible), we often consider the lexicographic shelling induced on $\Delta(\overline{P})$ by simply deleting the cone points $\zh$ and $\oneh$ from $\Delta(P)$. One might wonder if this lexicographic shelling of $\overline{P}$ gives a different maximal chain descent order than $\cord{P}{\lm}$. It does not. The facets of $\Delta(\overline{P})$ are simply the facets of $\Delta(P)$ with $\zh$ and $\oneh$ deleted while the shelling order, all codimension one intersections, and all face containments are preserved by deleting the cone points. Thus, the maximal chain descent orders are isomorphic via the map which deletes $\zh$ and $\oneh$ from maximal chains of $P$.
\end{remark}

A maximal chain descent order induced by an {\EL} is shown in \cref{ex:cord which is not lex order}. Using \cref{prop:nongradedmaxchainincreaselexincrease}, we can easily show that the relation of \cref{def:nongradedmaxchainorders} is antisymmetric, and so honestly a partial order. We also easily have the corollary that lexicographic order is at least as fine as the corresponding maximal chain descent order.

\begin{corollary}\label{cor:nongradedcordgreaterlexgreater}
Let $P$ be a finite, bounded poset with a {\CL} $\lm$. If $\cl{m}{m'}{\lm}$, then $\lex{\lm(m)}{\lm(m')}$.
\end{corollary}

\begin{example}\label{ex:cord which is not lex order}
\cref{fig:notlexposet} shows a poset $P$ with an {\EL} $\lm$ along with its induced maximal chain descent order $\cord{P}{\lm}$ and the induced lexicographic order on maximal chains. This example shows that $\cord{P}{\lm}$ may be strictly coarser than the induced lexicographic order.

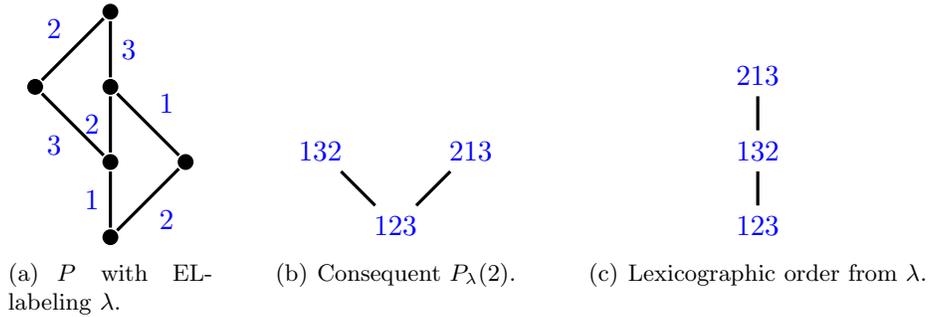
\begin{figure}[H]
    \centering
    \subfigure[$P$ with {\EL} $\lm$.]{
    \scalebox{1}{
    \begin{tikzpicture}[very thick]
    \node[fill,circle,inner sep=0pt,minimum size=6pt] (zh) at (0,0) {};
    \node[fill,circle,inner sep=0pt,minimum size=6pt] (a) at (0,1) {};
    \node[fill,circle,inner sep=0pt,minimum size=6pt] (b) at (0,2) {};
    \node[fill,circle,inner sep=0pt,minimum size=6pt] (c) at (1,1) {};
    \node[fill,circle,inner sep=0pt,minimum size=6pt] (oneh) at (0,3) {};
    \node[fill,circle,inner sep=0pt,minimum size=6pt] (d) at (-1,2) {};
    
    \draw (zh) --node[left,blue]{$1$} (a) --node[left,blue]{$2$} (b) --node[right,blue]{$3$} (oneh) --node[above left,blue]{$2$}  (d) --node[below left,blue]{$3$} (a);
    \draw (zh) --node[below right,blue]{$2$} (c) --node[above right,blue]{$1$} (b);
    \end{tikzpicture}
    }
    }
    \hspace{5mm}
    \subfigure[Consequent $\cord{P}{\lm}$.]{
    \scalebox{1}{
    \begin{tikzpicture}[very thick]
    \node[blue] (123) at (0,0) {$123$};
    \node[blue] (132) at (-1,1) {$132$};
    \node[blue] (213) at (1,1) {$213$};
    
    \draw (123) -- (132);
    \draw (123) -- (213);
    \end{tikzpicture}
    }
    }
    \hspace{5mm}
    \subfigure[Lexicographic order from $\lm$.]{
    \scalebox{1}{
    \begin{tikzpicture}[very thick]
    \node[blue] (123) at (0,0) {$123$};
    \node[blue] (132) at (0,1) {$132$};
    \node[blue] (213) at (0,2) {$213$};

    \node (space1) at (-2,0) {};
    \node (space2) at (2,0) {};
    
    \draw (123) -- (132) -- (213);
    \end{tikzpicture}
    }
    }
    \caption{An {\EL} with distinct lexicographic order and maximal chain descent order.}
    \label{fig:notlexposet}
\end{figure}
\end{example}

\end{subsection}

Now we turn to our motivating example of a maximal chain descent order before proving several fundamental properties.

\begin{subsection}{Motivating Example: Maximal Chain Descent Order for the Boolean Lattice is Weak Order on the Symmetric Group}\label{sec:motiv example boolean lat weak Sn}
Perhaps the most natural example of an EL-labeling is the labeling $\lm$ of the Boolean lattice $B_n$, the lattice of subsets of $[n]$ ordered by containment, in which $B\lessdot B'$ precisely when $B'=B\cup \set{i}$ for some $i\in [n]\setminus B$ and $\lm(B, B')=i$. We will prove that $\cord{{B_n}}{\lm}$ is isomorphic to the weak order on $S_n$ via the map assigning each maximal chain to its label sequence with respect to $\lm$.

\begin{theorem}\label{thm:boolean lat cord weak order symm grp}
Let $B_n$ be the Boolean lattice of subsets of $[n]$ with its standard {\EL} $\lm$. Then the map $m\mapsto \lm(m)$ is an isomorphism from the maximal chain descent order $\cord{{B_n}}{\lm}$ to the weak order on $S_n$ (the type A weak order).
\end{theorem}

\begin{proof}
The label sequences of the maximal chains $\M(B_n)$ are precisely the permutations of $[n]$. Moreover, each permutation $\pi$ occurs as the label sequence of exactly one maximal chain of $B_n$, namely, the maximal chain $m_{\pi}$ whose rank $i$ element is the union of the first $i$ entries of $\pi$ in one-line notation. For instance, $m_{3241}:\emptyset \lessdot \set{3} \lessdot \set{2,3} \lessdot \set{2,3,4} \lessdot \set{1,2,3,4}$ and $\lm(m_{3241})=3241$.

Every rank two interval in $B_n$ has the form shown in \cref{fig:rank two boolean int} where $A\subset [n]$ has $|A|\leq n-2$ and $i,j\in [n]\setminus A$ with $i<j$.

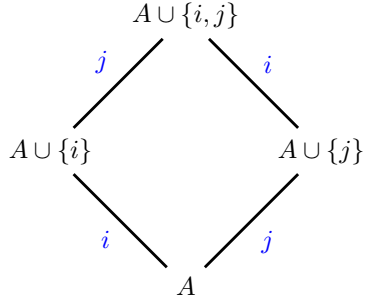
\begin{figure}[H]
    \centering
    \scalebox{0.9}{
    \begin{tikzpicture}[very thick]
    \node (A) at (0,0) {$A$};
    \node (Ai) at (-2,2) {$A\cup\set{i}$};
    \node (Aj) at (2,2) {$A\cup \set{j}$};
    \node (Aij) at (0,4) {$A\cup \set{i,j}$};
    
    \draw (A) --node[below left,blue]{$i$} (Ai) --node[above left,blue]{$j$} (Aij);
    \draw (A) --node[below right,blue]{$j$} (Aj) --node[above right,blue]{$i$} (Aij);
    \end{tikzpicture}
    }
    \caption{Typical rank two interval in the Boolean lattice.}
    \label{fig:rank two boolean int}
\end{figure}

Thus, if $m\to m'$ for maximal chains $m,m'\in \M(B_n)$, then $\lm(m')$ is obtained from $\lm(m)$ by transposing a unique pair of adjacent entries of $\lm(m)$ which are in ascending order in $\lm(m)$ and descending order in $\lm(m')$. For example, $m_{3214}:(\emptyset \lessdot \set{3} \lessdot \set{2,3} \lessdot \set{1,2,3} \lessdot \set{1,2,3,4}) \to m_{3241}:(\emptyset \lessdot \set{3} \lessdot \set{2,3} \lessdot \set{2,3,4} \lessdot \set{1,2,3,4})$ and $\lm(m_{3214})=3214$ while $\lm(m_{3241})=3241$. Therefore, if $m\to m'$, then $\lm(m)$ is covered by $\lm(m')$ in weak order on the symmetric group $S_n$. Moreover, this implies that if $\cl{m}{m'}{\lm}$, then $\lm(m)<_{wk} \lm(m')$ in weak order on $S_n$. Hence, if $m\to m'$, then $m$ is covered by $m'$ in $\cord{P}{\lm}$; this follows by contradiction because supposing that $m\to m'$ and $m\to \cl{m''}{m'}{\lm}$ for some other maximal chain $m''$ implies that $\lm(m)<_{wk} \lm(m'') <_{wk} \lm(m')$ which contradicts the fact that $\lm(m)\lessdot_{wk} \lm(m')$. Lastly, it is clear by construction that if $\pi \lessdot_{wk} \sigma$ for permutations $\pi,\sigma\in S_n$, then $m_{\pi}\to m_{\sigma}$. Thus, $\cordoti{m}{m'}{\lm}$ if and only if $\lm(m)\lessdot_{wk}\lm(m')$. Therefore, taking label sequences gives a poset isomorphism from $\cord{{B_n}}{\lm}$ to weak order on $S_n$.
\end{proof}

\end{subsection}

In \cref{sec:first examples}, we present more examples in depth, and for the most part, they do not require reading the remainder of this section.

\begin{subsection}{Fundamental Properties of Maximal Chain Descent Orders}\label{sec:fundamental properties}

Here we prove several fundamental structural properties of maximal chain descent orders. We note as above that while these statements are written for {\CL}s, they also apply to {\EL}s since {\EL}s are {\CL}s. The notations of \cref{def:chain notations} are used ubiquitously in this section.

\begin{proposition}\label{prop:non ranked cords have zero hat}
Let $P$ be a finite, bounded poset with a {\CL} $\lm$. Then the maximal chain descent order $\cord{P}{\lm}$ has a unique minimal element, denoted by $\zh$, given by the unique ascending maximal chain of $P$ with respect to $\lm$.
\end{proposition}

\begin{proof}
This proof uses the same central idea as the proof that a lexicographic order from a {\CL} induces a shelling order of $\Delta(P)$. Let $m_1,\dots,m_t$ be a total order of the maximal chains $\M(P)$ which is compatible with the lexicographic order induced by $\lambda$. By definition of a {\CL}, $m_1$ is the unique ascending maximal chain with respect to $\lambda$. We will proceed by induction on the index in this total order to show that $\cleq{m_1}{_\lambda m_i}$ for all $1\leq i\leq t$. This is trivial when $i=1$. Assume it holds for all $1\leq k\leq i$ for some $i\geq 1$. Since $i+1>1$, $m_{i+1}$ is not the unique ascending maximal chain of $P$, so $m_{i+1}$ has a descent at some position $j$. Thus, $m_{i+1}:x_0\lessdot x_1 \lessdot \dots \lessdot x_{j-1} \lessdot x_j \lessdot x_{j+1}\lessdot \dots \lessdot x_{n-1}\lessdot x_n$ with $\lambda( x_{j-1}, x_j)\not \leq \lambda(x_j, x_{j+1})$. Hence, $m_{i+1}$ restricted to the rooted interval $[x_{j-1},x_{j+1}]_{m_{i+1}^{x_{j-1}}}$ is not the unique ascending maximal chain of the rooted interval $[x_{j-1},x_{j+1}]_{m_{i+1}^{x_{j-1}}}$. Then since $\lambda$ is a {\CL}, there is a unique ascending saturated chain $c$ in the rooted interval $[x_{j-1},x_{j+1}]_{m_{i+1}^{x_{j-1}}}$ with respect to $\lm$ which lexicographically precedes $m_{i+1}$ restricted to the rooted interval $[x_{j-1},x_{j+1}]_{m_{i+1}^{x_{j-1}}}$. Then $m_{i+1}^{x_{j-1}} * c * {m_{i+1}}_{x_{j+1}} \to m_{i+1}$. Thus, $\cl{ m_{i+1}^{x_{j-1}} * c * {m_{i+1}}_{x_{j+1}} }{m_{i+1}}{\lm}$. And $m_{i+1}^{x_{j-1}} * c * {m_{i+1}}_{x_{j+1}}$ lexicographically precedes $m_{i+1}$, so $\cleqi{m_1}{m_{i+1}^{x_{j-1}} * c * {m_{i+1}}_{x_{j+1}}}{\lm}$ by the inductive hypothesis. Hence, $\cl{m_1}{m_{i+1}}{\lm}$, so $m_1$ is the $\zh$ of $\cord{P}{\lambda}$ by induction.
\end{proof}

On the other hand, descending label sequences give maximal elements of the maximal chain descent order. They do not give all maximal elements though as witnessed by \cref{ex:cord which is not lex order}.

\begin{proposition}\label{prop:decrchainmaxlelt}
Let $P$ be a finite, bounded poset which admits a {\CL} $\lm$. If maximal chain $m\in \M(P)$ has a descending label sequence with respect to $\lm$, then $m$ is a maximal element of $\cord{P}{\lm}$.
\end{proposition}

\begin{proof}
Since $m$ is descending with respect to $\lm$, $m$ has no ascents. Thus, there are no chains $m'\in \M(P)$ such that $m\to m'$ by \cref{def:increase and polygon}. So, $m$ is a maximal element of $\cord{P}{\lm}$.
\end{proof}

Since a {\CL} restricted to a rooted closed interval is a {\CL} of that interval (see \cref{prop:restrictinglabelings}), we can consider the maximal chain descent order induced by $\lm$ on the maximal chains of the interval. The next lemma shows that order relations in the maximal chain descent order of a rooted closed interval in a sense lift to order relations in the maximal chain descent order on the entire poset. This is an expression of the recursive nature of {\CL}s. We use this lemma ubiquitously.

\begin{lemma}\label{lem:restrictedrelslift}
Let $P$ be a finite, bounded poset which admits a {\CL} $\lambda$. Let $m=x_0\lessdot x_1 \lessdot \dots \lessdot x \lessdot \dots \lessdot y \lessdot \dots \lessdot x_{n-1}\lessdot x_n$ be a maximal chain of $P$ and let $c$ and $c'$ be maximal chains of the rooted interval $[x,y]_{m^x}$. If $\cleq{c}{_{\lm} c'}$ in the maximal chain descent order of the rooted interval $[x,y]_{m^x}$ induced by the restriction of $\lm$, then $\cleqi{m^x * c * m_y}{ m^x * c' * m_y}{\lm}$ in $\cord{P}{\lm}$.
\end{lemma}

\begin{proof}
First, we show that if $c\to c'$, then $m^x * c * m_y\to m^x * c' * m_y$. Then the result follows since $c\to c'$ are precisely the relations whose reflexive and transitive closure give $\cord{[x,y]}{\lm}$ and $m^x * c * m_y\to m^x * c' * m_y$ are among the relations whose reflexive and transitive closure give $\cord{P}{\lm}$. If $c$ and $c'$ differ by a polygon in $[x,y]$, then $m^x * c * m_y$ and  $m^x * c' * m_y$ also differ by a polygon in $P$. Also, the restriction of $\lm$ to the rooted interval $[x,y]_{m^x}$ is compatible with $\lm$ on $P$. Thus, an ascent or descent in $c$ or $c'$ gives a an ascent or descent in $m^x * c * m_y$ or $m^x * c' * m_y$, respectively. Thus, $c\to c'$ implies $m^x * c * m_y\to m^x * c' * m_y$. 
\end{proof}

As a corollary of \cref{lem:restrictedrelslift} and \cref{prop:non ranked cords have zero hat} we describe certain general order relations in a maximal chain descent order. In particular, if maximal chain $m$ is ascending on each interval where it differs from maximal chain $m'$, then $\cl{m}{m'}{\lm}$. This corollary also provides the key fact used in our later proofs about shelling orders.

\begin{corollary}\label{cor:order rel from ascending on differing intervals}
Let $P$ be a finite, bounded poset with a {\CL} $\lm$. Let $m,m'\in M(P)$ be distinct maximal chains of $P$ and let $y_0<y_1<y_2<\dots<y_k$ be all of the elements of $m$ satisfying $y_l\in m\cap m'$ for each $0\leq l\leq k$ while $m'\cap(y_l,y_{l+1})$ and $m\cap(y_l,y_{l+1})$ are non-empty and disjoint for each $0\leq l\leq k-1$. Suppose that for each $0\leq l\leq k-1$ $m$ is ascending with respect to $\lm$ when restricted to the rooted interval $[y_l,y_{l+1}]_{m^{y_l}}$. Then $\cl{m}{m'}{\lm}$ in $\cord{P}{\lm}$.
\end{corollary}

\begin{proof}
We have $k\geq 1$ since $m\neq m'$. We induct on $k$. When $k=1$, $\cl{m}{m'}{\lm}$ by \cref{prop:non ranked cords have zero hat} and \cref{lem:restrictedrelslift} directly.

Assume that $\cl{m}{m'}{\lm}$ when $k=n$ for some $n\geq 1$. Suppose $k=n+1$. By \cref{prop:non ranked cords have zero hat} and \cref{lem:restrictedrelslift} $\cl{m}{m^{y_n} * m'^{y_{n+1}}_{y_n} * m_{y_{n+1}} }{\lm}$. Now by construction $m^{y_k} * m'^{y_{n+1}}_{y_n} * m_{y_{n+1}}$ differs from $m'$ in $n$ intervals. By definition of a {\CL} the label sequences of $m$ and $m^{y_n} * m'^{y_{n+1}}_{y_n} * m_{y_{n+1}}$ agree up to $y_n$. Thus, $m^{y_n} * m'^{y_{n+1}}_{y_n} * m_{y_{n+1}}$ is ascending with respect to $\lm$ when restricted to the $n$ rooted intervals on which it differs from $m'$. Then the inductive hypothesis implies $\cl{m^{y_n} * m'^{y_{n+1}}_{y_n} * m_{y_{n+1}}}{m'}{\lm}$. Therefore, $\cl{m}{m'}{\lm}$, and the result holds by induction.
\end{proof}

\begin{remark}\label{rmk:not all order rels from increasing}
Not every order relation in a maximal chain descent order is of the form in \cref{cor:order rel from ascending on differing intervals}. We observe this in the example in \cref{fig:overlapping polygons case lower len two int}. The maximal chains labeled $1265$ and $3214$ are comparable, but neither is the ascending maximal chain of the entire poset with respect to the labeling. This can also be observed in the minimal labeling of the partition lattice in \cref{fig:Pi4 with min labeling} and \cref{fig:cord from Pi4 min labeling}.
\end{remark}

We now observe several examples which begin to reveal the subtlety of cover relations and rank in maximal chain descent orders.

\begin{remark}\label{rmk:polygon not always covers}
While every cover relation in a maximal chain descent order corresponds to an increase by a polygon move, it is perhaps surprising that not all increases by polygon moves result in cover relations. The example in \cref{fig:noncover diamond move example} and \cref{ex:CL non polygon complete} both exhibit this.
\end{remark}

\begin{example}\label{ex:CL non polygon complete}
\cref{fig:CL noncover diamond move example} shows that an increase by a polygon move with respect to a {\CL} which is not an {\EL} need not give a cover relation in the induced maximal chain descent order. In this example, the only label which depends on the root is the label of $x_1\lessdot \oneh$ which is the red $1$ if the maximal chain labeled in red is used and the blue $3$ otherwise. We observe that this is the dual of Bj{\"o}rner and Wachs' {\CL} of the Bruhat order of $S_3$ from \cite{bruhatshellbjornerwachs1982}. We have that $(\zh \lessdot z_2 \lessdot x_1 \lessdot \oneh) \to (\zh \lessdot z_1 \lessdot x_1 \lessdot \oneh) $ while $(\zh \lessdot z_2 \lessdot x_1 \lessdot \oneh)$ is not covered by $(\zh \lessdot z_1 \lessdot x_1 \lessdot \oneh)$ in $\cord{P}{\lm}$. 
\end{example}

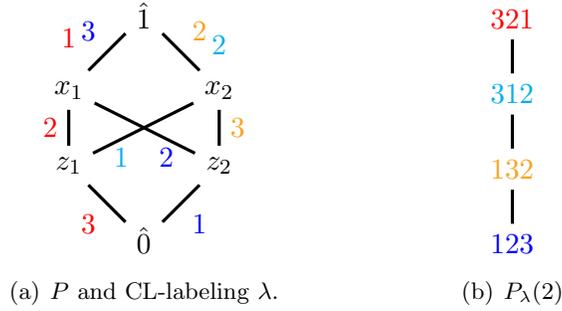
\begin{figure}[H]
    \centering
    \subfigure[$P$ and {\CL} $\lm$.]{
    \scalebox{1}{
    \begin{tikzpicture}[very thick]
    \node (zh) at (0,0) {$\zh$};
    \node (a1) at (-1,1) {$z_1$};
    \node (a2) at (1,1) {$z_2$};
    \node (b1) at (-1,2) {$x_1$};
    \node (b2) at (1,2) {$x_2$};
    \node (oneh) at (0,3) {$\oneh$};
    \node[red] (label1) at (-1,2.7) {$1$};

    \node (1) at (-0.3,1.1) {\textcolor{Cyan}{$1$}};
    \node (2) at (0.3,1.1) {\textcolor{blue}{$2$}};
    \node (c2) at (1,2.6) {\textcolor{Cyan}{$2$}};

    \node (spacer) at (2.5,0) {};
    \node (spacer) at (-2.5,0) {};
    
    \draw (zh) --node [red, below left]{$3$} (a1) --node [red, left]{$2$} (b1) --node [blue, above left]{$3$} (oneh);
    \draw (zh) --node [blue, below right]{$1$} (a2) --node [YellowOrange, right]{$3$} (b2) --node [YellowOrange, above right]{$2$} (oneh);
    \draw (a1) -- (b2);
    \draw (a2) -- (b1);
    \end{tikzpicture}
    }
    }
    \subfigure[$\cord{P}{\lm}$]{
    \scalebox{1}{
    \begin{tikzpicture}[very thick]
    \node[blue] (123) at (0,0) {$123$};
    \node[YellowOrange] (154) at (0,1) {$132$};
    \node[Cyan] (534) at (0,2) {$312$};
    \node[red] (543) at (0,3) {$321$};

    \node (spacer1) at (-1.5,0) {};
    \node (spacer2) at (1.5,0) {};
    
    \draw (123) -- (154) -- (534) -- (543);
    \end{tikzpicture}
    }
    }
    \caption{A {\CL} which is not polygon complete.}
    \label{fig:CL noncover diamond move example}
\end{figure}

\cref{fig:noncover diamond move example} and \cref{ex:CL non polygon complete} lead us to introduce the following definition.

\begin{definition}\label{def:polygon complete}
 Let $\lm$ be a {\CL} of a finite, bounded poset $P$. We say $\lm$ is \textbf{polygon complete} if $m\to m'$ implies $\cordoti{m}{m'}{\lm}$ in $\cord{P}{\lm}$. 
\end{definition}

In \cref{sec:polygon completeness}, we give a rather technical characterization of polygon complete {\EL}s. In that section and in \cref{sec:inversions in ELs}, we also give simpler conditions which are sufficient for polygon completeness or imply an {\EL} is not polygon complete. But there is still more we can say about cover relations in general, namely the number of elements which a maximal chain covers in a maximal chain descent order is at most the number of descents of the maximal chain with respect to the labeling.

\begin{lemma}\label{lem:number downward cord covers}
Let $P$ be a finite, bounded poset with a {\CL} $\lm$. For a maximal chain $m\in \M(P)$, the number of maximal chains $m'\in \M(P)$ such that $\cordoti{m'}{m}{\lm}$ is at most the number of descents of $m$ with respect to $\lm$. Moreover, if $\lm$ is polygon complete in the sense of \cref{def:polygon complete}, then the number of maximal chains $m'\in \M(P)$ such that $\cordoti{m'}{m}{\lm}$ is the number of descents of $m$ with respect to $\lm$.
\end{lemma}

\begin{proof}
By \cref{def:increase and polygon} and  \cref{prop:descents give unique incrs} the number of maximal chains $m'\in \M(P)$ such that $m'\to m$ is exactly the number of descents of $m$ with respect to $\lm$. Then since $m'\to m$ does not necessarily give a cover relation, the number of downward cover relations from $m$ is at most the number of descents of $m$ with respect to $\lm$. Further, if $\lm$ is polygon complete, then $m'\to m$ does give a cover relation. Thus, the number of downward cover relations from $m$ is the number of descents of $m$ with respect to $\lm$.
\end{proof}

\begin{example}\label{ex:overlapping polygons case lower len two int}
\cref{fig:overlapping polygons case lower len two int} shows a poset $P$ with an {\EL} $\lm$ and the resulting maximal chain descent order. The maximal chain descent order in this case is not ranked. Notice also that the number of downward cover relations in $\cord{P}{\lm}$ from each maximal chain of $P$ is at most the number of descents in its label sequence. In particular, the upper bound is reached by all maximal chains of $P$ except for the maximal chain labeled $2654$. The label seqence $2654$ has two descents, but $2134\to 2654$ while $\cordoti{2134\not}{2654}{\lm}$ in $\cord{P}{\lm}$.

\begin{figure}[H]
    \centering
    \subfigure[{$P$ with {\EL} $\lm$.}]{
    \scalebox{1}{
    \begin{tikzpicture}[very thick]
    \node[fill,circle,inner sep=0pt,minimum size=6pt] (zh) at (0,0) {};
    \node[fill,circle,inner sep=0pt,minimum size=6pt] (a2) at (0,1) {};
    \node[fill,circle,inner sep=0pt,minimum size=6pt] (a1) at (-1,1) {};
    \node[fill,circle,inner sep=0pt,minimum size=6pt] (a3) at (1,1) {};
    \node[fill,circle,inner sep=0pt,minimum size=6pt] (b2) at (0,2) {};
    \node[fill,circle,inner sep=0pt,minimum size=6pt] (c2) at (0,3) {};
    \node[fill,circle,inner sep=0pt,minimum size=6pt] (b1) at (-1,2) {};
    \node[fill,circle,inner sep=0pt,minimum size=6pt] (b3) at (1,2) {};
    \node[fill,circle,inner sep=0pt,minimum size=6pt] (c1) at (-1,3) {};
    \node[fill,circle,inner sep=0pt,minimum size=6pt] (oneh) at (0,4) {};

    \node (1) at (-0.35,1.1) {\textcolor{blue}{$1$}};
    \node (6) at (0.35,1.1) {\textcolor{blue}{$6$}};
    \node (2) at (0.65,1.1) {\textcolor{blue}{$2$}};
    \node (4) at (0.53,2.02) {\textcolor{blue}{$4$}};

    \node (spacer) at (1.5,0) {};
    \node (spacer) at (-1.5,0) {};
    
    \draw (zh) --node[below left,blue] {$1$} (a1) --node[ left,blue] {$2$} (b1) --node[left,blue] {$6$} (c1) --node[above left,blue] {$5$} (oneh);
    \draw (a1) --node[ left,blue] {$3$} (b2) --node[below left,blue] {$1$} (c2) --node[right,blue] {$4$} (oneh);
    \draw (zh) --node[ left,blue] {$2$} (a2) -- (b1) --node[left,blue] {$3$} (c2);
    \draw (zh) --node[below right,blue] {$3$} (a3) --node[right, blue] {$1$} (b3) --node[above right,blue] {$5$} (c2);
    \draw (a2) -- (b3);
    \draw (a3) -- (b2);
    \draw (b3) -- (c1);
   
    \end{tikzpicture}
    }}
    \hspace{5mm}
    \subfigure[Consequent $\cord{P}{\lm}$.]{
    \scalebox{1}{ 
    \begin{tikzpicture}[very thick]
    \node[blue] (1234) at (0,0) {$1234$};
    \node[blue] (1314) at (1.5,1) {$1314$};
    \node[blue] (2134) at (-0.5,1) {$2134$};
    \node[blue] (1265) at (-1.5,1) {$1265$};
    \node[blue] (2165) at (-1,2) {$2165$};
    \node[blue] (2645) at (-1,3) {$2645$};
    \node[blue] (2654) at (-1.5,4) {$2654$};
    \node[blue] (3145) at (-0.5,4) {$3145$};
    \node[blue] (3154) at (-1,5) {$3154$};
    \node[blue] (3214) at (0,6) {$3214$};
    
    \draw (1234) -- (1314) -- (3214);
    \draw (1234) -- (1265) -- (2165) -- (2645) -- (2654) -- (3154) -- (3214);
    \draw (1234) -- (2134) -- (2165);
    \draw (2645) -- (3145) -- (3154);

    \end{tikzpicture}
    }}
    \caption{A poset $P$ with {\EL} $\lm$ and its maximal chain descent order.}
    \label{fig:overlapping polygons case lower len two int}
\end{figure}
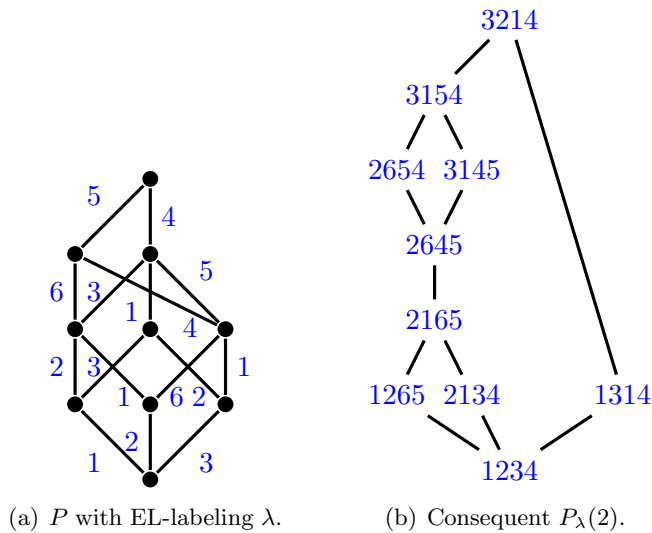

\end{example}

\end{subsection}

\begin{subsection}{Equivalence of Linear Extensions of \texorpdfstring{$\mathbf{\cord{P}{\boldsymbol{\lm}}}$ and Shellings Induced by $\boldsymbol{\lm}$}{Lg}}
Now we turn to how maximal chain descent orders encode the shellings induced by $\lm$ and prove \cref{thm:introduction shelling order with descent restriction iff linext of cord}. The following lemma says that any linear extension of a maximal chain descent order gives a shelling of the order complex of the original poset. Bj{\"o}rner and Wachs' original proofs that {\EL}s and {\CL}s induce shellings essentially go through with a modification using properties of maximal chain descent orders which we showed previously. We emphasize the modification in the proof below.

\begin{lemma}\label{lem:linextsgiveshellings}
Let $P$ be a finite, bounded poset which admits a CL-labeling $\lambda$. 
Then any linear extension of $\cord{P}{\lambda}$ gives a shelling order of the order complex $\Delta(P)$ and the restriction map of any such shelling is given by  $R(m)=\sett{x\in m}{\text{x is a descent of $m$ w.r.t. $\lm$}}$ for any maximal chain $m\in \M(P)$. The homology facets of the shelling of $\Delta(\overline{P})$ induced by any linear extension of $\cord{P}{\lm}$ are given by the maximal chains $m\in \M(P)$ with descending label sequence with respect to $\lm$.
\end{lemma}

\begin{proof}
Let $m_1,\dots, m_t$ be a linear extension of $\cord{P}{\lm}$. Consider $m_i\cap m_j$ for $i<j$. We will show that $m_j$ has a descent at some $x\in m_j$ such that $x\not \in m_i$. Then the fact that this is a shelling order follows simply from the proof of Theorem 5.8 in \cite{shllblnonpureIbjornerwachs1996}, for instance, and the definition of $\cord{P}{\lm}$. 

We consider the maximal intervals on which $m_i$ and $m_j$ differ. That is, let $y_0<y_1<y_2<\dots<y_k$ be all the elements of $m_j$ such that $y_l\in m_i\cap m_j$ for each $0\leq l\leq k$ while $m_j\cap(y_l,y_{l+1})$ and $m_i\cap(y_l,y_{l+1})$ are non-empty and disjoint for each $0\leq l\leq k-1$. If $m_j$ is ascending with respect to $\lm$ on each rooted interval $[y_l,y_{l+1}]_{m^{y_l}_j}$, then $\cl{m_j}{m_i}{\lm}$ by \cref{cor:order rel from ascending on differing intervals}. However, this contradicts the fact that $m_1,\dots, m_t$ is a linear extension of $\cord{P}{\lm}$ and $i<j$. Thus, $m_j$ has a descent at some $x\in m_j$ such that $x\not \in m_i$.

As for the restriction map, $R$ as defined above gives the necessary partition of $\Delta(P)$ by \cref{thm:clshell} and \cref{prop:restriction version of shelling}. Now we verify the restriction set containment condition of \cref{prop:restriction version of shelling}. Assume $R(m_i)\subseteq m_j$ for some $i\neq j$. Assume that the maximal intervals on which $m_i$ and $m_j$ differ are the same as above. Then since $R(m_i)$ is exactly the set of elements of $m_i$ at which $m_i$ has descents and $R(m_i)\subseteq m_j$, $m_i\cap[y_l,y_{l+1}]$ is ascending with respect to $\lm$ and the root $m_i^{y_l}$ for each $1\leq l\leq k-1$. Thus, $\cl{m_i}{m_j}{\lm}$ by \cref{cor:order rel from ascending on differing intervals}. This implies $i<j$ because $m_1,m_2,\dots, m_t$ is a linear extension of $\cord{P}{\lm}$. Therefore, $R$ is the restriction map of this shelling. 

Lastly, a maximal chain $m\in \M(P)$ has $m\setminus \set{\zh,\oneh}$ a homology facet of the shelling of $\Delta(\overline{P})$ induced by $m_1,m_2,\dots,m_t$ if and only if $R(m)=m\setminus \set{\zh,\oneh}$. Thus, $m\setminus \set{\zh,\oneh}$ is a homology facet if and only if $m$ has a descent with respect to $\lm$ at each element of $m\setminus \set{\zh,\oneh}$.
\end{proof}

\begin{remark}\label{rmk:lin ext shellings give more shellings than lex}
It is clear from \cref{cor:nongradedcordgreaterlexgreater} that the lexicographic shellings in the sense Bj{\"o}rner and Wachs (see \cref{sec:bkgrlexshell}) are among the linear extension shellings of \cref{lem:linextsgiveshellings}. On the other hand, one linear extension of the maximal chain descent order in \cref{ex:cord which is not lex order} is $123,213,132$. This gives a shelling order of $\Delta(P)$ by \cref{lem:linextsgiveshellings}. However, this total order is not compatible with the lexicographic order induced by the labeling. This shows that these linear extension shellings can give strictly more shelling orders of $\Delta(P)$ than the lexicographic ones. This can also be observed by applying \cref{lem:linextsgiveshellings} to the Boolean lattice with its standard {\EL} as in \cref{sec:motiv example boolean lat weak Sn} and any linear extension of the weak order on the symmetric group except for the lexicographic order on permutations.
\end{remark}

\begin{remark}\label{rmk:type A coxeter complex shellings}
We observe that \cref{lem:linextsgiveshellings} applied to the Boolean lattice and its {\EL} from \cref{sec:motiv example boolean lat weak Sn} recovers a special case of a result due to Bj{\"o}rner in \cite{coxetercomplexesbjorner1984}. Namely, any linear extension of weak order on a Coxeter group induces a shelling of the corresponding Coxeter complex. In the case of type A, the Coxeter complex is the order complex of the proper part of the Boolean lattice, so \cref{thm:boolean lat cord weak order symm grp} and \cref{lem:linextsgiveshellings} give the type A case of Bj{\"o}rner's result.
\end{remark}

Another consequence of \cref{lem:linextsgiveshellings} is that we can detect some homology facets (all the homology facets if the labeling is polygon complete) from the number of downward cover relations from a given element in a maximal chain descent order. 

\begin{corollary}\label{cor:homology facets from downward covers}
Let $P$ be a finite, bounded poset which admits a CL-labeling $\lambda$. Let $\Omega$ be any linear extension of $\cord{P}{\lm}$. Let $m\in \M(P)$ be a maximal chain of length $n$. If the number of maximal chains $m'\in \M(P)$ such that $\cordoti{m'}{m}{\lm}$ is $n-1$, then $m\setminus \set{\zh,\oneh}$ is a homology facet of the shelling of $\Delta(\overline{P})$ induced by $\Omega$. Moreover, if $\lm$ is polygon complete, then $m\setminus \set{\zh,\oneh}$ is a homology facet of the shelling of $\Delta(\overline{P})$ induced by $\Omega$ if and only if the number of downward cover relations from $m$ in $\cord{P}{\lm}$ is $n-1$.
\end{corollary}

\begin{proof}
This follows from \cref{lem:number downward cord covers} and the homology facet statement of \cref{lem:linextsgiveshellings}.
\end{proof}

Conversely, we have the following lemma which implies that, in a precise sense, maximal chain descent orders encode all shellings ``derived from" a {\CL}. This lemma also reinforces the name ``maximal chain descent order." 

\begin{lemma}\label{lem:converse to lin ext shellings}
 Let $P$ be a finite, bounded poset with a {\CL} $\lm$. Suppose a total order $m_1,m_2,\dots,m_t$ on the maximal chains of $P$ induces a shelling order of the order complex $\Delta(P)$ with the property that for each $1\leq i\leq t$ the restriction face $R(m_i)$ of $m_i$ is precisely the face $R(m_i)=\sett{x\in m_i}{\text{x is a descent of $m_i$ w.r.t. $\lm$}}$. Then $m_1,m_2,\dots,m_t$ is a linear extension of $\cord{P}{\lm}$.
\end{lemma} 

\begin{proof}
It suffices to show that if $m_i\to m_j$ with respect to $\lm$, then $i<j$. Let $y$ be the unique element of $m_j$ such that $m_j\setminus m_i=\set{y}$. Let $x$ and $z$ be the elements of $m_j$ satisfying $x\lessdot y\lessdot z$. By definition of an increase by a polygon move (\cref{def:increase and polygon}), $y$ is a descent of $m_j$ and $m_i\cap [x,z]$ is ascending with respect to $\lm$ and the root $m_i^x$. Thus, $R(m_i)\subseteq m_j$. This implies $i<j$ by \cref{prop:restriction version of shelling} since $R$ is the restriction map of the shelling $m_1,m_2,\dots,m_t$.
\end{proof}

Combined, the previous two lemmas prove \cref{thm:introduction shelling order with descent restriction iff linext of cord}.

\begin{theorem}\label{thm:shelling order with descent restriction iff linext of cord}
Let $P$ be a finite, bounded poset with a {\CL} $\lm$. For any total order $\Omega:m_1,m_2,\dots, m_t$ on the maximal chains of $P$, the following are equivalent:
\begin{itemize}
    \item [(1)] $\Omega$ is a linear extension of $\cord{P}{\lm}$.
    \item [(2)] $\Omega$ induces a shelling order of the order complex $\Delta(P)$ with the property that for each $1\leq i\leq t$ the restriction face $R(m_i)$ of $m_i$ is precisely the face \\ $R(m_i)=\sett{x\in m_i}{\text{x is a descent of $m_i$ w.r.t. $\lm$}}$.
\end{itemize}
\end{theorem}

\begin{proof}
Statement (1) implies statement (2) by \cref{lem:linextsgiveshellings} and statement (2) implies statement (1) by \cref{lem:converse to lin ext shellings}.
\end{proof}

Since {\EL}s are instances of {\CL}s, the previous theorem holds for {\EL}s as well. We state this point as a corollary for emphasis.

\begin{corollary}\label{cor:EL restriction shelling iff lin ext cord}
Let $P$ be a finite, bounded poset with an {\EL} $\lm$. For any total order $\Omega:m_1,m_2,\dots, m_t$ on the maximal chains of $P$, the following are equivalent:
\begin{itemize}
    \item [(1)] $\Omega$ is a linear extension of $\cord{P}{\lm}$.
    \item [(2)] $\Omega$ induces a shelling order of the order complex $\Delta(P)$ with the property that for each $1\leq i\leq t$ the restriction face $R(m_i)$ of $m_i$ is precisely the face \\ $R(m_i)=\sett{x\in m_i}{\text{x is a descent of $m_i$ w.r.t $\lm$}}$.
\end{itemize}
\end{corollary}

We also immediately have that if two different {\CL}s of the same poset give the same descent set for each maximal chain, then they induce the same maximal chain descent order.

\begin{corollary}\label{cor:cl with same descent set isom EL}
Let $P$ be a finite, bounded poset with two possibly different {\CL}s $\lm$ and $\lm'$. Suppose that for each $m\in \M(P)$ and each $x\in m$, $m$ has descent at $x$ with respect to $\lm$ if and only if $m$ has descent at $x$ with respect to $\lm'$. Then $\cord{P}{\lm}=\cord{P}{\lm'}$.
\end{corollary}

\begin{proof}
Both $\cord{P}{\lm}$ and $\cord{P}{\lm'}$ are posets on $\M(P)$. By \cref{thm:shelling order with descent restriction iff linext of cord} $\cord{P}{\lm}$ and $\cord{P}{\lm'}$ both have exactly the same set of linear extensions. A poset can be constructed simply from its set of linear extensions, so $\cord{P}{\lm}=\cord{P}{\lm'}$.
\end{proof}

Next we turn to the topic of cover relations in maximal chain descent orders. 

\end{subsection}

\section{Understanding Cover Relations via a Characterization of Polygon Complete {\EL}s}\label{sec:polygon completeness}

We characterize polygon completeness for {\EL}s (see \cref{def:polygon complete}) with two technical conditions below in \cref{thm:characterization of non-polygon completeness}. However, we begin with a simpler concrete sufficient condition for polygon completeness of {\EL}s in \cref{def:polygonstrongEL}. The proof of this sufficient condition provides us the opportunity to get a taste for some of the proof techniques we will use for \cref{thm:characterization of non-polygon completeness}, but in a more constrained context where we may be less delicate.

\begin{subsection}{Polygon Strong Implies Polygon Complete}\label{sec:polygon strong implies polygon complete}

We begin with several necessary lemmas some of which are possibly interesting in their own right. Then we present the definition of a polygon strong {\EL} before we show that it is sufficient for polygon completeness. This first lemma is quite straightforward and shows that any {\EL} of a poset of rank one or two is polygon complete. It provides the base cases for later induction arguments.

\begin{lemma}\label{lem:rank two polygon complete}
Let $P$ be a finite, bounded poset in which the length of the longest maximal chain is either one or two. Let $\lm$ be an {\CL} of $P$. Then $\lm$ is polygon complete.
\end{lemma}

\begin{proof}
If the length of the longest maximal chain of $P$ is one, then $P$ has exactly one maximal chain and $\lm$ is vacuously polygon complete because there are no polygons to check.

Suppose the longest maximal chain of $P$ has length two. Then every maximal chain of $P$ has length two. Thus, every maximal chain of $P$, except for the unique ascending chain, is a descent. Thus, the increases by polygon moves of $P$ with respect to $\lm$ are of the form $m_0\to m$ where $m_0$ is the unique ascending chain of $P$ and $m\neq m_0$ is any other maximal chain of $P$. Hence, every maximal chain of $P$, except for $m_0$, is a maximal element of $\cord{P}{\lm}$. Therefore, $\cord{P}{\lm}$ is ranked with rank one and $\lm$ is polygon complete.
\end{proof}

The following lemma is quite useful for working with maximal chain descent orders, particularly for proofs by induction on chain length. Intuitively, the lemma says that if two maximal chains $m$ and $m'$ agree along an initial segment and are comparable in a maximal chain descent order, then each maximal chain between them in the maximal chain descent order agrees with $m$ and $m'$ on that same initial segment. The example in \cref{fig:noncover diamond move example} shows, among other things, that the mirrored statement for final segments is not true.

\begin{lemma}\label{lem:nongradedonlyabovepolygon}
Let $P$ be a finite, bounded poset with a {\CL} $\lm$. Let $m$ and $m'$ be maximal chains $m:y_0=\zh \lessdot y_1 \lessdot \dots \lessdot y_{i-1} \lessdot y_i \lessdot y_{i+1} \lessdot \dots \lessdot y_{t-1} \lessdot y_t=\oneh$ and $m':y_0=\zh \lessdot y_1 \lessdot \dots \lessdot y_{i-1} \lessdot y'_i \lessdot y'_{i+1} \lessdot \dots \lessdot y'_{t'-1} \lessdot y_t'=\oneh$. Suppose $m=m_0 \to m_1\to m_2 \to \dots \to m_{k}\to m_{k+1} = m'$. Then $m_j^{y_{i-1}} = m^{y_{i-1}} =m'^{y_{i-1}}$ for each $1\leq j\leq k$.

\end{lemma}

\begin{proof}
Suppose seeking a contradiction that there is some $m_j$ such that $m_j^{y_{i-1}} \neq m^{y_{i-1}}$. We observe that for any of the maximal chains in the sequence to disagree with $m$ at or below $y_{i-1}$, there must be some $m_l$ such that the polygon corresponding to $m_l\to m_{l+1}$ has bottom element strictly less than $y_{i-1}$. Let $m_{j'}\to m_{j'+1}$ be the first increase by a polygon move in the sequence $m=m_0 \to m_1\to m_2 \to \dots \to m_{k}\to m_{k+1} = m'$ with corresponding polygon whose bottom element is strictly less than $y_{i-1}$. We must have $0\leq j'\leq k$. Then $\lex{\lm(m_{j'})}{\lm(m_{j'+1})}$ by \cref{prop:nongradedmaxchainincreaselexincrease}. The first entry at which the label sequences of $m_{j'}$ and $m_{j'+1}$ differ comes before the $(i-1)$th position. However, $m^{y_{i-1}} =m'^{y_{i-1}} = m_{j'}^{y_{i-1}}$ since $m_{j'}\to m_{j'+1}$ is the first increase by a polygon move with corresponding polygon whose bottom element is below $y_{i-1}$. So, the label sequences of $m_{j'}$ and $m'$ agree in their first $i-1$ entries. Thus, $\lex{\lm(m')}{\lm(m_{j'+1})}$ which contradicts \cref{cor:nongradedcordgreaterlexgreater} since $\cleqi{m_{j'+1}}{m'}{\lm}$. Hence, $m_j^{y_{i-1}} = m^{y_{i-1}} =m'^{y_{i-1}}$ for each $1\leq j\leq k$.
\end{proof}

A consequence of \cref{lem:restrictedrelslift} and \cref{lem:nongradedonlyabovepolygon} is that in a maximal chain descent order induced by an EL-labeling, whether or not an increase by a polygon move gives a cover relation only depends on what happens in the poset above the bottom of the corresponding polygon. In other words, whether or not an increase by a polygon move with respect to an {\EL} gives a cover relation does not depend on the root used to get to the bottom of the polygon. Besides being useful for later arguments, it seems this fact may be of interest in its own right.

\begin{corollary}\label{cor:EL cord covers don't depend on root}
Let $P$ be a finite, bounded poset which admits an {\EL} $\lm$. Let $m,m'\in \M(P)$ be maximal chains of $P$ with $$m:x_0=\zh\lessdot x_1\lessdot \dots\lessdot x_{i-1} \lessdot \dots \lessdot x_{i+l} \lessdot x_{r-1} \lessdot x_r=\oneh$$ and $$m':x_0=\zh\lessdot x_1\lessdot \dots \lessdot x_{i-1} \lessdot x'_i \lessdot x_{i+l}\lessdot \dots \lessdot x_{r-1} \lessdot x_{r}=\oneh$$ such that $m\to m'$. Suppose $\cordot{m}{_{\lm} m'}$ in $\cord{P}{\lm}$. Then for any maximal chain $c$ of $[\zh,x_{i-1}]$, we have $\cordot{c * m_{x_{i-1}}}{_{\lm} c * m'_{x_{i-1}}}$ in $\cord{P}{\lm}$.
\end{corollary}

\begin{proof}
Since $\lm$ is an {\EL}, the labels of $c * m_{x_{i-1}}$ and $c * m'_{x_{i-1}}$ agree with the labels of $m$ and $m'$, respectively, above $x_{i-1}$. Thus, $c * m_{x_{i-1}} \to c * m'_{x_{i-1}}$. Now suppose seeking a contradiction that $c * m_{x_{i-1}}$ is not covered by $c * m'_{x_{i-1}}$ in $\cord{P}{\lm}$. Then there are maximal chains $m_1,\dots, m_k\in \M(P)$ with $k\geq 1$ such that $c * m_{x_{i-1}} \to m_1\to \dots \to m_k\to c * m'_{x_{i-1}}$. Then by \cref{lem:nongradedonlyabovepolygon} $m_j^{x_{i-1}} = c$ for each $1\leq j\leq k$. Thus, $m_{x_{i-1}} \to {m_1}_{x_{i-1}} \to \dots \to {m_k}_{x_{i-1}} \to m'_{x_{i-1}}$ in $[{x_{i-1}},\oneh]$ with respect to $\lm$. Then by \cref{lem:restrictedrelslift} $\cl{m^{x_{i-1}} * m_{x_{i-1}}}{\cl{m^{x_{i-1}} * {m_1}_{x_{i-1}}}{m^{x_{i-1}} * m'_{x_{i-1}}}{\lm}}{\lm}$ in $\cord{P}{\lm}$. However, since $m^{x_{i-1}}=m'^{x_{i-1}}$, we then have $\cl{m}{\cl{m^{x_{i-1}} * {m_1}_{x_{i-1}}}{m'}{\lm}}{\lm}$ which contradicts that $\cordot{m}{_{\lm} m'}$ in $\cord{P}{\lm}$. Therefore, $\cordot{c * m_{x_{i-1}}}{_{\lm} c * m'_{x_{i-1}}}$ in $\cord{P}{\lm}$.
\end{proof}

Another consequence of \cref{lem:nongradedonlyabovepolygon} is that whenever $\oneh$ is the top element of the polygon corresponding to an increase by a polygon move, that polygon move gives a cover relation in the maximal chain descent order.  

\begin{corollary}\label{lem:top polygons give cord covers}
Let $P$ be a finite, bounded poset with a {\CL} $\lm$. Suppose $m\to m'$ such that $m'\setminus m = \set{x}$ with $x\lessdot \oneh$. Then $\cordoti{m}{m'}{\lm}$ in $\cord{P}{\lm}$.
\end{corollary}

\begin{proof}
This follows immediately from \cref{prop:descents give unique incrs} and \cref{lem:nongradedonlyabovepolygon}.
\end{proof}

Now we turn to a concrete sufficient condition for polygon completeness of {\EL}s which we call polygon strong. Polygon strong is simpler and easier to verify than the conditions in \cref{thm:characterization of non-polygon completeness}. Polygon strong is a weakening of Bj{\"o}rner's notion of strongly lexicographically shellable from \cite{shellblposets} which arose from studying admissible lattices. The proof of this condition also allows us to encounter some proof techniques which we will use to prove \cref{thm:characterization of non-polygon completeness}.

\begin{definition}\label{def:polygonstrongEL}
Let $P$ be a finite, bounded poset with an EL-labeling $\lm$. We say $\lm$ is a \textbf{polygon strong EL-labeling} if for each descent $x\lessdot y \lessdot z$, $\lm(y \lessdot z)< \lm(y'\lessdot z)$ where $y'$ is the coatom of $[x,z]$ contained in the unique ascending maximal chain of $[x,z]$ with respect to $\lm$.
\end{definition}

The {\EL} pictured in \cref{fig:example of polygon strong EL} is polygon strong.

\begin{figure}[H]
    \centering
    \scalebox{1}{
    \begin{tikzpicture}[very thick]
    \node[fill,circle,inner sep=0pt,minimum size=6pt] (zh) at (0,0) {};
    \node[fill,circle,inner sep=0pt,minimum size=6pt] (a) at (-1,1) {};
    \node[fill,circle,inner sep=0pt,minimum size=6pt] (b) at (0,1) {};
    \node[fill,circle,inner sep=0pt,minimum size=6pt] (c) at (1,1) {};
    \node[fill,circle,inner sep=0pt,minimum size=6pt] (oneh) at (0,2) {};
    
    \draw (zh) --node[below left,blue]{$1$} (a) --node[above left,blue]{$2$} (oneh);
    \draw (zh) --node[left,blue]{$2$} (b) --node[left,blue]{$1$} (oneh);
    \draw (zh) --node[below right,blue]{$3$} (c) --node[above right,blue]{$1$} (oneh);
    \end{tikzpicture}
    }
    \caption{A polygon strong {\EL}.
    }
    \label{fig:example of polygon strong EL}
\end{figure}
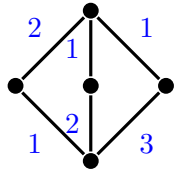

\begin{theorem}\label{thm:polygonstrongELincreasesgivecovers}
Let $P$ be a finite, bounded poset with a polygon strong {\EL} $\lm$. Then $\lm$ is polygon complete.
\end{theorem}

\begin{proof}
We first observe that the restriction of a polygon strong {\EL} to any closed interval $[x,y]$ of $P$ is also a polygon strong {\EL}. This is straightforward because any interval of $[x,y]$ is also an interval of $P$. Next we observe that if $m\to m'$ for maximal chains $m$ and $m'$ with $x\lessdot \oneh$ contained in $m$ and $x'\lessdot \oneh$ contained in $m'$, then $\lm(x, \oneh)\geq \lm(x', \oneh)$ since $\lm$ is polygon strong.

We will proceed by induction on the length of the longest maximal chain of $P$. We have that $\lm$ is polygon complete if the length of the longest maximal chain of $P$ is one or two by \cref{lem:rank two polygon complete}. Assume $\lm$ is polygon complete whenever the length of the longest maximal chain of $P$ is any $k$ with $1\leq k\leq r$ for some $r\geq 2$. Let $P$ have longest maximal chain of length $r+1$. Assume $m\to m'$. Since $m$ and $m'$ differ by a polygon, $m:\zh=x_0\lessdot \dots \lessdot x_{i-1} \lessdot x_i \lessdot \dots \lessdot x_{i+l}\lessdot \dots \lessdot x_{r+1} =\oneh$ and $m':\zh=x_0\lessdot \dots \lessdot x_{i-1} \lessdot x'_i \lessdot x_{i+l}\lessdot \dots \lessdot x_{r+1} =\oneh$ for some $l\geq 1$. So, $x_{i-1} \lessdot x_i\lessdot \dots \lessdot x_{i+l}$ is the unique ascending maximal chain of $[x_{i-1}, x_{i+l}]$ with respect to $\lm$ while $x_{i-1} \lessdot x'_i \lessdot x_{i+l}$ is a descent with respect to $\lm$. 

Now, seeking a contradiction, suppose that $\cordoti{m\not}{m'}{\lm}$. By \cref{lem:top polygons give cord covers} $i+l<r+1$, so $m^{x_{r}}\to m'^{x_{r}}$. Since $\cordoti{m\not}{m'}{\lm}$, there are maximal chains $m_1,m_2,\dots,m_s \in \M(P)$ with $s\geq 1$ such that $m\to m_1\to m_2\to \dots \to m_s\to m'$. Then by \cref{lem:nongradedonlyabovepolygon}, $m_j^{x_{i-1}}=m^{x_{i-1}}=m^{x_{i-1}}$ for all $1\leq j\leq s$. There are two cases we must consider. Either $x_r\in m_j$ for all $1\leq j\leq s$ or there is some $1\leq j\leq s$ such that $x_r\not \in m_j$.

Suppose $x_r\in m_j$ for all $1\leq j\leq s$. Then $m^{x_{r}}\to m_1^{x_{r}}\to m_2^{x_{r}}\to \dots \to m_s^{x_{r}}\to m'^{x_{r}}$. Then since $s\geq 1$, $\cordoti{m^{x_{r}} \not}{m'^{x_{r}}}{\lm}$ in $\cord{[\zh,x_{r}]}{\lm}$ despite the fact that $m^{x_{r}}\to m'^{x_{r}}$. However, this contradicts the fact that $\lm$ restricted to $[\zh,x_{r}]$ is polygon complete by the inductive hypothesis because the length of the longest maximal chain in $[\zh,x_{r}]$ is at most $r$ and $\lm$ restricted to $[\zh,x_{r}]$ is polygon strong. 

Thus, there is a first $m_t$ for $0\leq t\leq s$ (say $m_0=m$) such that $x_{r}\in m_t$, but $x_{r}\not \in m_{t+1}$. Let $z_{t+1}\lessdot \oneh$ be contained in $m_{t+1}$. Then $\lm(z_{t+1},\oneh)<\lm(x_{r},\oneh)$ since $\lm$ is polygon strong and $x_r\lessdot \oneh$ is contained in $m_t$ while $x_r\not \in m_{t+1}$. Now let $z_{j}\lessdot \oneh$ be contained in $m_j$. Since $\lm$ is polygon strong, we have $$ \lm(x_{r},\oneh)= \lm(z_{1},\oneh)= \lm(z_{2},\oneh)= \dots = \lm(z_{t},\oneh)> \lm(z_{t+1},\oneh)\geq \dots \geq \lm(z_{k},\oneh)\geq \lm(x_{r},\oneh)$$ by our second observation in the first paragraph. But this implies the contradiction that $\lm(x_{r},\oneh) > \lm(x_{r},\oneh)$. Therefore, $\lm$ is polygon complete, so the theorem holds by induction.
\end{proof}

\end{subsection}

\begin{subsection}{Applications of \texorpdfstring{\cref{thm:polygonstrongELincreasesgivecovers}}{Lg}: Examples of Polygon Strong {\EL}s}\label{sec:examples of polygon strong ELs}

Here we prove that many well known {\EL}s are polygon strong, and so polygon complete. This section is not self contained, but we provide references and, when feasible, brief explanations. We expect that many more {\EL}s are polygon strong. We begin with the examples where \cref{thm:polygonstrongELincreasesgivecovers} applies most naturally which are the $M$-chain {\EL}s of finite supersolvable lattices due to Stanley.

\begin{theorem}\label{thm:Sn EL is polygon strong}
Stanley's $M$-chain {\EL}s of any finite supersolvable lattice from \cite{stanleysupersolvablelats1972} are polygon strong. Thus, these {\EL}s are polygon complete.
\end{theorem}

\begin{proof}
First, we observe that the label sequences of an $M$-chain {\EL} $\lm$ are all permutations of $[n]$ where $n$ is the rank of the supersolvable lattice. Thus, in any rank two interval of the lattice, the label sequence of each maximal chain is either $a,b$ with $a<b$ if the chain is the unique ascending chain with respect to $\lm$ or $b,a$ otherwise. Thus, $\lm$ is polygon strong, and so polygon complete by \cref{thm:polygonstrongELincreasesgivecovers}.  
\end{proof}

Finite distributive lattices are supersolvable and their linear extension {\EL} labelings considered later in \cref{sec:findistlats} are exactly the $M$-chain {\EL}s. Thus, the linear extension {\EL}s are polygon strong as well.

\begin{corollary}\label{cor:dist lats lin ext els polygon strong}
The linear extension {\EL}s of finite distributive lattices considered in \cref{sec:findistlats} are polygon strong. Thus, these {\EL}s are polygon complete.
\end{corollary}

\begin{remark}\label{rmk:all Sn ELs polygon strong}
In fact, our proof also works in the more general context of McNamara's $S_n$ {\EL}s of finite posets from \cite{snelsupersolvmcnamara2003}. This is because our proof only relied on all label sequences being permutations of $n$. Thus, any $S_n$ {\EL} of a finite poset is polygon strong, and so polygon complete. As a note, McNamara showed that a finite lattice admitting an $S_n$ {\EL} is equivalent to that lattice being supersolvable. 
\end{remark}

We record this useful fact about the label sequences of rank two intervals from {\EL}s in which the label sequences of maximal chains are permutations because we will use it in other places.

\begin{proposition}\label{prop:Sn descent swaps}
Let $P$ be a finite, bounded poset with an $S_n$ {\EL} $\lm$. Suppose $m' \in \M(P)$ has a descent at $x\in m'$ with $\text{rk}(x)=i$. Then the unique $m\in \M(P)$ such that $m\to m'$ and $m'\setminus m=\set{x}$ guaranteed by \cref{prop:descents give unique incrs} has $\lm(m')=\lm(m)(i,i+1)$. 
\end{proposition}

\begin{proof}
Let $\lm(m)=\lm_1\lm_2\dots \lm_i \lm_{i+1} \dots \lm_n$. Then since $m'\setminus m=\set{x}$, $\lm(m')= \lm_1\lm_2\dots \lm'_i \lm'_{i+1} \dots \lm_n$. Then since $\lm$ is an $S_n$ {\EL}, $\set{\lm_i, \lm_{i+1}}= \set{\lm'_i, \lm'_{i+1}}$. Lastly, since $m$ has a descent at $x$ and $m'\to m$, $\lm_i>\lm_{i+1}$ and $\lm'_i< \lm'_{i+1}$. Hence, $\lm(m')=\lm(m)(i,i+1)$.
\end{proof}

Another generalization of \cref{thm:Sn EL is polygon strong} is the following result for similar {\EL}s of upper-semimodular and lower-semimodular lattices.

\begin{theorem}\label{thm:upper and lowwer semimod latts polygon strong}
Let $L$ be a finite upper-semimodular or lower-semimodular lattice. Let $\lm$ be an {\EL} of $L$ induced by an admissible map on $L$ as in \cite{shellblposets}. Then $\lm$ is polygon strong. Thus, $\lm$ is polygon complete.
\end{theorem}

\begin{proof}
In Proposition 3.6 \cite{shellblposets}, it is shown that the labeling $\lm$ due to Stanley is an {\EL}. In Theorem 3.7 \cite{shellblposets}, $\lm$ is shown to be an SL-labeling (strongly lexicographic) in the sense of Definition 3.4 \cite{shellblposets}. When restricted to intervals of length two in $L$, the defining condition of an SL-labeling is precisely the defining condition of a polygon strong {\EL}. Then applying \cref{thm:polygonstrongELincreasesgivecovers} completes the proof.
\end{proof}

Next we turn to the case of finite geometric lattices. First, we briefly recall the definition of a minimal labeling of a finite geometric lattice. Minimal labelings are a class of {\EL}s introduced by B{\"o}rner in \cite{shellblposets} and were shown to characterize finite geometric lattices by Davidson and Hersh in \cite{geomlatsdavidsonhersh2014}. (We refer to those citations for the well known definition of a geometric lattice and more in depth discussion of minimal labelings.) For an element $x$ in a geometric lattice $L$, we denote the set of atoms of $L$ which are below $x$ by $A(x)$. Let $\Omega$ be any total ordering of the atoms of $L$. Then the \textbf{minimal labeling induced by $\boldsymbol{\Omega}$} is the edge labeling $\lm_{\Omega}$ of $L$ given as follows: if $x\lessdot y$, then $\lm_{\Omega}(x,y)=\min_{\Omega}(A(y)\setminus A(x))$. 

\begin{theorem}\label{thm:min labelings polygon strong}
Every minimal labeling $\lm$ of a finite geometric lattice is polygon strong. Thus, $\lm$ is polygon complete. 
\end{theorem}

\begin{proof}
Let $x\lessdot y\lessdot z$ be an ascending saturated chain in a geometric lattice $L$ with respect to a minimal labeling $\lm_{\Omega}$ induced by a total atom order $\Omega$. Then $\lm_{\Omega}(x,y)$ is the minimal atom with respect to $\Omega$ which is below $z$, but not below $x$. This implies that for any $y'\neq y$ satisfying $x\lessdot y'\lessdot z$, $\lm_{\Omega}(y',z)=\lm_{\Omega}(x,y)<\lm_{\Omega}(y,z)$. Thus, $\lm_{\Omega}$ is polygon strong, and so polygon complete by \cref{thm:polygonstrongELincreasesgivecovers}.
\end{proof}

\begin{example}\label{ex:Pi4 min lableing}
\cref{fig:Pi4 with min labeling} shows the partition lattice $\Pi_4$, which is a geometric lattice, with a minimal labeling $\lm_{\Omega}$. The total atom order $\Omega$ is the one induced by the labels of the covers below the atoms. We label covers by the index of the corresponding atom with respect to $\Omega$. \cref{fig:cord from Pi4 min labeling} exhibits the induced maximal chain descent order $\cord{{\Pi_4}}{\lm_{\Omega}}$. The maximal chains of $\Pi_4$ 

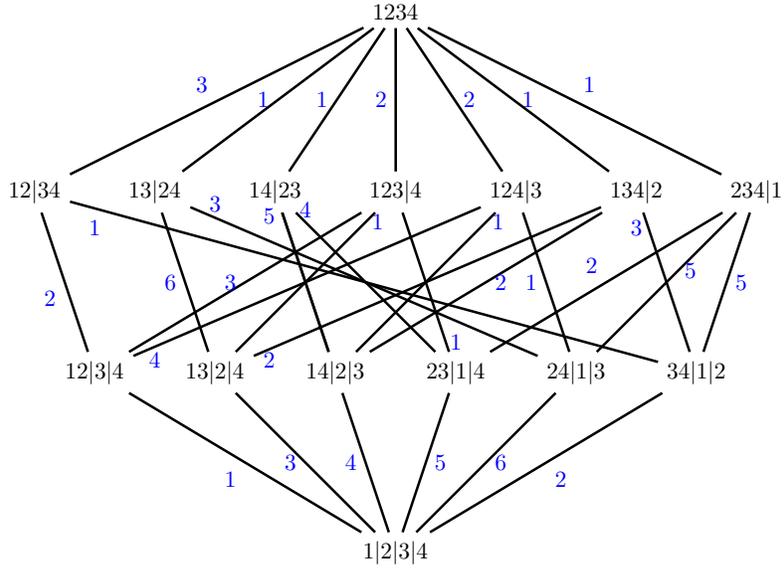
\begin{figure}[H]
    \centering
    \scalebox{0.8}{\begin{tikzpicture}[very thick]
  \node (zh) at (0,0) {\scalebox{1}{$1|2|3|4$}};
  \node (12) at (-5,3) {\scalebox{1}{$12|3|4$}};  
  \node (13) at (-3,3) {\scalebox{1}{$13|2|4$}};
  \node (14) at (-1,3) {\scalebox{1}{$14|2|3$}};
  \node (23) at (1,3) {\scalebox{1}{$23|1|4$}};
  \node (24) at (3,3) {\scalebox{1}{$24|1|3$}};
  \node (34) at (5,3) {\scalebox{1}{$34|1|2$}};
  \node (1234) at (-6,6) {\scalebox{1}{$12|34$}};
  \node (1324) at (-4,6) {\scalebox{1}{$13|24$}};
  \node (1423) at (-2,6) {\scalebox{1}{$14|23$}};
  \node (123) at (0,6) {\scalebox{1}{$123|4$}};
  \node (124) at (2,6) {\scalebox{1}{$124|3$}};
  \node (134) at (4,6) {\scalebox{1}{$134|2$}};
  \node (234) at (6,6) {\scalebox{1}{$234|1$}};
  \node (oneh) at (0,9) {\scalebox{1}{$1234$}};

  \node[blue] (23 1423) at (-1.5,5.7) {$4$};
  \node[blue] (14 1423) at (-2.1,5.6) {$5$};
  \node[blue] (24 1324) at (-3,5.8) {$3$};
  \node[blue] (34 1234) at (-5,5.4) {$1$};
  \node[blue] (12 124) at (-4,3.2) {$4$};
  \node[blue] (13 123) at (-0.3,5.5) {$1$};
  \node[blue] (13 134) at (-2.1,3.2) {$2$};
  \node[blue] (14 124) at (1.7,5.5) {$1$};
  \node[blue] (23 123) at (1,3.5) {$1$};
  \node[blue] (34 134) at (4,5.4) {$3$};
  \node[blue] (24 234) at (4.9,4.7) {$5$};
  
  \draw (zh) --node[below left,blue] {$1$} (12) --node[below left,blue] {$2$} (1234) --node[above left,blue] {$3$} (oneh) --node[above right,blue] {$1$} (234) --node[right,blue] {$5$} (34) --node[below right,blue] {$2$} (zh) --node[left,blue] {$3$} (13) --node[left,blue] {$6$} (1324) --node[left,blue] {$1$} (oneh) --node[right,blue] {$1$} (134) --node[right,blue] {$2$} (14) --node[left,blue] {$4$} (zh);
  \draw (zh) --node[right,blue] {$5$} (23) -- (1423) --node[left,blue] {$1$} (oneh);
  \draw (zh) --node[right,blue] {$6$} (24) --node[left,blue] {$1$} (124) --node[right,blue] {$2$} (oneh);
  \draw (24) -- (234);
  \draw (13) -- (134);
  \draw (13) -- (123);
  \draw (14) -- (1423);
  \draw (14) -- (1423);
  \draw (12) --node[left,blue] {$3$} (123)--node[left,blue] {$2$} (oneh);
  \draw (12) -- (124);
  \draw (14) -- (124);
  \draw (23) -- (123);
  \draw (23) --node[above left,blue] {$2$} (234);
  \draw (24) -- (1324);
  \draw (34) -- (1234);
  \draw (34) -- (134);
\end{tikzpicture}} 
\caption{The partition lattice $\Pi_4$ with minimal labeling $\lm_{\Omega}$.}
    \label{fig:Pi4 with min labeling}
\end{figure}

\begin{figure}[H]
    \centering
\scalebox{0.9}{\begin{tikzpicture}[very thick]
  \node (m1234) at (0,3) {$m_{1234}$}; 
  \node (m12u34) at (0,0) {$m_{12}^{34}$};
  \node (m1243) at (2,3) {$m_{1243}$};
  \node (m1324) at (-1,6) {$m_{1324}$};
  \node (m2314) at (1,6) {$m_{2314}$};
  \node (m1342) at (-5,9) {$m_{1342}$};
  \node (m13u24) at (2.5,9) {$m_{13}^{24}$};
  \node (m1423) at (5,6) {$m_{1423}$};
  \node (m2413) at (3,6) {$m_{2413}$};
  \node (m34u12) at (-2,3) {$m_{34}^{12}$};
  \node (m2341) at (1,9) {$m_{2341}$};
  \node (m23u14) at (5,12) {$m_{23}^{14}$};
  \node (m1432) at (-1,9) {$m_{1432}$};
  \node (m24u13) at (3,12) {$m_{24}^{13}$};
  \node (m3412) at (-5,6) {$m_{3412}$};
  \node (m2431) at (-3,9) {$m_{2431}$};
  \node (m3421) at (-3,6) {$m_{3421}$};
  \node (m14u23) at (5,9) {$m_{14}^{23}$};
  
  \draw (m12u34) -- (m34u12) -- (m3412) -- (m1342) -- (m1324) -- (m1234) -- (m12u34) -- (m1243) -- (m1423) -- (m14u23) -- (m23u14) -- (m2314) -- (m1234);
  \draw (m3412) -- (m1432) -- (m1423);
  \draw (m34u12) -- (m3421) -- (m2431) -- (m2413) -- (m1243);
  \draw (m3421) -- (m2341) -- (m2314);
  \draw (m1324) -- (m13u24) -- (m24u13) -- (m2413);
\end{tikzpicture} }
    \caption{$\cord{{\Pi_4}}{\lm_{\Omega}}$ induced by the minimal labeling $\lm_{\Omega}$ from \cref{fig:Pi4 with min labeling}.}
    \label{fig:cord from Pi4 min labeling}
\end{figure}
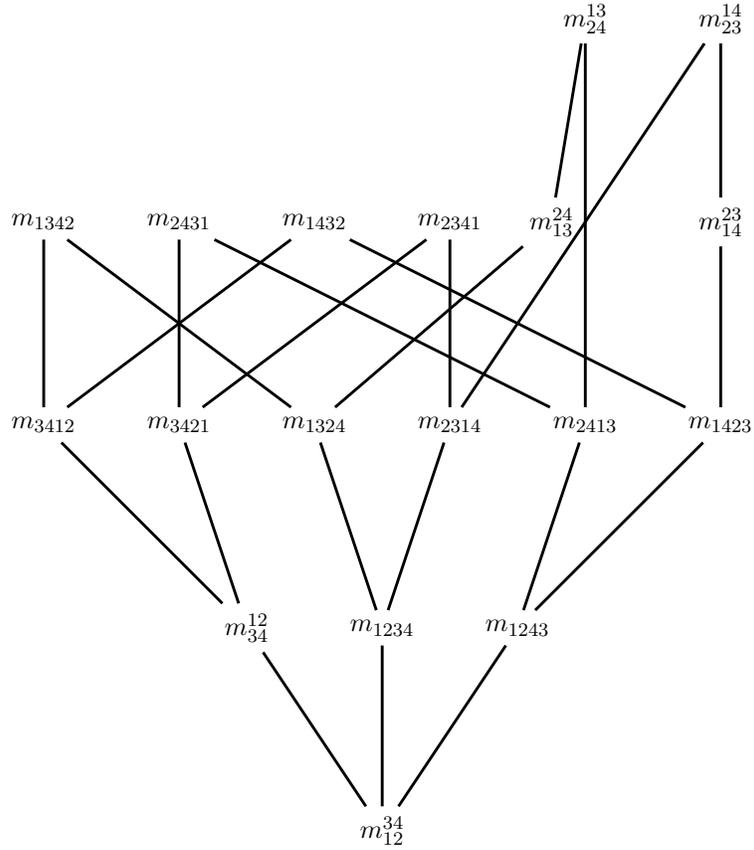
are denoted as follows: $m_{ijkl}$ denotes the chain $1|2|3|4 \lessdot ij|k|l \lessdot ijk|l \lessdot ijkl$ and $m_{ij}^{kl}$ denotes the chain $1|2|3|4 \lessdot ij|k|l \lessdot ij|kl \lessdot ijkl$. This example illustrates that $\cord{P}{\lm}$ may have multiple maximal elements and that $\cord{P}{\lm}$ need not be ranked despite $\lm$ being polygon complete.
\end{example}

We now consider certain {\EL}s due to Dyer in \cite{reflordeldyer1993} of closed intervals in Bruhat order of any Coxeter group. (See \cite{bjornerbrenitcxgp} for background on general Coxeter groups.) Let $u$ and $w$ be group elements of a Coxeter system $(W,S)$. If $u\lessdot w$ in Bruhat order on $W$, then $u^{-1}w=t$ for some reflection $t\in T$ where $T$ is the set of reflections of $(W,S)$. The cover relation $u\lessdot w$ is labeled by $\lm(u,w)=t$ and the reflections $T$ are totally ordered by any of the so called reflection orders introduced in Definition 2.1 \cite{reflordeldyer1993}. Then $\lm$ is an {\EL} of any closed interval in Bruhat order on $W$ by Section 4 of \cite{reflordeldyer1993}. We refer to these labelings as reflection order {\EL}s.

\begin{theorem}\label{thm:dyer refl order polygon strong}
Every reflection order {\EL} $\lm$ of a closed interval in Bruhat order of any Coxeter group is polygon strong. Thus, $\lm$ is polygon complete. 
\end{theorem}

\begin{proof}
The fact that $\lm$ is polygon strong follows directly from the characterization of label sequences of rank two intervals in Lemma 4.1 (i) of \cite{reflordeldyer1993}. Then we apply \cref{thm:polygonstrongELincreasesgivecovers}.
\end{proof}

Next we turn to a generalization of $S_n$ {\EL}s. In \cite{posetedgelabelsmodularitymcnamarathomas2006}, McNamara and Thomas generalized McNamara's notion of an $S_n$ {\EL} of a finite poset to the non-graded case with the notion of an interpolating {\EL}. Interpolating {\EL}s were used to study modularity. These interpolating {\EL}s turn out to be polygon strong as well.

\begin{theorem}\label{thm:interpolating ELs polygon strong}
Any interpolating {\EL} $\lm$ of a finite, bounded poset in the sense of McNamara and Thomas in \cite{posetedgelabelsmodularitymcnamarathomas2006} is polygon strong. Thus, $\lm$ is polygon complete.
\end{theorem}

\begin{proof}
Condition (ii) of the definition of an interpolating {\EL} (Definition 1.2 in \cite{posetedgelabelsmodularitymcnamarathomas2006}) allows us to use essentially the same reasoning as we used for the case of $S_n$ {\EL}s. Thus, $\lm$ is polygon strong, and so polygon complete by \cref{thm:polygonstrongELincreasesgivecovers}.   
\end{proof}

Lastly, we show that an {\EL} of the Tamari lattice and an {\EL} of intervals in general Cambrian semilattices are polygon complete.

\begin{theorem}\label{thm:bjorner and wachs EL of tamari lattice}
Bj{\"o}rner and Wachs' {\EL} $\lm$ of the Tamari lattice defined in Section 9 of \cite{shllblnonpureIIbjornerwachs1997} is polygon strong. Hence, $\lm$ is polygon complete.
\end{theorem}

\begin{proof}
The fact that $\lm$ is polygon strong follows from the proof of Theorem 9.2 in \cite{shllblnonpureIIbjornerwachs1997}. Then we apply \cref{thm:polygonstrongELincreasesgivecovers}.
\end{proof}

\begin{theorem}\label{thm:kallipolit muhle EL of cambrian lattice}
The {\EL} $\lm$ of a closed interval in any Cambrian semilattice given by Kallipoliti and M{\"u}hle in Section 3.1 of \cite{topcambriansemilat2013} is polygon strong. Hence, $\lm$ is polygon complete.
\end{theorem}

\begin{proof}
The fact that $\lm$ is polygon complete follows directly from Lemma 3.4 in \cite{topcambriansemilat2013}. Then we again apply \cref{thm:polygonstrongELincreasesgivecovers}.
\end{proof}

\end{subsection}

\begin{subsection}{Characterization of Polygon Complete {\EL}s}\label{sec:polygon complete  charact}

Here we characterize polygon complete {\EL}s in \cref{thm:characterization of non-polygon completeness}. Because the statements are slightly simpler, we actually characterize the {\EL}s which are not polygon complete which gives a characterization of polygon completeness for {\EL}s by negating the conditions in \cref{thm:characterization of non-polygon completeness}. The two technical conditions appearing in the following two lemmas provide the characterization. \cref{fig:non cover circled elts} is a schematic illustrating these conditions. It is clarifying to check these conditions in the example from \cref{fig:noncover diamond move example} where the elements are labeled to match the statements of \cref{lem:circledelementgivesnoncover} and \cref{lem:noncovergivescircledelements}.

\begin{figure}[H]
    \centering
    \scalebox{0.9}{
    \begin{tikzpicture}[very thick]
    \node (y) at (0,6) {$y$};
    \node (x1) at (0,5) {$x_1$};
    \node (x2) at (3,5) {$x_2$};
    \node (x3) at (6,5) {$x_3$};
    \node (x4) at (9,5) {$x_4$};
    \node (xn) at (-3,5) {$x_n$};
    \node (xn-1) at (-6,5) {$x_{n-1}$};
    \node (z1) at (1.5,4) {$z_1$};
    \node (z2) at (4.5,4) {$z_2$};
    \node (z3) at (7.5,4) {$z_3$};
    \node (zn) at (-1.5,4) {$z_n$};
    \node (zn-1) at (-4.5,4) {$z_{n-1}$};
    \draw [dotted] (9,3) -- (9.5,3);
    \draw [dotted] (-6.5,3) -- (-6,3);
    
    \node (yi-1) at (0,0) {$y_{i-1}$};
    \node (yi+s) at (0,2) {$y_{i+s}$};
    \node (y'i) at (0.5,1) {$y'_i$};
    \node (yi) at (-0.5,1) {$y_i$};
    \node (zh) at (0,-2) {$\zh$};

    \draw [dashed] (zh) -- (yi-1);
    \draw (yi-1) -- (yi);
    \draw [dashed] (yi) --node [left]{$m$} (yi+s);
    \draw (y'i) --node [right]{$m'$} (yi+s);
    \draw (yi-1) -- (y'i);
    
    \node (bend1) at (1.8,2) {$m_1$};
    \node (bend2) at (4,2) {$m_2$};
    \node (bend3) at (7,2) {$m_3$};
    \node (bendn) at (-1.8,2) {$m_n$};
    \node (bendn-1) at (-4,2) {$m_{n-1}$};
    
    \draw [dashed] (yi+s) -- (x1);
    \draw [dashed] (yi-1) -- (bend1) -- (z1);
    \draw [dashed] (yi-1) -- (bend2) -- (z2);
    \draw [dashed] (yi-1) -- (bend3) -- (z3);
    \draw [dashed] (yi-1) -- (bendn-1) -- (zn-1);
    \draw [dashed] (yi-1) -- (bendn) -- (zn);

    \draw (x1) -- (y);
    \draw (x2) -- (y);
    \draw (x3) -- (y);
    \draw (x4) -- (y);
    \draw (xn-1) -- (y);
    \draw (xn) -- (y);
    \draw [dashed] (z1) --node [below left]{$c_1$} (x1);
    \draw (z1) -- (x2);
    \draw [dashed] (z2) --node [below left]{$c_2$} (x2);
    \draw [dashed] (z3) --node [below left]{$c_3$} (x3);
    \draw [dashed] (zn) --node [below left]{$c_n$} (xn);
    \draw [dashed] (zn-1) --node [below left]{$c_{n-1}$} (xn-1);
    \draw (z2) -- (x3);
    \draw (z3) -- (x4);
    \draw (zn) -- (x1);
    \draw (zn-1) -- (xn);
    
    \end{tikzpicture}}
    \caption{Illustration of conditions (i) and (ii) in \cref{lem:circledelementgivesnoncover} and \cref{lem:noncovergivescircledelements}.
    } 
    \label{fig:non cover circled elts}
\end{figure}
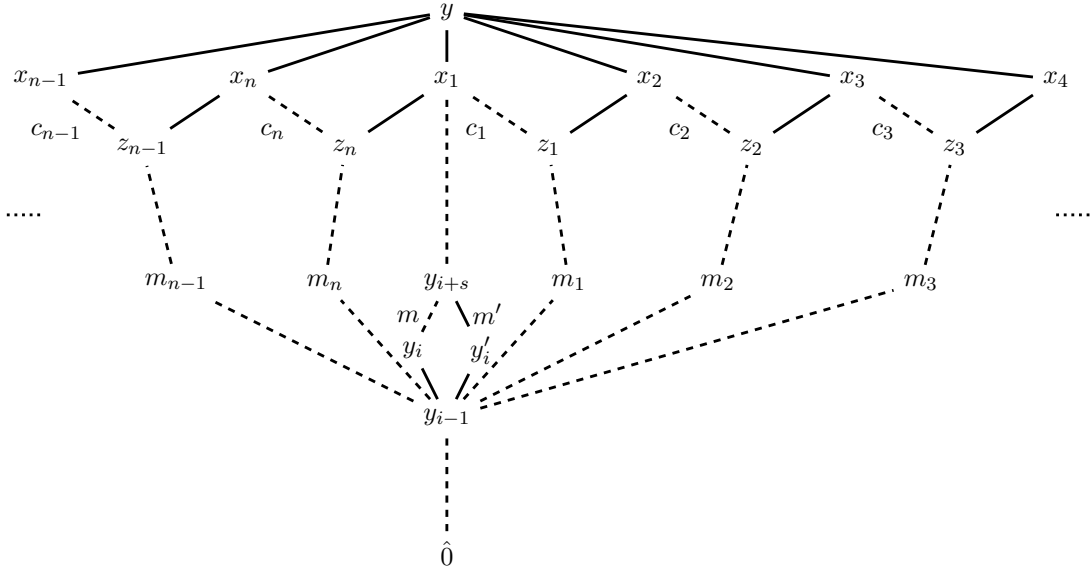

\begin{lemma}\label{lem:circledelementgivesnoncover}
Let $P$ be a finite, bounded poset in which the length of some maximal chain is at least three. Let $\lm$ be an {\EL} of $P$. Suppose that for some $n\geq 2$ there are elements $y,x_1,x_2,\dots,x_n,z_1,z_2,\dots,z_n\in P$ which, under the convention $x_{n+1}=x_1$, satisfy: 
\begin{itemize}
    \item [(i)] For $1\leq i\leq n$, $z_i\lessdot x_{i+1} \lessdot y$ is a descent in $[z_i,y]$ and $x_i\lessdot y$ is contained in the unique ascending saturated chain $c_i$ of $[z_i, y]$ with respect to $\lm$. 
    
    \item [(ii)] There are the following saturated chains of length at least one: $m,m'$ from $\zh$ to $x_1$ such that $m\to m'$ and $m_i$ from $\zh$ to $z_i$ for $1\leq i\leq n$ which satisfy the relations $\cleqi{m}{m_1 * c_1^{x_1}}{\lm}$ in $\cord{[\zh,x_1]}{\lm}$, $\cl{m_i * z_i * x_{i+1}}{m_{i+1} * c_{i+1}^{x_{i+1}}}{\lm}$ in $\cord{[\zh,x_{i+1}]}{\lm}$ for each $1\leq i\leq n$, and $\cleqi{m_n * z_n * x_1}{m'}{\lm}$ in $\cord{[\zh,x_{1}]}{\lm}$. (It is possible that $m$ contains $m_1$ and $c_1^{x_{1}}$.)
\end{itemize}
Then $\lm$ is not polygon complete.
\end{lemma}

\begin{proof}
It is helpful to refer to \cref{fig:non cover circled elts} to visually follow the proof. We assume there exist $y,x_1,x_2,\dots,x_n,z_1,z_2,\dots,z_n \in P$ such that conditions (i) and (ii) of \cref{lem:circledelementgivesnoncover} are met. We know $z_1\neq z_n$ since the chain $c_1$ is ascending and the chain $z_n\lessdot x_1 \lessdot y$ is a descent. By assumption in condition (ii), $m\to m'$. Thus, $m * y \to m' * y$. We will show that $\cordot{m * y \not}{_{\lm} m' * y}$ in $\cord{[\zh,y]}{\lm}$. By the assumptions of condition (i), $c_i \to z_i\lessdot x_{i+1} \lessdot y$ for each $1\leq i \leq n$. Thus, \begin{align*} m * y & \prec_{\lm} m_1 * c_1 \to m_1 * z_1 * x_2 * y \prec_{\lm} m_2 * c_2 \to m_2 * z_2 * x_3 * y \\
& \prec_{\lm} m_3 * c_3 \to m_3 * z_3 * x_4 * y \prec_{\lm} \dots \\
& \prec_{\lm} m_{n-1} * c_{n-1} \to m_{n-1} * z_{n-1} * x_n * y \prec_{\lm} m_{n} * c_n \to m_n * z_n * x_1 * y \\ 
& \prec_{\lm} m' * y.
\end{align*}
Hence, $\cordoti{m * y \not}{m' * y}{\lm}$ in $\cord{[\zh,y]}{\lm}$ since $n\geq 2$. Moreover, letting $c$ be any saturated chain from $y$ to $\oneh$ we have that $m * y * c \to m' * y * c$, but $\cordoti{m * y * c \not}{ m' * y * c}{\lm}$ in $\cord{P}{\lm}$ by \cref{lem:restrictedrelslift}. Therefore, $\lm$ is not polygon complete.
\end{proof}

In the next lemma, we show that conditions (i) and (ii) are also necessary for an {\EL} to be not polygon complete. Despite the appearance that the conditions of \cref{lem:circledelementgivesnoncover} are very technical and that they might be saying no more than that there exists an increase by a polygon move which does not give a cover relation, they are actually useful for verifying some {\EL}s are polygon complete. Condition (i) is particularly useful if we have control over the top labels of saturated chains in the polygons corresponding to increases by polygon moves as we do with polygon strong {\EL}s.

\begin{lemma}\label{lem:noncovergivescircledelements}
Let $P$ be a finite, bounded poset with an EL-labeling $\lm$. Suppose that for maximal chains $m,m'\in \M(P)$, $m\to m'$ while $\cordot{m \not}{_{\lm} m'}$ in $\cord{P}{\lm}$. Then there are elements $y,x_1,x_2,\dots,x_n,z_1,z_2,\dots,z_n\in P$ for $n\geq 2$ which, under the convention $x_{n+1}=x_1$, satisfy: 
\begin{itemize}
    \item [(i)] For $1\leq i\leq n$, $z_i\lessdot x_{i+1} \lessdot y$ is a descent in $[z_i,y]$ and $x_i\lessdot y$ is contained in the unique ascending saturated chain $c_i$ of $[z_i, y]$ with respect to $\lm$.
    
    \item [(ii)] There are the following saturated chains of length at least one: $m,m'$ from $\zh$ to $x_1$ such that $m\to m'$ and $m_i$ from $\zh$ to $z_i$ for $1\leq i\leq n$ which satisfy the relations $\cleqi{m}{m_1 * c_1^{x_1}}{\lm}$ in $\cord{[\zh,x_1]}{\lm}$, $\cl{m_i * z_i * x_{i+1}}{m_{i+1} * c_{i+1}^{x_{i+1}}}{\lm}$ in $\cord{[\zh,x_{i+1}]}{\lm}$ for each $1\leq i\leq n$, and $\cleqi{m_n * x_1}{m'}{\lm}$ in $\cord{[\zh,x_{1}]}{\lm}$. (It is possible that $m$ contains $m_1$ and $c_1^{x_{1}}$.)
\end{itemize}
\end{lemma}

\begin{proof}
Again \cref{fig:non cover circled elts} guides and illuminates this proof. Let $m:y_0=\zh \lessdot y_1 \lessdot \dots y_{i-1} \lessdot y_i \lessdot y_{i+1} \lessdot \dots \lessdot y_{i+s} \lessdot \dots \lessdot y_{t-1} \lessdot y_t=\oneh$ for some $s\geq 1$ with $y_i \lessdot y_{i+1} \lessdot \dots \lessdot y_{i+s}$ an ascending saturated chain with respect to $\lm$. Let $m':y_0=\zh \lessdot y_1 \lessdot \dots y_{i-1} \lessdot y'_i \lessdot y_{i+s} \lessdot \dots \lessdot y_{t-1} \lessdot y_t=\oneh$ with $y_{i-1} \lessdot y'_i \lessdot y_{i+s}$ a descent with respect to $\lm$. Assume $\cordot{m \not}{_{\lm} m'}$.

We will proceed by induction on the length of the longest maximal chain of $P$ to show the elements and chains of conditions (i) and (ii) exist. If the longest maximal chain is length one or two, then the statement is vacuously true since $\lm$ is polygon complete by \cref{lem:rank two polygon complete}. We assume the statement holds for posets with longest maximal chain of any length at most $l-1$ for some $l\geq 3$.

Assume $P$ has longest maximal chain of length $l\geq 3$. We first observe that if $y_{i+s}=\oneh$, then $\cordoti{m}{ m'}{\lm}$ by \cref{lem:nongradedonlyabovepolygon}. Thus, we may assume $y_{i+s}<\oneh$, so $i+s\leq t-1$. This implies $m^{y_{t-1}} \to m'^{y_{t-1}}$ which we will take advantage of repeatedly. Now since $\cordot{m \not}{_{\lm} m'}$, there must be maximal chains $d_0,d_1,d_2,\dots,d_k, d_{k+1}\in \M(P)$ such that $m=d_0\to d_1\to d_2\to \dots \to d_k\to d_{k+1}=m'$ with $k\geq 1$. There are two cases we must consider. Either the top elements of the polygons corresponding to each of the increases by a polygon move $d_j\to d_{j+1}$ for $0\leq j\leq k$ are strictly less than $\oneh$ in $P$ or the top element of the polygon corresponding to some increase by a polygon move $d_j\to d_{j+1}$ is $\oneh$.

Suppose the top element of each polygon corresponding to the increases by a polygon move $d_j\to d_{j+1}$ for $0\leq j\leq k$ is strictly less than $\oneh$ in $P$. Then $m^{y_{t-1}} \to d_1^{y_{t-1}} \to d_2^{y_{t-1}} \to \dots \to d_k^{y_{t-1}} \to m'^{y_{t-1}}$. Thus, $\cordoti{m^{y_{t-1}} \not}{m'^{y_{t-1}}}{\lm}$ in $\cord{[\zh,y_{t-1}]}{\lm}$ since $k\geq 1$. We previously observed that $m^{y_{t-1}} \to m'^{y_{t-1}}$. Now the length of the longest maximal chain in $[\zh,y_{t-1}]$ is some $l'\leq l-1$ since the longest maximal chain in $P$ has length $l$. If $l'\leq 2$, then $m^{y_{t-1}} \to m'^{y_{t-1}}$ with $\cordoti{m^{y_{t-1}} \not}{m'^{y_{t-1}}}{\lm}$ in $\cord{[\zh,y_{t-1}]}{\lm}$ contradicts \cref{lem:rank two polygon complete}. Thus, we may assume $l'\geq 3$. Then by the inductive hypothesis there exist elements and chains of $[\zh,y_{t-1}]$ satisfying conditions (i) and (ii). The same elements and chains satisfy conditions (i) and (ii) in $P$.

Next we consider the case that the top element of the polygon corresponding to some increase by a polygon move $d_j\to d_{j+1}$ is $\oneh$. We will construct the elements satisfying condition (i) and the chains of condition (ii) by considering the elements and chains involved each time the final edge of a chain changes in the sequence $m=d_0\to d_1\to d_2\to \dots \to d_k\to d_{k+1}=m'$. Let $y=\oneh$ and $x_1=y_{t-1}$. Let $d_{r_1}$ be the first maximal chain in the sequence $m=d_0\to d_1\to d_2\to \dots \to d_k\to d_{k+1}=m'$ such that $y_{t-1}\in d_{r_1}$, but $y_{t-1}\not \in d_{r_1 +1}$. Let $z_1\lessdot x_2\lessdot \oneh$ be the final three elements of $d_{r_1+1}$ which are uniquely determined by $d_{r_1+1}$ and must exist since $y_{t-1}\not \in d_{r_1+1}$. Then $d_{r_1}$ contains the ascending saturated chain $c_1$ from $z_1$ to $\oneh$ and $z_1\lessdot x_2\lessdot \oneh$ is a descent since $d_{r_1}\to d_{r_1+1}$ and $y_{t-1}\not \in d_{r_1+1}$. The saturated chains $c_1$ and $z_1\lessdot x_2\lessdot \oneh$ form the polygon corresponding to the polygon move $d_{r_1}\to d_{r_1+1}$. Now since $y_{t-1}\in d_{j'}$ for each $1\leq j'\leq d_{r_1}$, we have $m^{y_{t-1}}\to d_1^{y_{t-1}}\to \dots \to d_{r_1}^{y_{t-1}}$. Thus, $\cleqi{m^{y_{t-1}}}{d_{r_1}^{z_1} *  c_1^{y_{t-1}}}{\lm}$ in $\cord{[\zh,y_{t-1}]}{\lm}$. Let $m_1=d_{r_1}^{z_1}$.

Now since $d_{r_1+1}\to d_{r_1+2}\to \dots \to d_k \to m'$ and $y_{t-1}\lessdot \oneh$ is contained in $m'$, there is some $d_{r_n}$ which is the last maximal chain in the sequence $m=d_0\to d_1\to d_2\to \dots \to d_k\to d_{k+1}=m'$ such that $y_{t-1}\not \in d_{r_n}$, but $y_{t-1}\in d_{j'}$ for all $j'> r_n$. We do not yet know the value of $n$, but it is cleaner to introduce $d_{r_n}$ at this point. We will see that the following process results in the value of $n$ without depending on $n$. We note that $r_n<k+1$ since $y_{t-1}\in m'$. Let $x_n\in d_{r_n}$ be the unique element of the chain such that $x_n\lessdot \oneh$. Let $z_n$ be the unique element of $d_{r_n+1}$ such that $z_n\lessdot y_{t-1}$. Then to accomplish the increase by a polygon move $d_{r_n}\to d_{r_n+1}$ which removes $x_n$ and replaces it with $y_{t-1}$, we must have that $z_n<x_n$, $d_{r_n}$ contains the unique ascending chain $c_n$ from $z_n$ to $\oneh$ with $x_n\in c_n$, and $z_n\lessdot y_{t-1} \lessdot \oneh$ is a descent. Thus, $c_n$ and $z_n\lessdot y_{t-1} \lessdot \oneh$ form the polygon corresponding to $d_{r_n}\to d_{r_n+1}$. Then since $y_{t-1}\in d_{j'}$ for all $j'>r_n$, we have $d_{r_n+1}^{y_{t-1}}\to d_{r_n+2}^{y_{t-1}}\to \dots \to d_k^{y_{t-1}}\to d_{k+1}^{y_{t-1}}=m'^{y_{t-1}}$. Hence, we have $\cleqi{d_{r_n+1}^{z_n} *  z_n * y_{t-1}}{m'^{y_{t-1}}}{\lm}$ in $\cord{[\zh,y_{t-1}]}{\lm}$. Lastly, $d_{r_n+1}^{z_n}=d_{r_n}^{z_n}$, so $\cleqi{d_{r_n}^{z_n} * z_n * y_{t-1}}{m'^{y_{t-1}}}{\lm}$ in $\cord{[\zh,y_{t-1}]}{\lm}$. Let $m_n=d_{r_n}^{z_n}$, and recall $x_{n+1}=x_1$ by convention, so $x_{n+1}=y_{t-1}$.

Now we move from $d_{r_1}$ to $d_{r_n}$ by the same process to produce the remaining elements of condition (i) and chains of condition (ii).

Let $d_{r_2}$ be the first maximal chain after $d_{r_1+1}$ in the sequence $m=d_0\to d_1\to d_2\to \dots \to d_k\to d_{k+1}=m'$ such that $x_2\in d_{r_2}$, but $x_2\not\in d_{r_2+1}$. We note that $r_2>r_1+1$ since $z_1\lessdot x_2\lessdot \oneh$ is contained in $d_{r_1+1}$ and is a descent. Let $x_3$ be the unique element of $d_{r_2+1}$ such that $x_3\lessdot \oneh$ and let $z_2\in d_{r_2+1}$ be the unique element such that $z_2\lessdot x_3$. Thus, $z_2\lessdot x_3\lessdot \oneh$ is a descent and $d_{r_2}$ contains the unique ascending chain $c_2$ from $z_2$ to $\oneh$ with $x_2\in c_2$. The chains $c_2$ and $z_2\lessdot x_3\lessdot \oneh$ form the polygon corresponding to the polygon move $d_{r_2}\to d_{r_2+1}$. Since $x_2\in d_{j'}$ for all $r_1+1\leq j' \leq r_2$, we have $d_{r_1+1}^{x_2}\to \dots \to d_{r_2}^{x_2}$. Thus, $\cleqi{d_{r_1+1}^{z_1} * z_1 * x_2}{d_{r_2}^{z_2} * c_2^{x_2}}{\lm} $ in $\cord{[\zh,x_2]}{\lm}$. Lastly, $d_{r_1+1}^{z_1} = d_{r_1}^{z_1}$, so $\cleqi{d_{r_1}^{z_1} * z_1 * x_2}{d_{r_2}^{z_2} * c_2^{x_2}}{\lm} $ in $\cord{[\zh,x_2]}{\lm}$. Let $m_2=d_{r_2}^{z_2}$.

We continue this process until we reach $d_{r_n}$. Suppose we have constructed $d_{r_i}$, $x_{i+1}$, $z_i$, and $c_i$ by the above process and set $m_i=d_{r_i}^{z_i}$. Then $d_{r_{i+1}}$ is the first maximal chain after $d_{r_i+1}$ in the sequence $m=d_0\to d_1\to d_2\to \dots \to d_k\to d_{k+1}=m'$ such that $x_{i+1}\in d_{r_{i+1}}$, but $x_{i+1}\not \in d_{r_{i+1}+1}$. We note that $r_{i+1}>r_{i}+1$ since $z_i\lessdot x_{i+1}\lessdot \oneh$ is contained in $d_{r_{i}+1}$ and is a descent by construction. Let $x_{i+2}$ be the unique element of $d_{r_{i+1}+1}$ covered by $\oneh$ and let $z_{i+1}\in d_{r_{i+1}+1}$ be the unique element covered by $x_{i+2}$. Thus, $z_{i+1}\lessdot x_{i+2}\lessdot \oneh$ is a descent and $d_{r_{i+1}}$ contains the unique ascending saturated chain $c_{i+1}$ from $z_{i+1}$ to $\oneh$ with $x_{i+1}\in c_{i+1}$. The chains $c_{i+1}$ and $z_{i+1}\lessdot x_{i+2}\lessdot \oneh$ form the polygon corresponding to the polygon move $d_{r_{i+1}}\to d_{r_{i+1}+1}$. Then since $x_{i+1}\in d_{j'}$ for all $r_i+1\leq j'\leq r_{i+1}$, we have $d_{r_i+1}^{x_{i+1}}\to \dots \to d_{r_{i+1}}^{x_{i+1}}$. Thus, $\cleqi{d_{r_i+1}^{z_{i}} * z_{i} * x_{i+1}}{d_{r_{i+1}}^{z_{i+1}} * c_{i+1}^{x_{i+1}}}{\lm}$ in $\cord{[\zh,x_{i+1}]}{\lm}$. Lastly, $d_{r_i+1}^{z_{i}}=d_{r_i}^{z_{i}}$, so $\cleqi{d_{r_i}^{z_{i}} * z_{i} * x_{i+1}}{d_{r_{i+1}}^{z_{i+1}} * c_{i+1}}{\lm}$ in $\cord{[\zh,x_{i+1}]}{\lm}$. Let $m_{i+1} = d_{r_{i+1}}^{z_{i+1}}$. 

We are guaranteed to reach $d_{r_n}$ because the above process accounts for each change of the top edge in a chain in the sequence $m=d_0\to d_1\to d_2\to \dots \to d_k\to d_{k+1}=m'$. We know $n\geq 2$ because we are in the case where the top edge in some chain in the sequence $m=d_0\to d_1\to d_2\to \dots \to d_k\to d_{k+1}=m'$ changes from $y_{t-1}\lessdot \oneh$, and then must return to $y_{t-1}\lessdot \oneh$ because $y_{t-1}\lessdot \oneh$ is the top edge in both $m$ and $m'$. 

We have thus produced $y$, $z_1,\dots,z_n$, and $x_1,\dots, x_n$ satisfying condition (i) of the lemma. Letting $m_i=d_{r_i}^{z_{i}}$ for $1\leq i\leq n$ and letting $c_i$ be as constructed above for $1\leq i\leq n$, as well as, $m'= m'^{x_1}$ and $m= m^{x_1}$ (slightly abusing notation) we have produced the saturated chains in condition (ii) of the lemma. The chains $m= m^{x_1}$, $m'= m'^{x_1}$, and $m_i$ all have length at least one since $m$ and $m'$ have length at least three because $y_{i+s}<\oneh$.  This concludes the proof.
\end{proof}

\begin{theorem}\label{thm:characterization of non-polygon completeness}
Let $P$ be a finite, bounded poset with an EL-labeling $\lm$. Then $\lm$ fails to be polygon complete if and only if $P$ has some maximal chain of at least length three and there are elements $y,x_1,x_2,\dots,x_n,z_1,z_2,\dots,z_n\in P$ for $n\geq 2$ which, under the convention $x_{n+1}=x_1$, satisfy: 
\begin{itemize}
    \item [(i)] For $1\leq i\leq n$, $z_i\lessdot x_{i+1} \lessdot y$ is a descent in $[z_i,y]$ and $x_i\lessdot y$ is contained in the unique ascending saturated chain $c_i$ of $[z_i, y]$ with respect to $\lm$.
    
    \item [(ii)] There are the following saturated chains of length at least one: $m,m'$ from $\zh$ to $x_1$ such that $m\to m'$ and $m_i$ from $\zh$ to $z_i$ for $1\leq i\leq n$ which satisfy the relations $\cleqi{m}{m_1 * c_1^{x_1}}{\lm}$ in $\cord{[\zh,x_1]}{\lm}$, $\cl{m_i * z_i * x_{i+1}}{m_{i+1} * c_{i+1}^{x_{i+1}}}{\lm}$ in $\cord{[\zh,x_{i+1}]}{\lm}$ for each $1\leq i\leq n$, and $\cleqi{m_n * x_1}{m'}{\lm}$ in $\cord{[\zh,x_{1}]}{\lm}$. (It is possible that $m$ contains $m_1$ and $c_1^{x_{1}}$.)
\end{itemize}
\end{theorem}

\begin{proof}
The forward direction is \cref{lem:noncovergivescircledelements} and the backward direction is \cref{lem:circledelementgivesnoncover}.
\end{proof}

We now present a simple concrete condition on labelings of certain induced subposets which guarantees a {\CL} is not polygon complete. \cref{fig:induced subposet giving not polygon complete} is a schematic which illustrates this condition. We may also observe this condition in the examples from \cref{fig:noncover diamond move example} and \cref{fig:CL noncover diamond move example}.

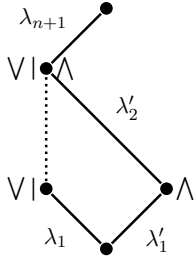
\begin{figure}[H]
    \centering
    \scalebox{0.8}{
    \begin{tikzpicture}[very thick]
    \node[fill,circle,inner sep=0pt,minimum size=6pt] (zh) at (0,0) {};
    \node[fill,circle,inner sep=0pt,minimum size=6pt] (a) at (-1,1) {};
    \node[fill,circle,inner sep=0pt,minimum size=6pt] (d) at (-1,3) {};
    \node[fill,circle,inner sep=0pt,minimum size=6pt] (e) at (0,4) {};
    \node[fill,circle,inner sep=0pt,minimum size=6pt] (b) at (1,1) {};
    
    \draw (zh) --node[below left]{$\lm_1$} (a)node[left]{$\bigvee \boldsymbol|$};
    \draw[dotted] (a) -- (d)node[left]{$\bigvee \boldsymbol|$};
    \draw (d) --node[above left]{$\lm_{n+1}$} (e);
    \draw (zh) --node[below right]{$\lm'_1$} (b)node[right]{$\bigwedge$} --node[above right]{$\lm'_2$} (d)node[right]{$\bigwedge$};
    \end{tikzpicture}
    }
    \caption{Illustration of the induced subposet condition in \cref{lem:easyincrnocovercondition}.}
    \label{fig:induced subposet giving not polygon complete}
\end{figure}

\begin{lemma}\label{lem:easyincrnocovercondition}
Let $P$ be a finite, bounded poset with a {\CL} $\lm$. Suppose there are saturated chains in $P$ of the form $c:x_1\lessdot x_2\lessdot \dots \lessdot x_k\lessdot x_{k+1}$ with $k\geq 3$ and $c':x_1\lessdot x_2' \lessdot x_k$ such that $c$ is ascending with respect to $\lm$, $c'$ is a descent with respect to $\lm$, and $\lm(x_k, x_{k+1})<\lm(x_2', x_k)$ all with respect to a root $r$ from $\zh$ to $x_1$. Then any maximal chain $m$ containing $c$ and $r$ increases by a polygon move to the maximal chain $m'$ obtained by replacing $c$ in $m$ with $c'$, but $m\not\precdot_{\lm} m'$.
\end{lemma}

\begin{proof}
First, for $m$ and $m'$ as defined above, $m\to m'$. Second, the chain $c$ is the unique ascending maximal chain of the rooted interval $[x_1,x_{k+1}]_{r}$ with respect to $\lm$. Also, $x_2'\lessdot x_{k}\lessdot x_{k+1}$ is not the unique ascending maximal chain of the rooted interval $[x_2',x_{k+1}]_{r*x'_2}$ since $\lm(x_k \lessdot x_{k+1})<\lm(x_2'\lessdot x_k)$. Thus, there is some saturated chain $c_0$ which is the unique ascending maximal chain of $[x_2',x_{k+1}]_{r*x_2}$ with respect to $\lm$. Now in the maximal chain descent order on $[x_1,x_{k+1}]_{r}$ induced by $\lm$, we have that $c$ is strictly less than $x_1 * c_0$ which is strictly less than $c'$. This follows from \cref{prop:non ranked cords have zero hat} and \cref{lem:restrictedrelslift}. Then \cref{lem:restrictedrelslift} extends this to the entire poset $P$ by producing a maximal chain which lies strictly between $m$ and $m'$ in the maximal chain order $\cord{P}{\lm}$. Thus, $m\not\precdot_{\lm} m'$.
\end{proof}

\end{subsection}

\section{A Sufficient Condition for Polygon Completeness via Inversions in Label Sequences}\label{sec:inversions in ELs}

In this section, we introduce a generalization of the notion of inversions of permutations. We speak more generally of inversions of maximal chains with respect to a {\CL}. 
The usual notion of inversions of permutations arises from the standard {\EL} of a Boolean lattice discussed in \cref{sec:motiv example boolean lat weak Sn}. We then formulate a condition on inversions of maximal chains of poset $P$ with respect to a {\CL} $\lm$ which implies that $\lm$ is polygon complete and that $\cord{P}{\lm}$ possesses some other strong properties.

\begin{definition}\label{def:inversions of maxl chains}
Let $P$ be a finite, bounded poset with a {\CL} $\lm$. Let $m:\zh=x_0\lessdot x_1\lessdot \dots \lessdot x_{r-1}\lessdot x_r=\oneh$ be a maximal chain of $P$. We say that the pair $(\lm(x_{i-1},x_i),\lm(x_{j-1},x_j))$ is an \textbf{inversion of $\mathbf{m}$ with respect to $\boldsymbol\lm$} if $1\leq i<j\leq n$ and $\lm(x_{i-1},x_i)\not \leq \lm(x_{j-1},x_j)$ in $\Lambda$. We denote the set of inversions of $m$ with respect to $\lm$ by $\textbf{inv}_{\boldsymbol\lm}(\mathbf{m})$.
\end{definition}

\begin{remark}\label{rmk:no notation problem in inv def}
There is a slight abuse of notation in \cref{def:inversions of maxl chains}. Technically, the label of a cover relation $x\lessdot y$ contained in maximal chain $m$ from a {\CL} $\lm$ should be written $\lm(m,x,y)$. However, inversions are only considered with reference to a particular maximal chain, so we use the notation $\lm(x,y)$ for both {\EL}s and {\CL}s in \cref{def:inversions of maxl chains} to avoid unnecessary clutter.
\end{remark}

\begin{figure}[H]
    \centering
    \scalebox{1}{
    \begin{tikzpicture}[very thick]
    \node[fill,circle,inner sep=0pt,minimum size=6pt] (zh) at (0,0) {};
    \node[fill,circle,inner sep=0pt,minimum size=6pt] (a) at (0,1) {};
    \node[fill,circle,inner sep=0pt,minimum size=6pt] (b) at (-1,2) {};
    \node[fill,circle,inner sep=0pt,minimum size=6pt] (c) at (1,2) {};
    \node[fill,circle,inner sep=0pt,minimum size=6pt] (d) at (0,3) {};
    \node[fill,circle,inner sep=0pt,minimum size=6pt] (oneh) at (0,4) {};
    
    \node (m) at (-1.5,2) {$m$};
    \node (m') at (1.5,2) {$m'$};
    
    \draw (zh) --node[left,blue] {$1$} (a) --node[below left,blue] {$2$} (b) --node[above left,blue] {$3$} (d) --node[left,blue] {$4$} (oneh);
    \draw (a) --node[below right,blue] {$5$} (c) --node[above right,blue] {$1$} (d) -- (oneh);
    
    \node[blue] (1) at (3,0) {$1$};
    \node[blue] (2) at (3,1) {$2$};
    \node[blue] (3) at (3,2) {$3$};
    \node[blue] (4) at (3,3) {$4$};
    \node[blue] (5) at (4,2) {$5$};
    
    \node (Lambda) at (2.5,1.5) {$\Lambda =$};
    
    \draw (1) -- (2) -- (3) -- (4);
    \draw (2) -- (5);
    \end{tikzpicture}
    }
    \scalebox{1}{ 
    \begin{tikzpicture}
    \node (filler) at (0,0) {};
    \node (inv1) at (0,2) {$\inv{m}{\lm} = \emptyset$};
    \node (inv2) at (0,1) {$\inv{m'}{\lm} = \set{(5,1),(5,4)}$};
    \end{tikzpicture}
    }
    \caption{Inversions with respect to an {\EL} with labels from poset $\Lambda$.}
    \label{fig:example of inversions wrt EL}
\end{figure}
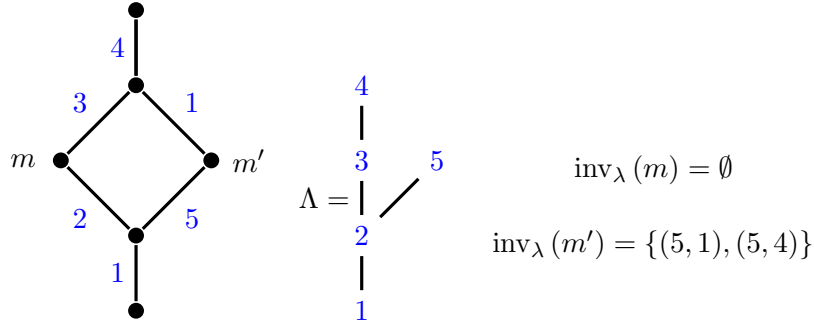

Next we define a natural condition on inversions which guarantees polygon completeness.
    
\begin{definition}\label{def:inversion ranked}
Let $P$ be a finite, ranked, bounded poset which admits an {\EL} or a {\CL} $\lm$. We say that $\lm$ is \textbf{inversion ranked} if $m\to m'$ implies $|\inv{m'}{\lm}|=|\inv{m}{\lm}|+1$.
\end{definition}

\begin{remark}\label{rmk:inversion example not inv ranked}
The {\EL} in the example from \cref{fig:example of inversions wrt EL} is not inversion ranked since $m\to m'$ while $|\inv{m}{\lm}|=0$ and $|\inv{m'}{\lm}|=2$.
\end{remark}

\begin{example}\label{ex:Sn EL is inversion ranked}
For any $S_n$ {\EL} $\lm$, the inversions with respect to $\lm$ are the usual inversions of the label sequences as permutations. In this case, $\lm$ is always inversion ranked by \cref{prop:Sn descent swaps}.
\end{example}

\begin{remark}\label{rmk:Sn CL not inversion ranked}
In contrast to \cref{ex:Sn EL is inversion ranked}, a {\CL} in which the label sequence of every maximal chain is some permutation of $[n]$ need not be inversion ranked. We observe this in \cref{ex:CL non polygon complete}. In that example, the chain labeled $123$, which has no inversions, increases by a polygon move to the chain labeled $321$, which has three inversions.
\end{remark}

\begin{example}\label{ex:inversion ranked EL}  \cref{fig:example inversion ranked EL} shows two maximal chains which could occur in an inversion ranked {\EL} $\lm$. In this example, the label set is $[4]$ with its standard total order. For brevity's sake, we set $\lm_i=\lm(x_{i-1},x_i)$ for $1\leq i\leq 5$ and $\lm'_3=\lm(x_2,x'_3)$ and $\lm'_4=\lm(x'_3,x_4)$. We have $$\inv{m}{\lm}=\set{(\lm_1,\lm_2),(\lm_1,\lm_3),(\lm_1,\lm_4),(\lm_1,\lm_5),(\lm_4,\lm_5)}$$ and $$\inv{m'}{\lm}=\set{(\lm_1,\lm_2),(\lm_1,\lm_4),(\lm_1,\lm_5),(\lm'_3,\lm'_4),(\lm'_3,\lm_5),(\lm'_4,\lm_5)},$$ so $\lm$ could be inversion ranked.

\begin{figure}[H]
    \centering
    \scalebox{1}{\begin{tikzpicture}[very thick]
    \node (x0) at (0,-1) {$x_0$};
    \node (zh) at (0,0) {$x_1$};
    \node (a) at (0,1) {$x_2$};
    \node (b) at (-1,2) {$x_3$};
    \node (c) at (1,2) {$x'_3$};
    \node (d) at (0,3) {$x_4$};
    \node (oneh) at (0,4) {$x_5$};
    
    \node (m) at (-2,2) {$m$};
    \node (m') at (2,2) {$m'$};
    
    \draw (x0) --node[left,blue]{$4$} (zh) --node[left,blue] {$1$} (a) --node[below left,blue] {$1$} (b) --node[above left,blue] {$2$} (d) --node[left,blue] {$1$} (oneh);
    \draw (a) --node[below right,blue] {$4$} (c) --node[above right,blue] {$3$} (d) -- (oneh);
    \end{tikzpicture}
    }
    \caption{Labeled chains which could occur in an inversion ranked {\EL}.}
    \label{fig:example inversion ranked EL}
\end{figure}
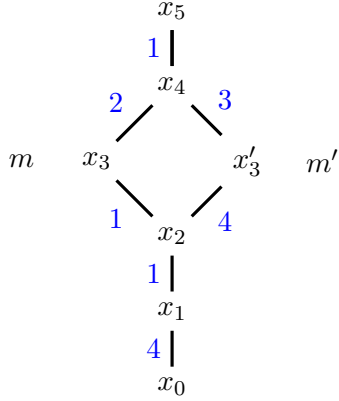

\end{example}

Next we observe that the notions of polygon strong and inversion ranked are generally distinct.

\begin{proposition}\label{prop:inv ranked distinct from polygon strong}
The notions of an inversion ranked {\EL} and a polygon strong {\EL} are distinct, that is, neither notion implies the other. 
\end{proposition}

\begin{proof}
Let $P$ be the poset with elements $\set{a,b,c,d}$ with maximal chains $a\lessdot b\lessdot d$ and $a\lessdot c\lessdot d$. Let $\lm$ be the {\EL} of $P$ given by $\lm(a,b)=1$, $\lm(b,d)=2$, $\lm(a,c)=4$, and $\lm(c,d)=3$. Then $\lm$ is inversion ranked, but not polygon strong. On the other hand, minimal labelings of geometric lattices are polygon strong by \cref{thm:min labelings polygon strong} while none of the minimal labelings of the partition lattice $\Pi_4$ are inversion ranked, in particular the minimal labeling shown in \cref{fig:Pi4 with min labeling} is not inversion ranked.
\end{proof}

In the ensuing theorem, we show that inversion ranked implies the following: polygon complete, the maximal chain descent order is ranked by number of inversions, and homology facets of the order complex are determined by their rank in the maximal chain descent order. This notion of inversion ranked is given in \cref{def:inversion ranked}.

\begin{theorem}\label{thm:inversion condition implies ranked cord}
Let $P$ be a finite, ranked, bounded poset of rank $n$. Suppose $P$ admits a {\CL} $\lm$ which is inversion ranked. Then $\lm$ is polygon complete. Moreover, $\cord{P}{\lm}$ is ranked with rank function $|\inv{\cdot}{\lm}|$ and $m\in \M(\overline{P})$ is a homology facet of the shellings of $\Delta(\overline{P})$ induced by any linear extension of $\cord{P}{\lm}$ if and only if $\zh*m*\oneh$ has rank $\binom{n}{2}$ in $\cord{P}{\lm}$.
\end{theorem}

\begin{proof}
Seeking a contradiction, suppose $m\to m'$ and $\cordoti{m \not}{m'}{\lm}$ in $\cord{P}{\lm}$ for maximal chains $m,m'\in \M(P)$. This implies that there are maximal chains $m_1,\dots, m_k\in \M(P)$ with $k\geq 1$ such that $m\to m_1\to \dots \to m_k\to m'$. Since $\lm $ is inversion ranked and since $k\geq 1$,  $|\inv{m'}{\lm}|=|\inv{m}{\lm}|+k+1\geq |\inv{m}{\lm}|+ 2$. However, this contradicts the fact that $|\inv{m'}{\lm}|=|\inv{m}{\lm}|+1$ which holds because $\lm $ is inversion ranked and $m\to m'$. Therefore, $\lm$ is polygon complete.

It follows directly from \cref{prop:non ranked cords have zero hat} that $\cord{P}{\lm}$ is ranked with rank function $|\inv{\cdot}{\lm}|$ since the unique ascending maximal chain $m_0$ of $P$ with respect to $\lm$ has $\inv{m_0}{\lm}=\emptyset$. By \cref{lem:linextsgiveshellings} any linear extension of $\cord{P}{\lm}$ induces a shelling of $\Delta(\overline{P})$ and a maximal chain $m\in \M(\overline{P})$ is a homology facet with respect to such a shelling if and only if $\zh*m*\oneh$ is descending with respect to $\lm$. Clearly a maximal chain $m\in \M(\overline{P})$ has $\zh*m*\oneh$ descending with respect to $\lm$ if and only if $|\inv{\zh*m*\oneh}{\lm}|= \binom{n}{2}$. Thus, $m\in \M(\overline{P})$ is a homology facet of $\Delta(\overline{P})$ if and only if $\zh*m*\oneh$ has rank $\binom{n}{2}$ in $\cord{P}{\lm}$.
\end{proof}

The previous theorem applies to $S_n$ {\EL}s by \cref{ex:Sn EL is inversion ranked}. In the next section, we will prove an even stronger result for $S_n$ {\EL}s.

\begin{corollary}\label{cor:Sn inv ranked so cord satisfies inv ranked thm}
If $P$ is a finite poset with an $S_n$ {\EL} $\lm$, then $\lm$ is polygon complete, $\cord{P}{\lm}$ is ranked with rank function $|\inv{\cdot}{\lm}|$, and $m\in \M(\overline{P})$ is a homology facet of 
the shellings of $\Delta(\overline{P})$ given by any linear extension of $\cord{P}{\lm}$ if and only if $\zh*m*\oneh$ has rank $\binom{n}{2}$ in $\cord{P}{\lm}$.
\end{corollary}

\section{More Examples of Maximal Chain Descent Orders}\label{sec:first examples}

Our first example sets us up to discuss some related examples later which are endowed with especially rich structure.

\begin{subsection}{\texorpdfstring{$\mathbf{S_n}$}{Lg} {\EL}s of Finite Supersolvable Lattices}

Here we characterize the intervals in maximal chain descent orders induced by Stanley's $M$-chain {\EL}s of any finite supersolvable lattice from \cite{stanleysupersolvablelats1972}. Specifically, we show that any interval in such a maximal chain descent order is isomorphic to some interval in weak order on $S_n$ where $n$ is the rank of the supersolvable lattice. Proving injectivity of the isomorphism in the theorem is the aspect of the proof which requires care. We also apply previous results about polygon completeness and rank in the proof.

\begin{theorem}\label{thm:Sn cords weak lower intervals}
Let $P$ be a finite supersolvable lattice of rank $n$ with an $M$-chain {\EL} $\lm$ as in \cite{stanleysupersolvablelats1972}. Let $\cord{P}{\lm}$ be the maximal chain descent order induced by $\lm$. Then any interval in $\cord{P}{\lm}$ is isomorphic to some interval of weak order on $S_n$ via the map assigning to each maximal chain its label sequence.
\end{theorem}

\begin{proof}
We show that for any maximal chain $m\in \M(P)$, the lower interval of $\cord{P}{\lm}$ generated by $m$ is isomorphic to the lower interval of weak order on $S_n$ generated by $\lm(m)$ via taking label sequences. Then the statement for all closed intervals and open intervals follows immediately. Again the central fact is that the label sequences of maximal chains are permutations of $[n]$.

By \cref{cor:Sn inv ranked so cord satisfies inv ranked thm}, $\lm$ is inversion ranked, and so $\lm$ is polygon complete and $\cord{P}{\lm}$ is ranked by $|\inv{\cdot}{\lm}|$. (Alternatively, we could apply \cref{thm:Sn EL is polygon strong} to see that $\lm$ is polygon complete, but we use the conclusion that $P$ is ranked here as well.) Let $m_0$ be the unique ascending chain of $P$ with respec to $\lm$, so $m_0$ is the $\zh$ of $\cord{P}{\lm}$ and $\lm(m_0)$ is the identity permutation. We first show that the set of label sequences of the elements in the interval $[m_0,m]_{\lm}$ is the set of permutations in the lower interval of weak order $[\lm(m_0),\lm(m)]_{wk}$. Then we show that the map $c\mapsto \lm(c)$ is an isomorphism from $[m_0,m]_{\lm}$ to $[\lm(m_0),\lm(m)]_{wk}$. For both, we induct on the rank of $m$ in $\cord{P}{\lm}$.

For each descent of $m$, the unique chain $m'$ with $m'\to m$ from \cref{prop:descents give unique incrs} has label sequence $\lm(m')$ given by transposing the corresponding descent of $\lm(m)$ by \cref{prop:Sn descent swaps}. Thus, by induction on the rank of $m$, the set of label sequences of the elements in $[m_0,m]_{\lm}$ is the set of permutations in $[\lm(m_0),\lm(m)]_{wk}$. Thus, $c\mapsto \lm(c)$ is surjective from $[m_0,m]_{\lm}$ to $[\lm(m_0),\lm(m)]_{wk}$. The fact that $m'\to m$ implies the label sequence $\lm(m')$ is obtained from $\lm(m)$ by transposing a unique descent of $\lm(m)$ also means that the map $c\mapsto \lm(c)$ from $[m_0,m]_{\lm}$ to $[\lm(m_0),\lm(m)]_{wk}$ is order preserving. 

Next we show that $c\mapsto \lm(c)$ from $[m_0,m]_{\lm}$ to $[\lm(m_0),\lm(m)]_{wk}$ is injective. We again proceed by induction on the rank of $m$. If the rank of $m$ is zero, then $m=m_0$ by \cref{prop:non ranked cords have zero hat}. Since $m_0$ is the unique chain of $P$ whose label sequence is the identity permutation, this gives the base case. Now assume the rank of $m$ is greater than zero. Suppose $c,c'\in [m_0,m]_{\lm}$ with $\lm(c)=\lm(c')$. Observe that $m$ is the only element of $[m_0,m]_{\lm}$ with label sequence $\lm(m)$ by \cref{cor:nongradedcordgreaterlexgreater}. Thus, we may assume the rank of $c$ and $c'$ is strictly less than the rank of $m$. 

Let $c_1$ and $c_2$ be elements of $[m_0,m']_{\lm}$ such that $\cordot{\cleq{c}{_{\lm} c_1}}{_{\lm} m}$ and $\cordot{\cleq{c'}{_{\lm} c_2}}{_{\lm} m}$. We have $c_1\to m$ and $c_2\to m$. By \cref{prop:descents give unique incrs}, $c_1$ and $c_2$ are uniquely determined by the descents of $m$ to which they correspond. Also, the rank of $c_1$ and $c_2$ is one less than the rank of $m$. If $c_1=c_2$, then $c=c'$ by induction. If $c_1\neq c_2$, then by induction $c$ is the only element of $[m_0,c_1]_{\lm}$ with label sequence $\lm(c)$ and $c'$ is the only element of $[m_0,c_2]_{\lm}$ with label sequence $\lm(c')$. Thus, it suffices to show that $\cleq{c'}{_{\lm} c_1}$. 

Let $m:\zh=x_0\lessdot x_1\lessdot \dots \lessdot x_n=\oneh$ and $c_1:\zh=x_0\lessdot x_1\lessdot \dots x_{i-1}\lessdot x_i' \lessdot x_{i+1}\lessdot \dots \lessdot x_n=\oneh$ and $c_2:\zh=x_0\lessdot x_1\lessdot \dots x_{j-1}\lessdot x_j' \lessdot x_{j+1}\lessdot \dots \lessdot x_n=\oneh$. We must have $i\neq j$ since $c_1\neq c_2$. We have the following two cases: (i) $|i-j|\geq 2$ and (ii) $|i-j|=1$. 

(i) If $|i-j|\geq 2$, then the descents of $m$ corresponding to $c_1$ and $c_2$ share no common elements. We may assume without loss of generality that $i<j$. Let $c_3$ be the maximal chain of $P$ given by $c_3:\zh=x_0\lessdot x_1\lessdot \dots x_{i-1}\lessdot x_i' \lessdot x_{i+1}\lessdot \dots \lessdot x_{j-1}\lessdot x'_j \lessdot x_{j+1}\lessdot \dots \lessdot x_n=\oneh$.
We have $\cl{c_3}{c_1,c_2}{\lm}$ since we could have transposed the descents of $m$ at $x_i$ and $x_j$ in either order to reach $c_3$ (here we are using \cref{prop:descents give unique incrs}). The label sequence $\lm(c_3)$ is the meet of the label sequences $\lm(c_1)$ and $\lm(c_2)$ in weak order. Thus, $\lm(c')$ is less than $\lm(c_3)$ in weak order. We previously showed the the label sequences of elements in $[m_0,c_3]_{\lm}$ are the permutations in the weak order interval $[\lm(m_0),\lm(c_3)]_{\lm}$. Hence, there is some element $c''\in [m_0,c_3]_{\lm}$ with label sequence $\lm(c')$. Since $c''$ is also in $[m_0,c_2]_{\lm}$ and $c'$ is the unique such element with label sequence $\lm(c')$, $c''=c'$. Therefore, $\cl{\cleq{c'}{_{\lm} c_3}}{c_1}{\lm}$.

(ii) If $|i-j|=1$, the descents of $m$ corresponding to $c_1$ and $c_2$ share a common element. Without loss of generality, we may assume $j=i+1$. Thus, the saturated subchain of $m$ given by $x_{i-1}\lessdot x_i\lessdot x_{i+1}\lessdot x_{i+2}$ is a descending chain with respect to $\lm$. Let $d$ be the unique ascending saturated chain from $x_{i-1}$ to $x_{i+2}$ and let $c_3=m^{x_{i-1}} * d * m_{x_{i+2}}$. Then $\cl{c_3}{c_1,c_2}{\lm}$ by \cref{prop:non ranked cords have zero hat} and \cref{lem:restrictedrelslift}. Further, the label sequence $\lm(c_3)$ is the meet of $\lm(c_1)$ and $\lm(c_2)$ in weak order. Then, by the same argument as in case (i), we have $\cleq{c'}{_{\lm} c_3}$. Hence, $\cl{\cleq{c'}{_{\lm} c_3}}{c_1}{\lm}$. This completes the proof of injectivity.

Lastly, \cref{prop:descents give unique incrs} implies that if $\lm(c)\lessdot _{wk}\lm(c')$ in weak order on $S_n$ for $c,c'\in [m_0,m]_{\lm}$, then $\cordoti{c}{c'}{\lm}$. To see this we suppose $c,c'\in [m_0,m]_{\lm}$ and $\lm(c)\lessdot_{wk}\lm(c')$. By \cref{prop:descents give unique incrs} there is a unique maximal chain $c''$ such that $c''\to c'$ corresponding to the descent of $c'$ giving rise to $\lm(c)\lessdot_{wk}\lm(c')$. Thus, $c''\in [m_0,m]_{\lm}$ and $\lm(c'')=\lm(c)$. Then since the map to label sequences is injective on $[m_0,m]_{\lm}$, $c''=c$. Hence, by \cref{cor:Sn inv ranked so cord satisfies inv ranked thm} (or by \cref{thm:Sn EL is polygon strong}) $\cordoti{c}{c'}{\lm}$ if and only if $\lm(c)\lessdot_{wk} \lm(c')$. Therefore, $c\mapsto \lm(c)$ is an isomorphism from $[m_0,m]_{\lm}$ to $[\lm(m_0),\lm(m)]_{wk}$.
\end{proof}

In the proof of \cref{thm:Sn cords weak lower intervals}, we used the fact that weak order on $S_n$ is a lattice, but we never used the fact that $P$ was a lattice. The proof only relied on label sequences of maximal chains being permutations of $[n]$. Thus, the proof applies to McNamara's more general $S_n$ {\EL}s.

\begin{corollary}\label{cor:all Sn ELs cord ints weak ints}
Let $P$ be a finite, bounded poset with an $S_n$ {\EL} $\lm$ in the sense of McNamara \cite{snelsupersolvmcnamara2003}. Let $\cord{P}{\lm}$ be the maximal chain descent order induced by $\lm$. Then any interval in $\cord{P}{\lm}$ is isomorphic to some interval of weak order on $S_n$ via the map assigning to each maximal chain its label sequence.
\end{corollary}
 
As mentioned just before \cref{thm:interpolating ELs polygon strong}, McNamara and Thomas generalized the notion of $S_n$ {\EL}s to non-graded posets with the notion of interpolating {\EL}s in \cite{posetedgelabelsmodularitymcnamarathomas2006}. It may be interesting to see if interpolating {\EL}s have a clean structure theorem for intervals in their maximal chain descent orders. 

On the other hand, we might consider general {\CL}s in which the label sequence of every maximal chain of a rank $n$ poset is a permutation of $[n]$. We might call such {\CL}s $S_n$ {\CL}s. Bj{\"o}rner and Wachs' {\CL} of closed intervals (when restricted to lower intervals) in the Bruhat order of any Coxeter group is such a {\CL}. $S_n$ {\CL}s were also remarked upon in \cite{broadclassshelllatsschweigwoodroofe2017} where Schweig and Woodroofe show that comodernistic lattices (which they introduce to unify several well know classes of lattices among other purposes) admit $S_n$ {\CL}s. The most obvious potential characterization of intervals in the maximal chain descent order of an $S_n$ {\CL}, namely that each interval is isomorphic to some interval in weak order on $S_n$ via the map taking maximal chains to their label sequences, is not true as exhibited by the {\CL} in \cref{fig:CL noncover diamond move example}. As seen in that example, the fact that edge labels in a general {\CL} can depend on the root below the edge, means we can leave out permutations that are required as label sequences in an $S_n$ {\EL}. We do not know if there is a nice description of intervals from such $S_n$ {\CL}s.

\end{subsection}

\begin{subsection}{Finite Distributive Lattices}\label{sec:findistlats}
Finite distributive lattices are examples of finite supersolvable lattices. They are convenient examples to work with because the $M$-chain {\EL}s are especially easy to describe and are well controlled. We prove that for any of the $M$-chain {\EL}s of a finite distributive lattice described in \cref{rmk:dist lat Sn EL labels}, the corresponding maximal chain descent order is isomorphic to some order ideal in weak order on $S_n$ via the map assigning to each maximal chain its label sequence. We begin by recalling Birkhoff's well known Fundamental Theorem of Finite Distributive Lattices from \cite{birkhoffftfdl1937}.

\begin{theorem}\label{thm:ftfdl}
A poset $L$ is a finite distributive lattice if and only if $L=J(P)$ for some finite poset $P$ where $J(P)$ is the set of order ideals of $P$ ordered by inclusion.
\end{theorem}

Distributive lattices are instances of supersolvable lattices which have natural $M$-chain {\EL}s due to Stanley in \cite{stanleysupersolvablelats1972}. In the case of a distributive lattice $J(P)$, these $M$-chain {\EL}s have extra the structure of being given by the linear extensions of the poset $P$. 

\begin{remark}\label{rmk:dist lat Sn EL labels}
Linear extensions of $P$ give $S_n$ EL-labelings of $J(P)$ where $n=|P|$. Fix $e$, a linear extension of $P$.  For any cover relation $I\lessdot I'$ in $J(P)$ ($I$ and $I'$ are order ideals of $P$), $I'=I\cup \set{x}$ for some $x\in P$. We define an edge labeling $\lm_e$ of $J(P)$ by $\lm_e(I\lessdot I')=e(x)$. This edge labeling gives an $S_n$ EL-labeling since every element of $P$ must be added exactly once at some point along each maximal chain in $J(P)$. Let $\mathbf{\boldsymbol\Lin(P,e)}$ be the set of permutations appearing as label sequences of maximal chains in $J(P)$ when labeled by $\lm_e$.
\end{remark}

In \cite{permstatslinextsbjornerwachs1991}, Bj{\"o}rner and Wachs study the sets of label sequences of distributive lattices defined by the {\EL}s of \cref{rmk:dist lat Sn EL labels}. We use a special case of Proposition 4.1 from \cite{permstatslinextsbjornerwachs1991}. Their proposition and the special case stated below can both be proven straightforwardly using \cref{prop:Sn descent swaps}.

\begin{proposition}\label{prop:Sn labels of dist lat weak ideal}
Let $P$ be a finite poset with $|P|=n$. Let $e$ be a linear extension of $P$. Then $\Lin(P,e)$, the set of permutations appearing as label sequences of maximal chains in $J(P)$ when labeled by $\lm_e$, is an order ideal of the weak order on $S_n$.

The notation $\Lin(P,e)$ is defined in \cref{rmk:dist lat Sn EL labels}.
\end{proposition}

We already know from \cref{thm:Sn EL is polygon strong} that $\lm_{e}$ is polygon complete, and from \cref{cor:Sn inv ranked so cord satisfies inv ranked thm} we know that $\cord{J(P)}{\lm_e}$ is ranked by the number of inversions in addition to $\lm$ being polygon complete. We also know from \cref{thm:Sn cords weak lower intervals} that intervals in $\cord{J(P)}{\lm_e}$ are isomorphic to intervals in the weak order on $S_{|P|}$. Next we prove the stronger statement that $\cord{J(P)}{\lm_e}$ is isomorphic to the order ideal $\Lin(P,e)$ of the weak order on $S_{|P|}$.

\begin{theorem}\label{thm:dist lat Sn cord isom weak ideal}
Let $P$ be a finite poset with $|P|=n$. Let $e$ be a linear extension of $P$. Then the maximal chain descent order $\cord{J(P)}{\lm_e}$ is isomorphic to the order ideal $\Lin(P,e)$ of 
the weak order on $S_n$ via the map assigning 
to each maximal chain its label sequence with respect to $\lm_e$.

The notation $\Lin(P,e)$ is defined in \cref{rmk:dist lat Sn EL labels}.
\end{theorem}

\begin{proof}
We observe that the label sequences of maximal chains in $J(P)$ when labeled by $\lm_e$ are all distinct since they record the order in which the elements of $P$ were added to form the ideals in the chain. Thus, the map $m\mapsto \lm_e(m)$ is a bijection from $\cord{J(P)}{\lm_e}$ to $\Lin(P,e)$. Also, $\Lin(P,e)$ is an order ideal of the weak order on $S_n$ by \cref{prop:Sn labels of dist lat weak ideal}. By \cref{thm:Sn cords weak lower intervals}, $m\mapsto \lm_e(m)$ is a poset isomorphism from each interval of $\cord{J(P)}{\lm_e}$ to its image in weak order on $S_n$, so the map is order preserving. It only remains to show that the inverse map is also order preserving, i.e. that the lower intervals in $\cord{J(P)}{\lm_e}$ are not ``glued together" in an unruly way. 

For $m\in \M(P)$, let $\inv{\lm(m)}{}$ be the set of inversions of $\lm(m)$ as a permutation of $n$ (this is identical to the set of inversions $\inv{m}{\lm}$ from \cref{def:inversions of maxl chains}). Let $m,m'\in \M(P)$ be maximal chains such that $\lm_e(m')=\lm'_1 \lm'_2\dots \lm'_n$ and $\lm_e(m)<_{wk}\lm_e(m')$ in weak order on $S_n$. Thus, there is a descent $\lm'_i>\lm'_{i+1}$ of $\lm_e(m')$ such that $(\lm'_{i+1},\lm'_i)\not\in \inv{\lm_e(m)}{}$. Then by \cref{prop:Sn descent swaps}, there is a unique maximal chain $m''$ such that $m''\to m'$ and $\lm_e(m'')=\lm_e(m')(i,i+1)$. Thus, $\cl{m''}{m'}{\lm_e}$ and $\inv{\lm_e(m'')}{}= \inv{\lm_e(m')}{}\setminus \set{(\lm'_{i+1},\lm'_i)}$. Hence, $\inv{\lm_e(m)}{} \subseteq \inv{\lm_e(m'')}{}$, so $\lm_e(m)\leq_{wk} \lm_e(m'')$ by \cref{prop:wkordinvsetcontainment}. Then by induction $\cleq{m}{_{\lm_e} m''}$. Therefore, $\cl{m}{m'}{\lm_e}$ which completes the proof.
\end{proof}

\end{subsection}

A particularly nice class of distributive lattices are intervals in Young's Lattice. Since this example possesses elegant additional structure, we make it a separate subsection.

\begin{subsection}{Intervals in Young's Lattice} \label{subsec:bkgrndyoungslat}
We show here that maximal chain descent orders of intervals in Young's lattice can be realized as posets on standard Young tableaux. We briefly recall the definitions of Young  diagrams, Young tableau, and Young's lattice. For a thorough introduction to this topic see \cite{fultontableaux1997} or \cite{ec2stanley1999} Chapter 7, for instance. Then, using a result of Bj{\"o}rner and Wachs from \cite{genquotcxtrgrpsbjornerwachs1988}, we prove that certain cases of maximal chain descent orders are isomorphic to generalized quotients of the symmetric group which were introduced in \cite{genquotcxtrgrpsbjornerwachs1988}.
 
A \textbf{Young diagram} $\alpha$ is a collection of rows of left justified boxes in which the $i$th row from the top has at most as many boxes as the $(i-1)$th row. Whenever convenient, we consider a Young diagram to include an arbitrary number of extra rows with zero boxes. We may refer to a Young diagram as the non-increasing tuple of the lengths of its rows, i.e. an integer partion. \cref{fig:youngdiag321} shows the Young diagram $\alpha=(3,2,1)$. We index the boxes in a Young diagram by matrix coordinates, so the coordinates $(i,j)$ refer to the box in row $i$ and column $j$. 

\begin{figure}[H]
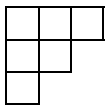

    \centering
    \yng(3,2,1)
    \caption{The Young diagram for $\alpha=(3,2,1)$.}
    \label{fig:youngdiag321}
\end{figure}

We say that Young diagram $\alpha$ contains Young diagram $\mu$ if the Young diagram of $\alpha$ contains the Young diagram of $\mu$. This is the same as each row of $\alpha$ being at least as large as the corresponding row of $\mu$ where we add as many $0$s as necessary so that the diagrams have the same number of rows. If $\alpha$ contains $\mu$, then we define a \textbf{skew diagram} $\alpha/\mu$ as the diagram of the boxes contained in $\alpha$ but not in $\mu$. 
 
\textbf{Young's Lattice}, denoted $\Y$, is the partial order on all Young diagrams by diagram containment. We note that $\Y$ is infinite, has a unique minimal element given by the empty partition $\emptyset$, and $\Y$ is graded by the number of boxes in the diagram. Further, $\Y$ is a distributive lattice. For a fixed Young diagram $\alpha$, we denote the principal order ideal of $\Y$ generated by $\alpha$ as $\Y(\alpha)$. For $\mu$ contained in $\alpha$, we denote the closed interval $[\mu,\alpha]$ by $\Y(\mu,\alpha)$. Since $\Y$ is a distributive lattice, $\Y(\alpha)$ and $\Y(\mu,\alpha)$ are finite distributive lattices. Thus, by \cref{thm:ftfdl}, $\Y(\alpha)$ and $\Y(\mu,\alpha)$ are the posets of order ideals of some finite posets. We let $P_{\alpha}$ be the partial order on the boxes in the Young diagram of $\alpha$ defined by the product order on the coordinates of the boxes, that is, box $(i,j)$ is less than or equal to box $(i',j')$ if and only if $i\leq i'$ and $j\leq j'$. Similarly, we let $P_{\alpha/\mu}$ be the partial order on the boxes of the skew diagram $\alpha/\mu$ defined by the product order on the coordinates of the boxes. A box in $P_{\alpha}$ or $P_{\alpha/\mu}$ is exactly covered by the adjacent box to the east and the adjacent box to the south, if such boxes exist in the relevant diagram. Then $\Y(\alpha)\cong J(P_{\alpha})$ and $\Y(\mu,\alpha)\cong J(P_{\alpha/\mu})$.

A \textbf{standard Young tableau} is an assignment of a positive integer from $[n]$ to each box of a Young diagram $\alpha$ with $n$ boxes such that:

\begin{itemize}
    \item [(1)] The box fillings strictly increase across each row from left to right.
    \item [(2)] The box fillings strictly increase down each column from top to bottom.
\end{itemize}

The Young diagram $\alpha$ is referred to as the \textbf{shape} of the Young tableau. Standard skew tableau are defined analogously as integer fillings of skew diagrams with the same row and column requirements. \cref{fig:youngtabexampls} shows an example of a standard Young tableau of shape $(3,2,1)$. In what follows, the arguments for Young diagrams and skew diagrams are the same, so we let $\alpha$ denote a Young diagram or a skew diagram. Thus, $\Y(\alpha)$ can denote any closed interval in $\Y$.

We will denote the collection of standard tableau of shape $\alpha$ by $\mathbf{ST_{\alpha}}$. For $T$, a standard tableau with $n$ boxes, and a box $b$ in $T$ (possibly given by its coordinates or some other description), we denote the filling of box $b$ in $T$ by $T(b)$. For $i\in [n]$, denote the box of $T$ whose filling is $i$ by $T^i$.

\begin{figure}[H]
    \centering
    \young(123,45,6)
    \caption{A standard Young tableau of shape $\alpha=(3,2,1)$.}
    \label{fig:youngtabexampls}
\end{figure}

The \textbf{row word} of a tableau $T$ is the word whose letters are the entries of $T$ obtained by reading the rows of $T$ from left to right where we read the rows from top to bottom. The row word of $T$ is denoted $\mathbf{w(T)}$. For instance, the row word of the tableau in \cref{fig:youngtabexampls} is $123456$. We may choose other reading orders of the boxes of a tableau to obtain other words. Reading the columns from top to bottom and the columns from left to right gives the \textbf{column word}.

Since $\Y(\alpha)\cong J(P_{\alpha})$, the maximal chains of $\Y(\alpha)$ are in bijection with the linear extensions of $P_{\alpha}$ by taking the order in which the boxes are added in the maximal chain to form the diagram for $\alpha$. By construction, the linear extensions of $P_{\alpha}$ precisely give the standard tableau of shape $\alpha$ by filling each box with its value under the linear extension. Thus, $\M(\Y(\alpha))$ is in bijection with $ST_{\alpha}$. We denote the standard tableau corresponding to maximal chain $m\in \M(\Y(\alpha))$ by $T_m$. We denote the maximal chain in $\M(\Y(\alpha))$ corresponding to standard tableau $T$ by $m_T$. Thus, $m_{T_m}=m$ and $T_{m_T}=T$.

Each linear extension of $P_{\alpha}$ gives an {\EL} of $\Y(\alpha)\cong J(P_{\alpha})$ as described in \cref{rmk:dist lat Sn EL labels}. Fixing a linear extension of $P_\alpha$ is the same as fixing a standard tableau $T\in ST_{\alpha}$. We will denote the {\EL} of $\Y(\alpha)$ induced by $T$ as $\lm_T$. In the following proposition, we observe that for tableau $Q\in ST_{\alpha}$, the label sequence $\lm_T(m_Q)$ can be read from the tableaux alone.

\begin{proposition} \label{prop:labelseqfromtableaux}
Let $T,Q\in ST_{\alpha}$ be standard tableaux of shape $\alpha$. Then the label sequence $\lm_T(m_Q)$ of the maximal chain $m_Q$ of $\Y(\alpha)$ is $\lm_T(m_Q) =(T(Q^1),T(Q^2),\dots, T(Q^n))$. Moreover, each label sequence occurs for exactly one maximal chain in $\Y(\alpha)$.
\end{proposition}

\begin{proof}
The box $Q^i$ is the box added to obtain the rank $i$ element of $m_Q$ from the rank $i-1$ element of $m_Q$. Now $T(Q^i)$ is the value of the box $Q^i$ under the linear extension defined by $T$, and so the label of the $i$th cover relation in $m_Q$. The uniqueness of label sequences is the same as for the linear extension {\EL}s of any finite distributive lattice.
\end{proof}

\begin{remark}\label{rmk:youngs lat lin ext cords tableau}
Each choice of standard tableau $T\in ST_{\alpha}$ defines a maximal chain descent order $\cord{\Y(\alpha)}{\lm_T}$. We may realize these maximal chain descent orders as partial orders on $ST_{\alpha}$. By \cref{thm:dist lat Sn cord isom weak ideal} each increase by a polygon move gives a cover relation in the corresponding maximal chain order induced by $\lm_T$. \cref{prop:labelseqfromtableaux} then implies that the cover relations in the maximal chain orders can be described by an operation on the tableaux themselves. 
\end{remark}

\begin{definition}\label{def:tabswap}
Let $T\in ST_{\alpha}$ be a standard tableau of shape $\alpha$ with $n$ boxes. For $1\leq i<j\leq n$, let $(i,j)T$ be the filling of $\alpha$ that is the same as $T$ except the entries $i$ and $j$ are switched. If $(i,j)T$ is also a standard tableau, then we call $(i,j)T$ the \textbf{ij tableau swap of $\mathbf{T}$}. Further, if $j=i+1$ above, we call $(i,i+1)T$ the \textbf{ith tableau swap of $\mathbf{T}$}.
\end{definition}

\begin{lemma}\label{lem:tabswapscordcovers}
Suppose $Q,R,T\in ST_{\alpha}$ are standard tableaux of shape $\alpha$ with $n$ boxes. Then $m_Q\precdot_{\lm_T} m_R$ in $\cord{\Y(\alpha)}{\lm_T}$ if and only if $R$ is the $i$th tableau swap of $Q$ and $T(Q^i)<T(Q^{i+1})$ for some $1\leq i \leq n-1$.
\end{lemma}

\begin{proof}
By\cref{thm:dist lat Sn cord isom weak ideal} (whose proof used \cref{thm:Sn cords weak lower intervals}), $m_Q\precdot_{\lm_T} m_R$ if and only if $\lm_T(m_R)$ is obtained from $\lm_T(m_Q)$ by transposing an ascent of $\lm_T(m_Q)$. Thus, by \cref{prop:labelseqfromtableaux}, $m_Q\precdot_{\lm_T} m_R$ if and only if \begin{align*}\lm_T(m_Q)=(T(Q^1),T(Q^2),\dots,& T(Q^i),T(Q^{i+1}),\dots, T(Q^n)) \\
\lm_T(m_R)=(T(Q^1),T(Q^2),\dots, & T(Q^{i+1}),T(Q^{i}),\dots, T(Q^n)) \end{align*}
with $ T(Q^i)<T(Q^{i+1})$. Hence, $m_Q\precdot_{\lm_T} m_R$ if and only if $R^j=Q^j$ for $j\neq i,i+1$, $R^i=Q^{i+1}$, and $R^{i+1}=Q^i$ which says that $R$ is the $i$th tableau swap of $Q$ since $R$ is a standard tableau.
\end{proof}

Thus, we have that $\cord{\Y(\alpha)}{\lm_{T}}$ is isomorphic to the poset on Young tableaux defined as the transitive closure of certain $i$th tableau swaps.

\begin{theorem}\label{thm:youngs lat cords tableau swaps}
For a standard tableau $T\in ST(\alpha)$ of shape $\alpha$, $\leq_T$ be the partial order on $ST(\alpha)$ defined as the reflexive and transitive closure of $Q\to (i,i+1)Q$ for any $Q\in ST(\alpha)$ such that the $i$th tableau swap $(i,i+1)Q$ from \cref{def:tabswap} is in $ST(\alpha)$ and $T(Q^i)<T(Q^{i+1})$. Then the map defined by $m\mapsto T_m$ is a poset isomorphism from $\cord{\Y(\alpha)}{\lm_{T}}$ to $(ST(\alpha),\leq_T)$. 
\end{theorem}

\begin{proof}
We observed that $m\mapsto T_m$ is a bijection from $\cord{\Y(\alpha)}{\lm_{T}}$ to $(ST(\alpha),\leq_T)$. Then since $\lm_T$ is polygon complete by \cref{thm:Sn EL is polygon strong} (or by \cref{cor:Sn inv ranked so cord satisfies inv ranked thm}), \cref{lem:tabswapscordcovers} implies that $m\to m'$ if and only if $T_m\to T_m'$. This proves the theorem.
\end{proof}

\begin{example}
\cref{fig:example of youngs lat cord} (a) shows the {\EL} of an interval in Young's lattice induced by the standard Young tableau $((1,4,6),(2,5),(3))$. \cref{fig:example of youngs lat cord} (b) contains the corresponding maximal chain descent order. \cref{fig:example of youngs lat cord} (c) shows the maximal chain descent order corresponding induced by the {\EL} from the standard tableau $((1,2,4),(3))$. The maximal chain descent orders are show as the partial orders on standard Young tableaux from \cref{thm:youngs lat cords tableau swaps}.

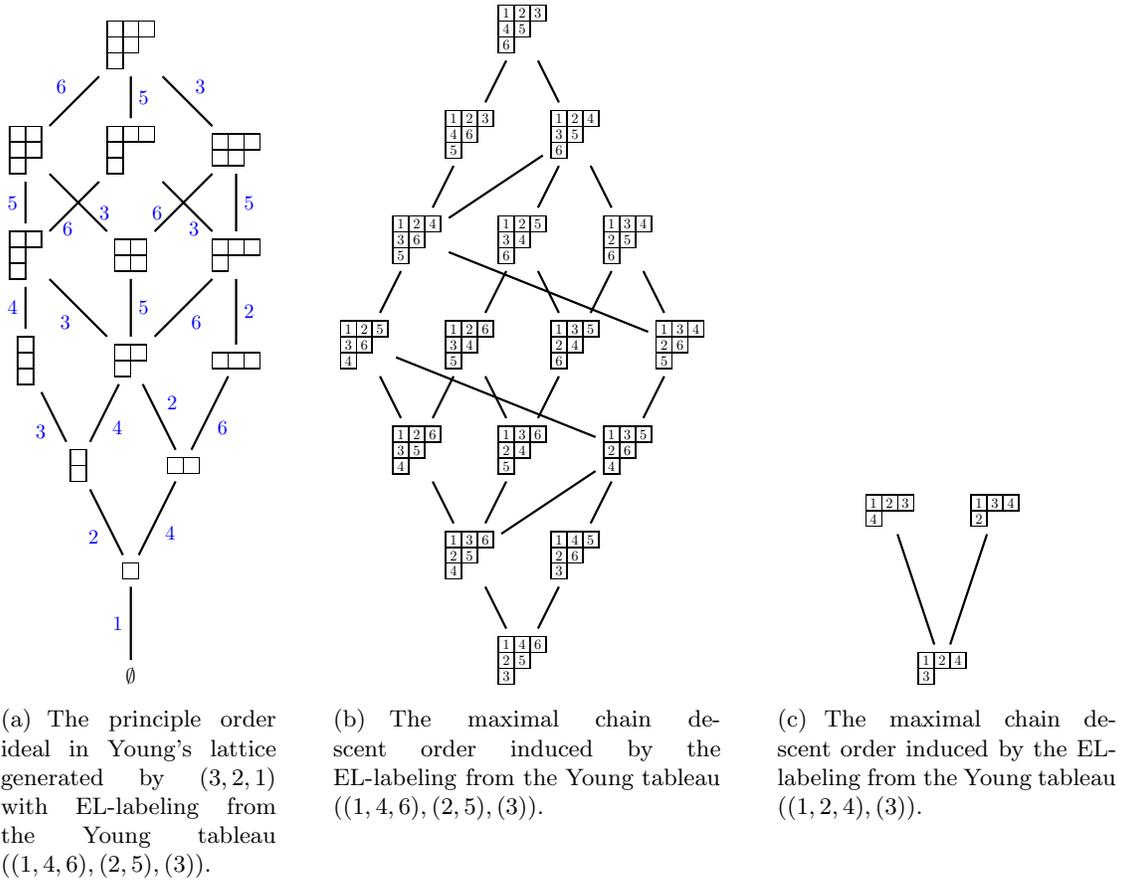
\begin{figure}[H]
    \centering
    \subfigure[The principle order ideal in Young's lattice generated by $(3,2,1)$ with {\EL} from the Young tableau $((1,4,6),(2,5),(3))$.]{\scalebox{0.7}{\begin{tikzpicture}[very thick]
  \node (0) at (0,0) {\scalebox{1}{$\emptyset$}};
  \node (1) at (0,2) {\scalebox{0.7}{\yng(1)}};  
  \node (11) at (-1,4) {\scalebox{0.7}{\yng(1,1)}};
  \node (2) at (1,4) {\scalebox{0.7}{\yng(2)}};
  \node (111) at (-2,6) {\scalebox{0.7}{\yng(1,1,1)}};
  \node (21) at (0,6) {\scalebox{0.7}{\yng(2,1)}};
  \node (3) at (2,6) {\scalebox{0.7}{\yng(3)}};
  \node (211) at (-2,8) {\scalebox{0.7}{\yng(2,1,1)}};
  \node (22) at (0,8) {\scalebox{0.7}{\yng(2,2)}};
  \node (31) at (2,8) {\scalebox{0.7}{\yng(3,1)}};
  \node (221) at (-2,10) {\scalebox{0.7}{\yng(2,2,1)}};
  \node (311) at (0,10) {\scalebox{0.7}{\yng(3,1,1)}};
  \node (32) at (2,10) {\scalebox{0.7}{\yng(3,2)}};
  \node (321) at (0,12) {\scalebox{0.7}{\yng(3,2,1)}};

  \node (l3) at (-0.5,8.8) {\textcolor{blue}{$3$}};
  \node (1l6) at (0.5,8.8) {\textcolor{blue}{$6$}};
  \node (2l6) at (-1.2,8.5) {\textcolor{blue}{$6$}};
  \node (l3) at (1.2,8.5) {\textcolor{blue}{$3$}};
 
  \draw (1) --node[below left,blue] {$2$} (11) --node[below left,blue] {$3$} (111) --node[left,blue] {$4$} (211) --node[left,blue] {$5$} (221) --node[above left,blue] {$6$} (321) --node[above right,blue] {$3$} (32) --node[right,blue] {$5$} (31) --node[right,blue] {$2$} (3) --node[below right,blue] {$6$} (2) --node[below right,blue] {$4$} (1);
  \draw (11) --node[below right,blue] {$4$} (21) --node[below left,blue] {$3$} (211) -- (311) -- (31) --node[below right,blue] {$6$} (21) --node[above right,blue] {$2$} (2);
  \draw (311) --node[right,blue] {$5$} (321);
  \draw (22) -- (221);
  \draw (22) -- (32);
  \draw (0) --node[left,blue] {$1$} (1);
  \draw (21) --node[right,blue] {$5$} (22);
\end{tikzpicture}} }
\hspace{5mm}
    \subfigure[The maximal chain descent order induced by the {\EL} from the Young tableau $((1,4,6),(2,5),(3))$.]{\scalebox{0.7}{\begin{tikzpicture}[very thick]
  \node (146253) at (0,0) {\scalebox{0.7}{\young(146,25,3)}};  
  \node (136254) at (-1,2) {\scalebox{0.7}{\young(136,25,4)}};
  \node (145263) at (1,2) {\scalebox{0.7}{\young(145,26,3)}};
  \node (126354) at (-2,4) {\scalebox{0.7}{\young(126,35,4)}};
  \node (136245) at (0,4) {\scalebox{0.7}{\young(136,24,5)}};
  \node (135264) at (2,4) {\scalebox{0.7}{\young(135,26,4)}};
  \node (125364) at (-3,6) {\scalebox{0.7}{\young(125,36,4)}};
  \node (126345) at (-1,6) {\scalebox{0.7}{\young(126,34,5)}};
  \node (135246) at (1,6) {\scalebox{0.7}{\young(135,24,6)}};
  \node (134265) at (3,6) {\scalebox{0.7}{\young(134,26,5)}};
  \node (124365) at (-2,8) {\scalebox{0.7}{\young(124,36,5)}};
  \node (125346) at (0,8) {\scalebox{0.7}{\young(125,34,6)}};
  \node (134256) at (2,8) {\scalebox{0.7}{\young(134,25,6)}};
  \node (123465) at (-1,10) {\scalebox{0.7}{\young(123,46,5)}};
  \node (124356) at (1,10) {\scalebox{0.7}{\young(124,35,6)}};
  \node (123456) at (0,12) {\scalebox{0.7}{\young(123,45,6)}};
  
  \draw (146253) -- (136254) -- (126354) -- (125364) -- (124365) -- (123465) -- (123456) -- (124356) -- (134256) -- (134265) -- (135264) -- (145263) -- (146253);
  \draw (136254) -- (136245) -- (126345) -- (125346) -- (124356);
  \draw (136245) -- (135246) -- (125346);
  \draw (124365) -- (124356);
  \draw (135246) -- (134256);
  \draw (135264) -- (125364);
  \draw (136254) -- (135264);
  \draw (134265) -- (124365);
  \draw (126354) -- (126345);
\end{tikzpicture}} }
\hspace{5mm}
\subfigure[The maximal chain descent order induced by the {\EL} from the Young tableau $((1,2,4),(3))$.]{\scalebox{0.7}{\begin{tikzpicture}[very thick]
  \node (1243) at (0,0) {\scalebox{0.7}{\young(124,3)}};  
  \node (1234) at (-1,3) {\scalebox{0.7}{\young(123,4)}};
  \node (1342) at (1,3) {\scalebox{0.7}{\young(134,2)}};

  \node (spacer1) at (-3,0) {};
  \node (spacer2) at (3,0) {};
  
  \draw (1243) -- (1234);
  \draw (1243) -- (1342);
\end{tikzpicture}} }
    \caption{Maximal chain descent orders induced by {\EL}s from two Young tableaux.}
    \label{fig:example of youngs lat cord}
\end{figure}
\end{example}

In the next proposition, we recall a special tableau called the row tableau. \cref{fig:youngtabexampls} shows the row tableau of shape $(3,2,1)$. We mention this tableau because it provides a connection between maximal chain descent orders and the generalized quotients of the symmetric group introduced by Bj{\"o}rner and Wachs in \cite{genquotcxtrgrpsbjornerwachs1988}. 

\begin{proposition}\label{prop:rowword} Let $\alpha$ be a Young diagram with $n$ boxes and let $\alpha_i$ be the length of the $i$th row of $\alpha$. Let $R_{\alpha}$ be the standard tableau of shape $\alpha$ obtained by labeling the first row of $\alpha$ by $1,2,\dots, \alpha_1$ increasing from left to right, labeling the second row by $\alpha_1+1,\alpha_1+2,\dots, \alpha_1+\alpha_2$, and so on. Then for any tableau of shape $T$, the row word of $T$ is $T(R_{\alpha}^1),T(R_{\alpha}^2),\dots, T(R_{\alpha}^n)$. 

The tableau $R_{\alpha}$ is called \textbf{the row tableau} of shape $\alpha$.
\end{proposition}

\begin{proof}
This is clear from \cref{prop:labelseqfromtableaux}.
\end{proof}

\end{subsection} \label{sec:bkgrndyoungslat}

In \cite{genquotcxtrgrpsbjornerwachs1988}, Bj{\"o}rner and Wachs introduce generalized quotients of Coxeter groups. In particular, they study the partial orders induced on these quotients by weak order and Bruhat order on the original Coxeter group. These quotients generalize the notion of quotients of Coxeter groups by parabolic subgroups which are particular choices of coset representatives of a parabolic subgroup. We follow Bj{\"o}rner and Wachs' notation and definitions which agree with the notation in \cref{sec:wkordsymmgroup} in type A. See \cite{bjornerbrenitcxgp} for general Coxeter groups.

Let $(W,S)$ be a Coxeter system. Let $l$ be the Coxeter length function for $(W,S)$. Subgroups of $W$ generated by a subset $J\subseteq S$, denoted $W_J$, are called parabolic subgroups. For $J\subseteq S$, \textbf{ordinary quotients} are the sets $W^J=\sett{w\in W}{l(ws)=l(w)+1 \quad \forall s\in J}$. The ordinary quotient $W^J$ intersects the left cosets of $W_J$ in their minimum length element. This is generalized in \cite{genquotcxtrgrpsbjornerwachs1988} as follows: 

\begin{definition}[Section 1, \cite{genquotcxtrgrpsbjornerwachs1988}]
For any subset $V\subseteq W$, let $$W/V=\sett{w\in W}{l(wv)=l(w)+l(v) \quad \forall v\in V}.$$ The set $W/V$ is called a \textbf{generalized quotient}.
\end{definition}

Restricting the (left) weak order on $W$ to the generalized quotient $W/V$ gives a partial order on $W/V$ which will be referred to as (left) \textbf{weak order}.

In \cite{genquotcxtrgrpsbjornerwachs1988} Section 7, they introduce a partial order on $ST_{\alpha}$ called \textbf{Left order} (the name coming from left weak order on the symmetric group). Left order is defined as the reflexive, transitive closure of the relation $Q<T$ if $T$ is the $i$th tableau swap of $Q$ and $i$ appears in a row above $i+1$ in $Q$. They show that that Left order is isomorphic to a generalized quotient of the symmetric group. 

\begin{theorem}[\cite{genquotcxtrgrpsbjornerwachs1988} Theorem 7.2]\label{thm:left weak tab order gen quotient} 
Let $\alpha$ have $n$ boxes and let $w(ST_{\alpha})$ be the set of row words of all standard tableau of shape $\alpha$. Then $w(ST_{\alpha})$ is a generalized quotient of $S_n$. Moreover, the map $T\mapsto w(T)$ is a poset isomorphism from Left order on $ST_{\alpha}$ to weak order on the generalized quotient $w(ST_{\alpha})$. 
\end{theorem}

Choosing the row tableau of a given shape, the induced maximal chain descent order is isomorphic to the Left order on $ST_{\alpha}$, and thus isomorphic to a generalized quotient of the symmetric group.

\begin{theorem}\label{thm:rowtabcordleftsame}
For $\alpha$ with $n$ boxes, the maximal chain order $\cord{\Y(\alpha)}{\lm_{R_{\alpha}}}$ with $R_{\alpha}$ the row tableau of shape $\alpha$ is isomorphic to Left order on $ST_{\alpha}$. Hence, $\cord{\Y(\alpha)}{\lm_{R_{\alpha}}}$ is also isomorphic to weak order on the generalized quotient $w(ST_{\alpha})$ of the symmetric group $S_n$.
\end{theorem}

\begin{proof}
By \cref{prop:labelseqfromtableaux} and \cref{lem:tabswapscordcovers}, the cover relations of $\cord{\Y(\alpha)}{\lm_{R_{\alpha}}}$ and Left order both correspond to $i$th tableau swaps $(i,i+1)T$ where $i$ appears in a row above $i+1$ in $T$. Alternatively, by \cref{prop:labelseqfromtableaux} and the definition of $\cord{\Y(\alpha)}{\lm_{R_{\alpha}}}$, $\cord{\Y(\alpha)}{\lm_{R_{\alpha}}}$ and Left order are both the reflexive, transitive closure of the $i$th tableau swaps on $ST_{\alpha}$. Then by \cref{thm:left weak tab order gen quotient}, $\cord{\Y(\alpha)}{\lm_{R_{\alpha}}}$ is isomorphic to weak order on the generalized quotient $w(ST_{\alpha})$.
\end{proof}

\begin{remark}\label{rmk:generalize generalized quotient result}
We make two observations. First, we note that the label sequences with respect to $\lm_{R_{\alpha}}$ in \cref{thm:rowtabcordleftsame} are not the row words of the tableaux corresponding to the maximal chains. Thus, the isomorphism to the generalized quotient in \cref{thm:rowtabcordleftsame} which is induced by \cref{thm:left weak tab order gen quotient} is not the same isomorphism given by \cref{thm:dist lat Sn cord isom weak ideal}. Second, we observe that Theorem 7.6 in \cite{genquotcxtrgrpsbjornerwachs1988} can easily be used to extend \cref{thm:rowtabcordleftsame} and show which linear extensions of a finite poset $P$ induce EL-labelings of $J(P)$ that give maximal chain orders isomorphic to left order on generalised quotients of the symmetric group.
\end{remark}

\begin{subsection}{The Partition Lattice}\label{sec:the partition lattice}
Another well known $M$-chain {\EL} of a supersolvable lattice is the following {\EL} of the partition lattice $\Pi_{n+1}$. Let $\Pi_{n+1}$ be the collection of set partitions of $[n+1]$ ordered by refinement. The {\EL} $\lm$ due to Gessel and appearing in \cite{shellblposets} is given as follows: if $x\lessdot y$ in $\Pi_{n+1}$, then $y$ is obtained from $x$ by merging exactly two blocks $B_1$ and $B_2$ of $x$ and $\lm(x,y)=\max(\set{\min B_1, \min B_2})$. We call $\lm$ the \textbf{max-min {\EL}}.

We prove that a class of labeled binary trees given in \cref{def:unordered increasing full binary trees} are in bijection with the maximal chains of the partition lattice. We then prove that the maximal chain descent order $\cord{{\Pi_{n+1}}}{\lm}$ is isomorphic to a naturally described poset on these trees. We also note that they are distinct from the trees used by Wachs in \cite{oncohomofpartitionlatwachs1998} to study the (co)homology of the partition lattice.

\begin{definition}\label{def:unordered increasing full binary trees}
For a positive integer $n$, let $\mathbf{PT(n)}$ denote the set of rooted, unordered, decreasing, full binary trees with $2n$ edges, vertices labeled by $\set{1,2,\dots,2n,2n+1}$, and leaf set $\set{1,2,\dots,n,n+1}$. Rooted means there is a distinguished vertex. Full binary means that each non-leaf vertex has exactly two children. Decreasing means that all of the descendants of a vertex have smaller labels than their ancestor. Unordered means we do not distinguish between the two possible orders of the children of an internal vertex, that is, we may assume that when drawn in the plane, the smaller of the two children of an internal vertex is drawn on the left and the larger is drawn on the right.

For each integer $0\leq k\leq n$, let $\mathbf{FPT(n,k)}$ denote the set of forests of rooted, unordered, decreasing, full binary trees with vertices labeled by $\set{1,\dots, n,n+1,\dots,n+1+k}$, leaf set \\$\set{1,2,\dots,n,n+1}$, $n+1-k$ components, and $2n-2(n-k)=2k$ total edges.

For a labeled tree $T$, let $\mathbf{L(T)}$ be its leaf set, i.e. the set of labels of leaves of $T$. Denote the full subtree of $T$ rooted at the vertex labeled $i$ by $\mathbf{T^i}$. If $i$ is the label of an internal vertex of $T$, let $\mathbf{T^i_1}$ and $\mathbf{T^i_2}$ denote the full subtrees of $T$ rooted at the two children of $i$.
\end{definition}

\begin{remark}\label{rmk:partition lattice forest with one tree is tree}
Notice that when $k=n$, $FPT(n,n)=PT(n)$.
\end{remark}

\begin{remark}\label{rmk:diff than wachs cohom trees}
The trees of \cref{def:unordered increasing full binary trees} are distinct from the trees used in \cite{oncohomofpartitionlatwachs1998} to compute the (co)homology of the partition lattice. The trees in \cite{oncohomofpartitionlatwachs1998} are leaf labeled, full binary trees with leaves labeled by $[n+1]$ while internal vertices are not labeled. The same underlying set of leaf labeled, full binary trees appears in \cref{def:unordered increasing full binary trees}, but we also label the internal vertices. In \cite{oncohomofpartitionlatwachs1998},the internal vertices are traversed in post order which describes the maximal chains of $\Pi_{n+1}$ up to cohomology relations, and so bijects the leaf labeled trees with a cohomology basis. In our trees, the internal vertices are traversed based on the vertex labels which sometimes disagrees with post order. The extra traversals of internal vertices allow trees of $PT(n)$ to biject with maximal chains instead of a cohomology basis.
\end{remark}

\begin{definition}\label{def:partition lattice chain from a tree}
Given a forest $F\in FPT(n,k)$, define a saturated chain $\mathbf{c(F)}:\zh=x_0\lessdot x_1\lessdot \dots \lessdot x_k$ in the partition lattice $\Pi_{n+1}$ beginning at the unique minimal element $\zh =1|2|\dots |n+1$ as follows: the blocks of $x_i$ are precisely the leaf sets of the components of $F$ restricted to vertices labeled at most $n+1+i$.
\end{definition}

\begin{remark}\label{rmk:partition lattice max chain from a tree}
For $F\in FPT(n,n)=PT(n)$, $c(F)$ is a maximal chain in $\Pi_{n+1}$. For any $k$ and any $F\in FTP(n,k)$, the top element of the saturated chain $c(F)$ has blocks which are precisely the leaf sets of the trees in $F$.
\end{remark}

\begin{example}\label{ex:partion lattice tree example}
\cref{fig:rooted unordered decreasing binary tree} contains an example of a rooted, unordered, decreasing, full binary tree in $PT(3)$, as well as examples of forests in $FPT(3,k)$ for $0\leq k\leq 3$. Let $F_0$ be the forest in \cref{fig:rooted unordered decreasing binary tree} (b), let $F_1$ be the forest in \cref{fig:rooted unordered decreasing binary tree} (c), let $F_2$ be the forest in \cref{fig:rooted unordered decreasing binary tree} (d), and let $T$ be the tree in \cref{fig:rooted unordered decreasing binary tree} (a). Then applying \cref{def:partition lattice chain from a tree} to each of these forests we have $c(F_0)=1|2|3|4$, $c(F_1)=1|2|3|4\lessdot 14|2|3$, $c(F_2)=1|2|3|4\lessdot 14|2|3\lessdot 124|3$, and $c(T)=1|2|3|4\lessdot 14|2|3\lessdot 124|3\lessdot 1234$.

\begin{figure}[H]
\centering
\subfigure[{$T\in PT(3)=FPT(3,3)$}]{
\scalebox{0.6}{\begin{tikzpicture}[very thick]
\node [style={draw,circle}]{7}
    child{ node[style={draw,circle},xshift=0,yshift=0] {3}
    }
    child{ node[style={draw,circle},yshift=0,xshift=0] {6}
        child{ node[style={draw,circle},xshift=0,yshift=0] {2}
        }
        child{ node[style={draw,circle},xshift=0,yshift=0] {5}
            child{ node[style={draw,circle},xshift=0,yshift=0] {1}
            } 
            child{ node[style={draw,circle},xshift=0,yshift=0] {4}
            }
        }
    };
\end{tikzpicture}}
}
\hspace{5mm}
\subfigure[{The unique $F_0\in FPT(3,0)$}]{
\scalebox{0.6}{\begin{tikzpicture}[very thick]
\node [style={draw,circle}] (a) at (0,0) {1};
\node [style={draw,circle}] (b) at (1,0) {2};
\node [style={draw,circle}] (c) at (2,0) {3};
\node [style={draw,circle}] (d) at (3,0) {4};
\end{tikzpicture}}
}
\hspace{5mm}
\subfigure[{$F_1\in FPT(3,1)$}]{
\scalebox{0.6}{\begin{tikzpicture}[very thick]
\node [style={draw,circle}] (a) at (0,0) {2};
\node [style={draw,circle}] (b) at (1,0) {3};
\node [style={draw,circle}] (c) at (3,1.5) {5}
    child{ node[style={draw,circle},xshift=0,yshift=0] {1}
    } 
    child{ node[style={draw,circle},xshift=0,yshift=0] {4}
    };
\end{tikzpicture}}
}
\hspace{5mm}
\subfigure[{$F_2\in FPT(3,2)$}]{
\scalebox{0.6}{\begin{tikzpicture}[very thick]
\node [style={draw,circle}] (a) at (0,0) {3};
\node [style={draw,circle}] (b) at (2,3) {6}
    child{ node[style={draw,circle},xshift=0,yshift=0] {2}
    }
    child{ node[style={draw,circle},xshift=0,yshift=0] {5}
        child{ node[style={draw,circle},xshift=0,yshift=0] {1}
        } 
        child{ node[style={draw,circle},xshift=0,yshift=0] {4}
        }
    };
\end{tikzpicture}}
}
    \caption{A rooted, unordered, decreasing, full binary tree in $PT(3)$ and forests in $FPT(3,k)$ for $0\leq k\leq 3$.}
    \label{fig:rooted unordered decreasing binary tree}
\end{figure}
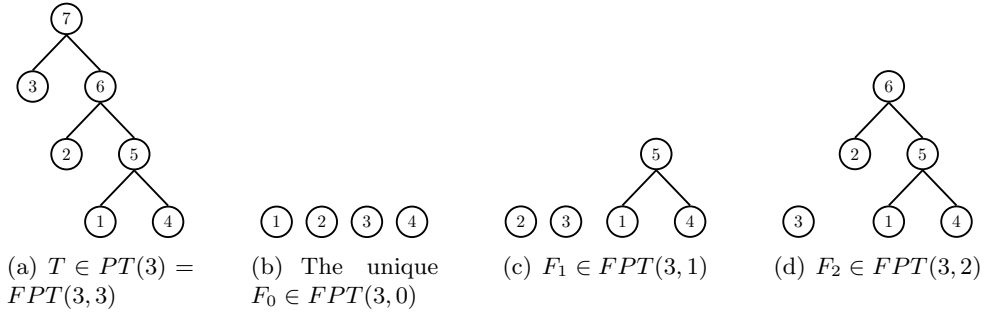

\end{example}

\begin{theorem}\label{thm:initial segs partition lattice chains in bijection with rooted unordered decreasing forests}
The map $c$ of \cref{def:partition lattice chain from a tree} is a bijection from $FPT(n,k)$ to the saturated chains of $\Pi_{n+1}$ which begin with the unique minimal element $1|2|\dots|n+1$ and have length $k$. In particular, $c:PT(n)\to \M(\Pi_{n+1})$ is a bijection.
\end{theorem}

\begin{proof}
We will prove \cref{thm:initial segs partition lattice chains in bijection with rooted unordered decreasing forests} by induction on $k$. As a base case we have $k=0$. The forest $F_0$ with $n+1$ disconnected vertices labeled $\set{1,\dots,n+1}$ is the unique forest in $FPT(n,0)$. We have $c(F_0)=1|2|\dots |n+1$ which is the unique saturated chain of $\Pi_{n+1}$ which begins with $1|2|\dots |n+1$ and has length $0$. Thus, $c$ is a bijection when $k=0$. 

Assume $c$ is a bijection when $k=l$ for some $l\geq 0$. Assume $k=l+1$. Observe that for each forest $F\in FTP(n,l+1)$, we obtain a forest $F'\in FTP(n,l)$ by deleting the internal vertex labeled $n+1+l+1$. Further, observe that the saturated chain $c(F')$ is obtained by restricting the saturated chain $c(F)$ to ranks $\set{0,1,\dots,l}$, i.e. by deleting the element of $c(F)$ at rank $l+1$. We first show that $c$ is injective. Suppose $c(F_1)=c(F_2)$ for forests $F_1,F_2\in FTP(n,l+1)$. Let $x_l\lessdot x_{l+1}$ be the elements of $c(F_1)=c(F_2)$ of rank $l$ and $l+1$, respectively. So, $c(F_1)^{x_{l}}=c(F_2)^{x_{l}}$. Thus, $c(F_1')=c(F_2')$ by our previous observation. Now by the induction hypothesis $c$ is injective on $FTP(n,l)$, so $F_1'=F_2'$. Exactly two blocks of $x_l$ are merged to form $x_{l+1}$. By \cref{rmk:partition lattice max chain from a tree} these two blocks are the leaf sets of two of the trees, call them $T_1$ and $T_2$, in $F_1'=F_2'$. Then the only way to form a forest $\tilde{F}\in FTP(n,l+1)$ from $F_1'=F_2'$ with $x_l$ as the top element of $c(\tilde{F})$ is to add a vertex labeled $n+1+l+1$ and make its two children the roots of $T_1$ and $T_2$. Thus, $F_1=F_2=\tilde{F}$, so $c$ is injective on $FTP(n,l+1)$. 

Next we show that $c$ is surjective from $FPT(n,l+1)$ to the saturated chains of $\Pi_{n+1}$ which begin with the unique minimal element $1|2|\dots|n+1$ and have length $l+1$. Let $c_1:1|2|\dots |n+1=y_0\lessdot y_1 \lessdot \dots \lessdot y_l\lessdot y_{l+1}$ be a saturated chain in $\Pi_{n+1}$. By the induction hypothesis $c$ is surjective from $FTP(n,l)$ to the saturated chains of $\Pi_{n+1}$ which begin with the unique minimal element $1|2|\dots|n+1$ and have length $l$. Thus, there is a forest $F\in FTP(n,l)$ such that $c(F)=c_1^{y_{l}}$. By \cref{rmk:partition lattice max chain from a tree} there are two trees $T_1$ and $T_2$ in $F$ whose leaf sets are the two blocks of $y_l$ which are merged to form $y_{l+1}$. Let $F_1$ be the forest formed from $F$ by adding a vertex labeled $n+1+l+1$ and make its two children the roots of $T_1$ and $T_2$. It is clear that $F_1\in FTP(n,l+1)$ since $F\in FTP(n,l)$ and we added the internal vertex $n+1+l+1$ and reduced the number of connected components by exactly one while leaving all other labels the same. By construction the top element of $c(F_1)$ is $y_{l+1}$ while $c(F_1)\setminus\set{y_{l+1}}=c_1^{y_l}$, so $c(F_1)=c_1$. Hence, $c$ is surjective from $FPT(n,l+1)$ to the saturated chains of $\Pi_{n+1}$ which begin with the unique minimal element $1|2|\dots|n+1$ and have length $l+1$. Therefore, the theorem holds by induction.
\end{proof}

\begin{proposition}\label{prop:partition lattice label seqs from trees}
Let $T$ be a tree with $T\in TP(n)$. Let $\lm$ be the max-min {\EL} of the partition lattice $\Pi_{n+1}$. Then for each $1\leq i \leq n$, the $i$th entry in the label sequence $\lm(c(T))$ is the maximum of the minima of the leaf sets of $T^{n+1+i}_1$ and $T^{n+1+i}_2$, that is $\lm(c(T))_i=\max\set{\min(L(T^{n+1+i}_1)),\min(L(T^{n+1+i}_2))}$.
\end{proposition}

\begin{proof}
By definition of $c(T)$ the two blocks merged to form the rank $i$ element of $c(T)$ from the rank $(i-1)$ element of $c(T)$ are exactly $L(T^{n+1+i}_1)$ and $L(T^{n+1+i}_2)$. Then the proposition follows by definition of $\lm$.
\end{proof}

Next we see that two maximal chains in the partition lattice differing by a polygon can be described by two simple operations on the corresponding trees, one for each of the two types of rank two intervals in the partition lattice.

\begin{lemma}\label{lem:partition lattice cord covers from trees}
Let $T,S\in TP(n)$ be trees. Then $c(T)$ and $c(S)$ differ by a polygon at rank $i$ in the partition lattice $\Pi_{n+1}$ if and only if exactly one of the two following conditions holds: \begin{itemize}
    \item [(i)] $S$ is obtained from $T$ by swapping the labels $n+1+i$ and $n+1+i+1$, or

    \item [(ii)] $n+1+i$ is a child of $n+1+i+1$ in $T$ and $S$ is obtained from $T$ by swapping the full subtree of $T$ whose root is the child of $n+1+i+1$ which is not $n+1+i$ and a full subtree of $T$ whose root is a child of $n+1+i$.
\end{itemize}
\end{lemma}

\begin{proof}
First, we observe that conditions (i) and (ii) are mutually exclusive because if $n+1+i$ is a child of $n+1+i+1$, then swapping the labels $n+1+i$ and $n+1+i+1$ results in a tree which is not decreasing. We now prove the forward direction. We have $c(T):1|2|\dots |n+1=x_0\lessdot x_1\lessdot \dots \lessdot x_{i-1}\lessdot x_i \lessdot x_{i+1}\lessdot \dots \lessdot x_n=12\dots n+1$ and $c(S):1|2|\dots |n+1=x_0\lessdot x_1\lessdot \dots \lessdot x_{i-1}\lessdot x'_i \lessdot x_{i+1}\lessdot \dots \lessdot x_n=12\dots n+1$ for some $x'_i\neq x_i$. There are two cases: either (a) $x_{i-1}$ contains blocks $B_1,B_2,B_3$, and $B_4$; $x_i$ contains blocks $B_1\cup B_2$, $B_3$, and $B_4$; and $x_{i+1}$ contains blocks $B_1\cup B_2$ and $B_3\cup B_4$, or (b) $x_{i-1}$ contains blocks $B_1,B_2$, and $B_3$; $x_i$ contains blocks $B_1\cup B_2$ and $B_3$; and $x_{i+1}$ contains blocks $B_1\cup B_2\cup B_3$.

We show that these two cases precisely give rise to conditions (i) and (ii) of the theorem, respectively. 

In case (a), we have the that $x'_i$ is obtained from $x_{i-1}$ by merging blocks $B_3$ and $B_4$
and $x_{i+1}$ is formed from $x'_i$ by merging blocks $B_1$ and $B_1$. Thus, swapping the labels $n+1+i$ and $n+1+i+1$ in $T$ gives a tree $\tilde{S}\in TP(n)$ with $c(\tilde{S})=c(S)$. Then by \cref{thm:initial segs partition lattice chains in bijection with rooted unordered decreasing forests}, $\tilde{S}=S$. Thus, condition (i) holds. 

Assume we are in case (b). By definition of $c$, $T$ restricted to labels at most $n+1+i-1$ contains connected components $T_1$, $T_2$, and $T_3$ with leaf sets $L(T_1)=B_1$, $L(T_2)=B_2$, and $L(T_3)=B_3$. Further, $T_1$, $T_2$, and $T_3$ are full subtrees of $T$. Also, by definition of $c$, $n+1+i$ is a child of $n+1+i+1$ and the root of $T_3$ is a child of $n+1+i+1$ since $B_1$ and $B_2$ are merged to form $x_{i}$ from $x_{i-1}$ and $B_3$ is merged with $B_1\cup B_2$ to form $x_{i+1}$ from $x_i$. Now $x'_i$ is formed from $x_{i-1}$ either by merging $B_1$ and $B_3$ or by merging $B_2$ and $B_3$, then $x_{i+1}$ is formed from $x'_i$ by merging $B_2$ with $B_1\cup B_3$ or by merging $B_1$ with $B_2\cup B_3$, respectively. Thus, in the first case, swapping the subtrees $T_2$ and $T_3$ results in a tree $\tilde{S}\in TP(n)$ with $c(\tilde{S})=c(S)$. In the second case, swapping the subtrees $T_1$ and $T_3$ results in a tree $\tilde{S}\in TP(n)$ with $c(\tilde{S})=c(S)$. Either way, \cref{thm:initial segs partition lattice chains in bijection with rooted unordered decreasing forests} implies $\tilde{S}=S$. Hence, condition (ii) holds which completes the proof of the forward direction.

For the backward direction, if condition (i) is satisified, then $c(T)$ and $c(S)$ differ by a polygon at rank $i$ and the interval $[x_{i-1},x_{i+1}]$ is of type (a) above. If condition (ii) is satisified, then $c(T)$ and $c(S)$ differ by a polygon at rank $i$ and the interval $[x_{i-1},x_{i+1}]$ is of type (b). This completes the proof.
\end{proof}

The tree operations of \cref{lem:partition lattice cord covers from trees} give rise to the following partial order on $TP(n)$ which we then show is isomorphic to $\cord{\Pi_{n+1}}{\lm}$.

\begin{definition}\label{def:poset on partition lattice chain trees}
Define a partial order on the trees $PT(n)$ as follows: let $\preceq$ be the reflexive, transitive closure of $T\rightharpoonup S$ if $S$ is obtained from $T$ by either condition (i) or (ii) in \cref{lem:partition lattice cord covers from trees} and $\max\set{\min(L(T^{n+1+i}_1)),\min(L(T^{n+1+i}_2))}<\max\set{\min(L(T^{n+1+i+1}_1)),\min(L(T^{n+1+i+1}_2))}$.
\end{definition}

\begin{theorem}\label{thm:partition cord and tree poset isom}
Let $\lm$ be the max-min {\EL} of $\Pi_{n+1}$. Then the map $c: (PT(n),\preceq) \to \cord{{\Pi_{n+1}}}{\lm}$ is a poset isomorphism. Moreover, the cover relations of $(PT(n),\preceq)$ are precisely given by $T\rightharpoonup S$ for $S,T\in PT(n)$ as given in \cref{def:poset on partition lattice chain trees}.
\end{theorem}

\begin{proof}
The map $c:(PT(n),\preceq) \to \cord{{\Pi_{n+1}}}{\lm}$ is a bijection by \cref{thm:initial segs partition lattice chains in bijection with rooted unordered decreasing forests}. By \cref{def:poset on partition lattice chain trees} and \cref{lem:partition lattice cord covers from trees}, $m\to m'$ if and only if $c^{-1}(m)\rightharpoonup c^{-1}(m')$. Thus, $c: (PT(n),\preceq)\to \cord{{\Pi_{n+1}}}{\lm}$ is a poset isomorphism. Lastly, by \cref{thm:Sn EL is polygon strong} (alternatively, by \cref{cor:Sn inv ranked so cord satisfies inv ranked thm}) the cover relations of $(PT(n),\preceq)$ are precisely given by $T\rightharpoonup S$ for $S,T\in PT(n)$ since $\lm$ is an $S_n$ {\EL}.
\end{proof}

As a corollary, \cref{thm:Sn cords weak lower intervals} and \cref{thm:partition cord and tree poset isom} imply that intervals in $(PT(n),\preceq)$ are isomorphic to intervals in the weak order on $S_n$.

\begin{corollary}\label{cor:ints in tree poset isom weak ints}
Every interval in the maximal chain descent order $\cord{{\Pi_{n+1}}}{\lm}$ and 
every interval in the poset $(PT(n),\preceq)$ from \cref{def:poset on partition lattice chain trees} is isomorphic to some interval in weak order on the symmetric group $S_n$.
\end{corollary}

\begin{example}\label{ex:Pi4 Sn tree poset}
The partition lattice $\Pi_4$ with the max-min {\EL} is shown in \cref{fig:Pi4 with min max EL}. The induced maximal chain descent order is pictured in \cref{fig:Pi4 min max cord} and illustrates \cref{thm:partition cord and tree poset isom} and \cref{cor:ints in tree poset isom weak ints}.

\newsavebox{\motthf}
\sbox{\motthf}{
   \scalebox{0.4}{\begin{tikzpicture}[very thick]
\node [style={draw,circle}]{7}
    child{ node[style={draw,circle},xshift=0,yshift=0] {4}
    }
    child{ node[style={draw,circle},yshift=0,xshift=0] {6}
        child{ node[style={draw,circle},xshift=0,yshift=0] {3}
        }
        child{ node[style={draw,circle},xshift=0,yshift=0] {5}
            child{ node[style={draw,circle},xshift=0,yshift=0] {1}
            } 
            child{ node[style={draw,circle},xshift=0,yshift=0] {2}
            }
        }
    };
\end{tikzpicture}}
 }

\newsavebox{\mothtf}
\sbox{\mothtf}{
   \scalebox{0.4}{\begin{tikzpicture}[very thick]
\node [style={draw,circle}]{7}
    child{ node[style={draw,circle},xshift=0,yshift=0] {4}
    }
    child{ node[style={draw,circle},yshift=0,xshift=0] {6}
        child{ node[style={draw,circle},xshift=0,yshift=0] {2}
        }
        child{ node[style={draw,circle},xshift=0,yshift=0] {5}
            child{ node[style={draw,circle},xshift=0,yshift=0] {1}
            } 
            child{ node[style={draw,circle},xshift=0,yshift=0] {3}
            }
        }
    };
\end{tikzpicture}}
 }

\newsavebox{\motfth}
\sbox{\motfth}{
   \scalebox{0.4}{\begin{tikzpicture}[very thick]
\node [style={draw,circle}]{7}
    child{ node[style={draw,circle},xshift=0,yshift=0] {3}
    }
    child{ node[style={draw,circle},yshift=0,xshift=0] {6}
        child{ node[style={draw,circle},xshift=0,yshift=0] {4}
        }
        child{ node[style={draw,circle},xshift=0,yshift=0] {5}
            child{ node[style={draw,circle},xshift=0,yshift=0] {1}
            } 
            child{ node[style={draw,circle},xshift=0,yshift=0] {2}
            }
        }
    };
\end{tikzpicture}}
 }

\newsavebox{\motuthf}
\sbox{\motuthf}{
   \scalebox{0.4}{\begin{tikzpicture}[very thick]
\node [style={draw,circle}]{7}
    child{ node[style={draw,circle},xshift=-15,yshift=0] {5}
            child{ node[style={draw,circle},xshift=0,yshift=0] {1}
            } 
            child{ node[style={draw,circle},xshift=0,yshift=0] {2}
            }
        }
    child{ node[style={draw,circle},xshift=15,yshift=0] {6}
        child{ node[style={draw,circle},xshift=0,yshift=0] {3}
        }
        child{ node[style={draw,circle},xshift=0,yshift=0] {4}
    }
    };
\end{tikzpicture}}
 }

\newsavebox{\mtthof}
\sbox{\mtthof}{
   \scalebox{0.4}{\begin{tikzpicture}[very thick]
\node [style={draw,circle}]{7}
    child{ node[style={draw,circle},xshift=0,yshift=0] {4}
    }
    child{ node[style={draw,circle},yshift=0,xshift=0] {6}
        child{ node[style={draw,circle},xshift=0,yshift=0] {1}
        }
        child{ node[style={draw,circle},xshift=0,yshift=0] {5}
            child{ node[style={draw,circle},xshift=0,yshift=0] {2}
            } 
            child{ node[style={draw,circle},xshift=0,yshift=0] {3}
            }
        }
    };
\end{tikzpicture}}
 }

\newsavebox{\mothft}
\sbox{\mothft}{
   \scalebox{0.4}{\begin{tikzpicture}[very thick]
\node [style={draw,circle}]{7}
    child{ node[style={draw,circle},xshift=0,yshift=0] {2}
    }
    child{ node[style={draw,circle},yshift=0,xshift=0] {6}
        child{ node[style={draw,circle},xshift=0,yshift=0] {4}
        }
        child{ node[style={draw,circle},xshift=0,yshift=0] {5}
            child{ node[style={draw,circle},xshift=0,yshift=0] {1}
            } 
            child{ node[style={draw,circle},xshift=0,yshift=0] {3}
            }
        }
    };
\end{tikzpicture}}
 }
 
\newsavebox{\mothutf}
\sbox{\mothutf}{
   \scalebox{0.4}{\begin{tikzpicture}[very thick]
\node [style={draw,circle}]{7}
    child{ node[style={draw,circle},xshift=-15,yshift=0] {5}
            child{ node[style={draw,circle},xshift=0,yshift=0] {2}
            } 
            child{ node[style={draw,circle},xshift=0,yshift=0] {4}
            }
        }
    child{ node[style={draw,circle},xshift=15,yshift=0] {6}
        child{ node[style={draw,circle},xshift=0,yshift=0] {1}
        }
        child{ node[style={draw,circle},xshift=0,yshift=0] {3}
    }
    };
\end{tikzpicture}}
 }

\newsavebox{\moftth}
\sbox{\moftth}{
   \scalebox{0.4}{\begin{tikzpicture}[very thick]
\node [style={draw,circle}]{7}
    child{ node[style={draw,circle},xshift=0,yshift=0] {3}
    }
    child{ node[style={draw,circle},yshift=0,xshift=0] {6}
        child{ node[style={draw,circle},xshift=0,yshift=0] {2}
        }
        child{ node[style={draw,circle},xshift=0,yshift=0] {5}
            child{ node[style={draw,circle},xshift=0,yshift=0] {1}
            } 
            child{ node[style={draw,circle},xshift=0,yshift=0] {4}
            }
        }
    };
\end{tikzpicture}}
 }

\newsavebox{\mtfoth}
\sbox{\mtfoth}{
   \scalebox{0.4}{\begin{tikzpicture}[very thick]
\node [style={draw,circle}]{7}
    child{ node[style={draw,circle},xshift=0,yshift=0] {3}
    }
    child{ node[style={draw,circle},yshift=0,xshift=0] {6}
        child{ node[style={draw,circle},xshift=0,yshift=0] {1}
        }
        child{ node[style={draw,circle},xshift=0,yshift=0] {5}
            child{ node[style={draw,circle},xshift=0,yshift=0] {2}
            } 
            child{ node[style={draw,circle},xshift=0,yshift=0] {4}
            }
        }
    };
\end{tikzpicture}}
 }

\newsavebox{\mthfuot}
\sbox{\mthfuot}{
   \scalebox{0.4}{\begin{tikzpicture}[very thick]
\node [style={draw,circle}]{7}
    child{ node[style={draw,circle},xshift=-15,yshift=0] {5}
            child{ node[style={draw,circle},xshift=0,yshift=0] {3}
            } 
            child{ node[style={draw,circle},xshift=0,yshift=0] {4}
            }
        }
    child{ node[style={draw,circle},xshift=15,yshift=0] {6}
        child{ node[style={draw,circle},xshift=0,yshift=0] {1}
        }
        child{ node[style={draw,circle},xshift=0,yshift=0] {2}
    }
    };
\end{tikzpicture}}
 }

\newsavebox{\mtthfo}
\sbox{\mtthfo}{
   \scalebox{0.4}{\begin{tikzpicture}[very thick]
\node [style={draw,circle}]{7}
    child{ node[style={draw,circle},xshift=0,yshift=0] {1}
    }
    child{ node[style={draw,circle},yshift=0,xshift=0] {6}
        child{ node[style={draw,circle},xshift=0,yshift=0] {4}
        }
        child{ node[style={draw,circle},xshift=0,yshift=0] {5}
            child{ node[style={draw,circle},xshift=0,yshift=0] {2}
            } 
            child{ node[style={draw,circle},xshift=0,yshift=0] {3}
            }
        }
    };
\end{tikzpicture}}
 }

\newsavebox{\mtthuof}
\sbox{\mtthuof}{
   \scalebox{0.4}{\begin{tikzpicture}[very thick]
\node [style={draw,circle}]{7}
    child{ node[style={draw,circle},xshift=-15,yshift=0] {5}
            child{ node[style={draw,circle},xshift=0,yshift=0] {2}
            } 
            child{ node[style={draw,circle},xshift=0,yshift=0] {3}
            }
        }
    child{ node[style={draw,circle},xshift=15,yshift=0] {6}
        child{ node[style={draw,circle},xshift=0,yshift=0] {1}
        }
        child{ node[style={draw,circle},xshift=0,yshift=0] {4}
    }
    };
\end{tikzpicture}}
 }

\newsavebox{\moftht}
\sbox{\moftht}{
   \scalebox{0.4}{\begin{tikzpicture}[very thick]
\node [style={draw,circle}]{7}
    child{ node[style={draw,circle},xshift=0,yshift=0] {2}
    }
    child{ node[style={draw,circle},yshift=0,xshift=0] {6}
        child{ node[style={draw,circle},xshift=0,yshift=0] {3}
        }
        child{ node[style={draw,circle},xshift=0,yshift=0] {5}
            child{ node[style={draw,circle},xshift=0,yshift=0] {1}
            } 
            child{ node[style={draw,circle},xshift=0,yshift=0] {4}
            }
        }
    };
\end{tikzpicture}}
 }

\newsavebox{\mtfuoth}
\sbox{\mtfuoth}{
   \scalebox{0.4}{\begin{tikzpicture}[very thick]
\node [style={draw,circle}]{7}
    child{ node[style={draw,circle},xshift=-15,yshift=0] {5}
            child{ node[style={draw,circle},xshift=0,yshift=0] {2}
            } 
            child{ node[style={draw,circle},xshift=0,yshift=0] {4}
            }
        }
    child{ node[style={draw,circle},xshift=15,yshift=0] {6}
        child{ node[style={draw,circle},xshift=0,yshift=0] {1}
        }
        child{ node[style={draw,circle},xshift=0,yshift=0] {3}
    }
    };
\end{tikzpicture}}
 }

\newsavebox{\mthfot}
\sbox{\mthfot}{
   \scalebox{0.4}{\begin{tikzpicture}[very thick]
\node [style={draw,circle}]{7}
    child{ node[style={draw,circle},xshift=0,yshift=0] {2}
    }
    child{ node[style={draw,circle},yshift=0,xshift=0] {6}
        child{ node[style={draw,circle},xshift=0,yshift=0] {1}
        }
        child{ node[style={draw,circle},xshift=0,yshift=0] {5}
            child{ node[style={draw,circle},xshift=0,yshift=0] {3}
            } 
            child{ node[style={draw,circle},xshift=0,yshift=0] {4}
            }
        }
    };
\end{tikzpicture}}
 }

\newsavebox{\mtftho}
\sbox{\mtftho}{
   \scalebox{0.4}{\begin{tikzpicture}[very thick]
\node [style={draw,circle}]{7}
    child{ node[style={draw,circle},xshift=0,yshift=0] {1}
    }
    child{ node[style={draw,circle},yshift=0,xshift=0] {6}
        child{ node[style={draw,circle},xshift=0,yshift=0] {3}
        }
        child{ node[style={draw,circle},xshift=0,yshift=0] {5}
            child{ node[style={draw,circle},xshift=0,yshift=0] {2}
            } 
            child{ node[style={draw,circle},xshift=0,yshift=0] {4}
            }
        }
    };
\end{tikzpicture}}
 }

\newsavebox{\mthfto}
\sbox{\mthfto}{
   \scalebox{0.4}{\begin{tikzpicture}[very thick]
\node [style={draw,circle}]{7}
    child{ node[style={draw,circle},xshift=0,yshift=0] {1}
    }
    child{ node[style={draw,circle},yshift=0,xshift=0] {6}
        child{ node[style={draw,circle},xshift=0,yshift=0] {2}
        }
        child{ node[style={draw,circle},xshift=0,yshift=0] {5}
            child{ node[style={draw,circle},xshift=0,yshift=0] {3}
            } 
            child{ node[style={draw,circle},xshift=0,yshift=0] {4}
            }
        }
    };
\end{tikzpicture}}
 }

\newsavebox{\mofutth}
\sbox{\mofutth}{
   \scalebox{0.4}{\begin{tikzpicture}[very thick]
\node [style={draw,circle}]{7}
    child{ node[style={draw,circle},xshift=-15,yshift=0] {5}
            child{ node[style={draw,circle},xshift=0,yshift=0] {1}
            } 
            child{ node[style={draw,circle},xshift=0,yshift=0] {4}
            }
        }
    child{ node[style={draw,circle},xshift=15,yshift=0] {6}
        child{ node[style={draw,circle},xshift=0,yshift=0] {2}
        }
        child{ node[style={draw,circle},xshift=0,yshift=0] {3}
    }
    };
\end{tikzpicture}}
 }

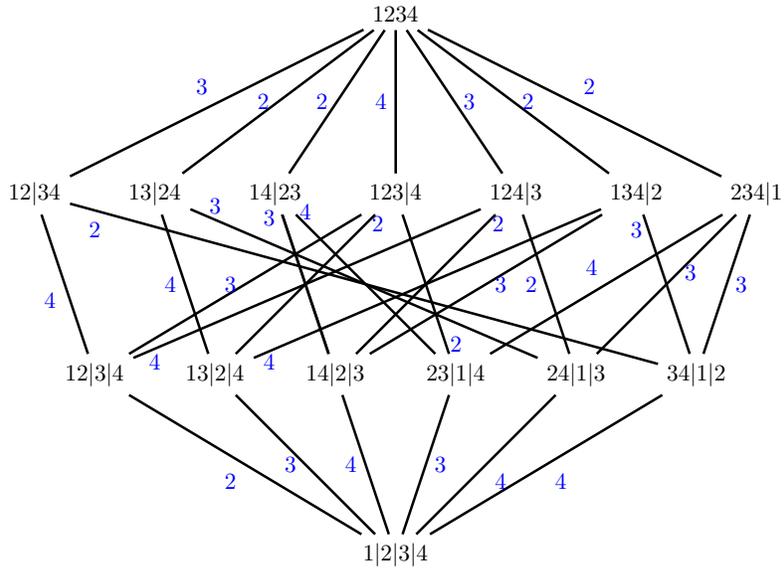
\begin{figure}[H]
    \centering
    \scalebox{0.8}{\begin{tikzpicture}[very thick]
  \node (zh) at (0,0) {\scalebox{1}{$1|2|3|4$}};
  \node (12) at (-5,3) {\scalebox{1}{$12|3|4$}};  
  \node (13) at (-3,3) {\scalebox{1}{$13|2|4$}};
  \node (14) at (-1,3) {\scalebox{1}{$14|2|3$}};
  \node (23) at (1,3) {\scalebox{1}{$23|1|4$}};
  \node (24) at (3,3) {\scalebox{1}{$24|1|3$}};
  \node (34) at (5,3) {\scalebox{1}{$34|1|2$}};
  \node (1234) at (-6,6) {\scalebox{1}{$12|34$}};
  \node (1324) at (-4,6) {\scalebox{1}{$13|24$}};
  \node (1423) at (-2,6) {\scalebox{1}{$14|23$}};
  \node (123) at (0,6) {\scalebox{1}{$123|4$}};
  \node (124) at (2,6) {\scalebox{1}{$124|3$}};
  \node (134) at (4,6) {\scalebox{1}{$134|2$}};
  \node (234) at (6,6) {\scalebox{1}{$234|1$}};
  \node (oneh) at (0,9) {\scalebox{1}{$1234$}};

  \node[blue] (23 1423) at (-1.5,5.7) {$4$};
  \node[blue] (14 1423) at (-2.1,5.6) {$3$};
  \node[blue] (24 1324) at (-3,5.8) {$3$};
  \node[blue] (34 1234) at (-5,5.4) {$2$};
  \node[blue] (12 124) at (-4,3.2) {$4$};
  \node[blue] (13 123) at (-0.3,5.5) {$2$};
  \node[blue] (13 134) at (-2.1,3.2) {$4$};
  \node[blue] (14 124) at (1.7,5.5) {$2$};
  \node[blue] (23 123) at (1,3.5) {$2$};
  \node[blue] (34 134) at (4,5.4) {$3$};
  \node[blue] (24 234) at (4.9,4.7) {$3$};
  
  \draw (zh) --node[below left,blue] {$2$} (12) --node[below left,blue] {$4$} (1234) --node[above left,blue] {$3$} (oneh) --node[above right,blue] {$2$} (234) --node[right,blue] {$3$} (34) --node[below right,blue] {$4$} (zh) --node[left,blue] {$3$} (13) --node[left,blue] {$4$} (1324) --node[left,blue] {$2$} (oneh) --node[right,blue] {$2$} (134) --node[right,blue] {$3$} (14) --node[left,blue] {$4$} (zh);
  \draw (zh) --node[right,blue] {$3$} (23) -- (1423) --node[left,blue] {$2$} (oneh);
  \draw (zh) --node[below right,blue] {$4$} (24) --node[left,blue] {$2$} (124) --node[right,blue] {$3$} (oneh);
  \draw (24) -- (234);
  \draw (13) -- (134);
  \draw (13) -- (123);
  \draw (14) -- (1423);
  \draw (14) -- (1423);
  \draw (12) --node[left,blue] {$3$} (123)--node[left,blue] {$4$} (oneh);
  \draw (12) -- (124);
  \draw (14) -- (124);
  \draw (23) -- (123);
  \draw (23) --node[above left,blue] {$4$} (234);
  \draw (24) -- (1324);
  \draw (34) -- (1234);
  \draw (34) -- (134);
\end{tikzpicture}} 
\caption{The partition lattice $\Pi_4$ with the max-min {\EL}.}
    \label{fig:Pi4 with min max EL}
\end{figure}

\begin{figure}[H]
    \centering
\scalebox{0.8}{\begin{tikzpicture}[very thick]
  \node (m1234) at (0,0) {\usebox{\motthf}}; 
  \node (m12u34) at (1.5,4) {\usebox{\motuthf}};
  \node (m1243) at (-1.5,4) {\usebox{\motfth}};
  \node (m1324) at (-4.5,4) {\usebox{\mothtf}};
  \node (m2314) at (5.5,4) {\usebox{\mtthof}};
  \node (m1342) at (-7.5,8) {\usebox{\mothft}};
  \node (m13u24) at (-4.5,8) {\usebox{\mothutf}};
  \node (m1423) at (-2.5,8) {\usebox{\moftth}};
  \node (m2413) at (0.8,8) {\usebox{\mtfoth}};
  \node (m34u12) at (3,8) {\usebox{\mthfuot}};
  \node (m2341) at (5.5,8) {\usebox{\mtthfo}};
  \node (m23u14) at (7.5,8) {\usebox{\mtthuof}};
  \node (m1432) at (-7.5,12) {\usebox{\moftht}};
  \node (m24u13) at (-4.5,12) {\usebox{\mtfuoth}};
  \node (m3412) at (-1.5,12) {\usebox{\mthfot}};
  \node (m2431) at (1.5,12) {\usebox{\mtftho}};
  \node (m3421) at (4.2,12) {\usebox{\mthfto}};
  \node (m14u23) at (7.5,12) {\usebox{\mofutth}};
  
  \draw (m1234) -- (m1324) -- (m1342) -- (m1432) -- (m1423) -- (m1243) -- (m1234) -- (m12u34) -- (m34u12) -- (m3421) -- (m2341) -- (m2314) -- (m1234);
  \draw (m1324) -- (m13u24) -- (m24u13) -- (m2413) -- (m1243);
  \draw (m1423) -- (m14u23) -- (m23u14) -- (m2314);
  \draw (m1342) -- (m3412) -- (m34u12);
  \draw (m2413) -- (m2431) -- (m2341);
  
\end{tikzpicture} }
    \caption{$\cord{{\Pi_4}}{\lm}$ induced by the max-min {\EL} $\lm$ of $\Pi_4$.}
    \label{fig:Pi4 min max cord}
\end{figure}
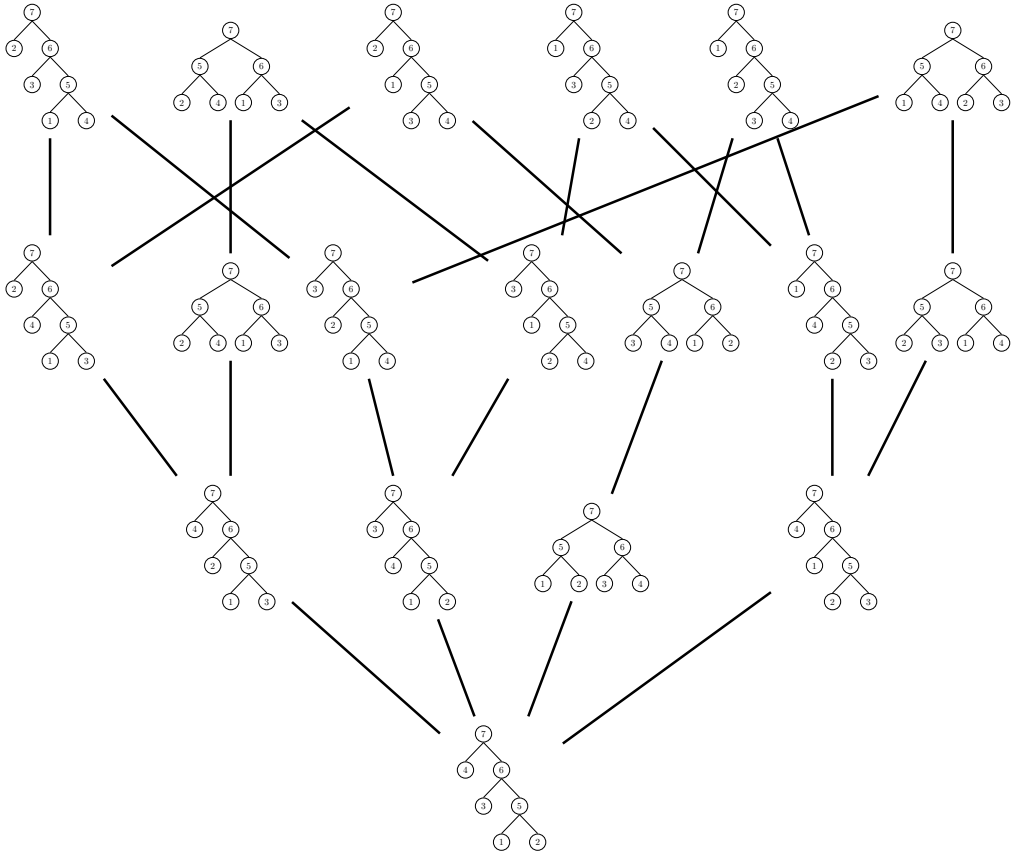
\end{example}

\begin{remark}\label{rmk:noncrosssing partition lattice}
We can restrict $PT(n)$ to those trees with ``non-crossing" leaf sets and obtain a subposet of $(PT(n),\preceq)$ which is isomorphic to the maximal chain descent order of the max-min {\EL} $\lm$ restricted to the non-crossing partition lattice $NC_{n+1}$. Alternatively, we could also similarly construct a different poset isomorphic to $\cord{{NC_{n+1}}}{\lm}$ from the rooted $k$-ary trees which Edelman and Simion used to study chains in $NC_{n+1}$ in \cite{chainsncnedelmansimion1994} and \cite{multchainsncpartstreesedelman1982}. A similar simple operation on the trees describes the cover relations.
\end{remark}
\end{subsection}

\section*{Acknowledgements}
The author is grateful to his Ph.D. advisor Patricia Hersh, discussions with whom resulted in the definition of a maximal chain descent order, and whose mentorship and guidance were invaluable to this work. The author also thanks Alex Chandler and Ben Hollering whose long ago conversations about diamonds in thin posets were always in the back of the author's mind.

\bibliography{maximal_chain_descent_orders_bibliography}

\vspace{10mm}

Stephen Lacina

Department of Mathematics, University of Oregon, Eugene, OR 97403

\textit{Email address:} slacina@uoregon.edu
\end{document}